\documentclass[11pt]{amsart}

\usepackage{amsmath,amscd,amssymb}
\usepackage[mathscr]{eucal}

\usepackage[hyperindex,breaklinks]{hyperref}

\newtheorem{theorem}{Theorem}[section]
\newtheorem{lemma}[theorem]{Lemma}
\newtheorem{proposition}[theorem]{Proposition}
\newtheorem{corollary}[theorem]{Corollary}

\theoremstyle{definition}

\newtheorem{example}[theorem]{Example}

\theoremstyle{remark}
\newtheorem{remark}[theorem]{Remark}

\numberwithin{equation}{section}

\def\C{{\mathbb C}}
\def\R{{\mathbb R}}
\def\N{{\mathbb N}}

\def\supp{{\rm supp}\,}

\begin{document}

\title[On the theory of generalized Stieltjes transforms]
{On the theory of generalized Stieltjes transforms}

\author{Alexander Gomilko}
\address{Faculty of Mathematics and Computer Science\\
Nicolas Copernicus University\\
ul. Chopina 12/18\\
87-100 Toru\'n, Poland
}

\address{
Institute of Telecommunications and Global Information Space, \\
National Academy of Sciences of Ukraine, \\
Chokolivsky Blvd. 3, 03186 Kyiv, Ukraine}

\email{gomilko@mat.umk.pl}

\author{Yuri Tomilov}
\address{Institute of Mathematics\\
Polish Academy of Sciences\\
\'Sniadeckich Str. 8\\
00-656 Warsaw, Poland
}

\address{
Faculty of Mathematics and Computer Science\\
Nicolas Copernicus University\\
ul. Chopina 12/18\\
87-100 Toru\'n, Poland
}

\email{ytomilov@impan.pl}

\thanks{The authors were  partially
supported by the NCN grant 2017/27/B/ST1/00078.}



\date{\today}

\begin{abstract}
We identify
measures arising
in the representations of products of generalized Stieltjes
transforms as generalized Stieltjes transforms, provide optimal estimates for the size of those measures, and address a similar issue
for generalized Cauchy transforms. In the latter case, in two particular settings, we give criteria
ensuring that the measures are positive.
On this way, we also obtain new, applicable conditions for representability
of functions as generalized Stieltjes transforms,
thus providing a partial answer to a problem posed by Sokal
and shedding a light at spectral multipliers emerged recently in probabilistic studies.
As a byproduct of our approach, we improve  several known results on Stieltjes
and Hilbert transforms.

\end{abstract}

\maketitle

\tableofcontents
\section{Introduction}

This paper is concerned with the theory of generalized Stieltjes functions
and corresponding Stieltjes transforms,
and it offers several insights on their structure and algebraic properties.
To set the scene, for $\alpha \ge 0$ let
$\mathcal M^+_\alpha (\R_+)$ be the set of positive Borel measures on
$\R_+=[0,\infty)$ such that $(1+s)^{-\alpha} \in L^1(\R_+, \mu).$
The generalized Stieltjes transform of order $\alpha>0$ is defined on $\mathcal M^+_\alpha (\R_+)$
by
\begin{equation}\label{SSS0}
S_\alpha[\mu](z)=\int_{0}^\infty \frac{\mu(dt)}{(z+t)^\alpha}, \qquad \mu \in \mathcal M^+_\alpha(\R_+), \qquad z>0.
\end{equation}
The set of $S_\alpha[\mu],$ where $\mu$ varies through $\mathcal M^+_\alpha(\R_+),$
coincides modulo constants with the
so-called class of generalized Stieltjes functions, arising in various areas of analysis.
Recall that a function
$f:(0,\infty)\mapsto [0,\infty)$ is said to be {\it generalized Stieltjes} of order $\alpha>0$ if
\begin{equation}\label{SSS00}
f(z)=a+S_\alpha[\mu](z), \qquad z>0,
\end{equation}
where  $a \ge 0$ and $\mu \in \mathcal M^+_\alpha(\R_+).$
In this case we will write $f \sim(a,\mu)_{\alpha},$ call $(a,\mu)_\alpha$ the Stieltjes representation of $f,$ and denote the class of such $f$
by $\mathcal S_\alpha.$

The studies on generalized Stieltjes functions  go back to 1930's and to well-known papers
by Hirschman and Widder. Apart from classical analysis, these functions  arise in many areas of mathematics
including probability theory \cite{Bond}, \cite{Cia},  \cite{Sh}, \cite{Asch}, mathematical physics \cite{Demni1}, \cite{Demni}, \cite{Faraut}, \cite{Tica},  \cite{Vershik}, complex analysis \cite{MCbook}, \cite{Jerb},   geometry
\cite{Mene},
and operator theory \cite{BGTMZ}, \cite{GHT}, \cite{Heymann}, \cite{HirschInt}, \cite{HovelW}, \cite{Mart}, \cite{Mart1}, \cite{Schwartz},
even though they were outside mainstream research
 for a long while.
(Here the division into separate topics is rather conditional, and the references are just samples.)
Recently, the interest in generalized Stieltjes functions renewed,
and several pertinent papers appeared on this subject, see e.g.  \cite{Ask}, \cite{Banuelos}, \cite{Berg}, \cite{KPAs},
\cite{KPGen}, and \cite{KP23}.
Generalized Stieltjes functions also arise in a somewhat more general form of generalized
Cauchy transforms, mostly in the setting of the unit circle, see \cite{MCbook}, and numerous references given there.
For some applications of generalized Cauchy transforms in the framework of the real line one may consult \cite{Jerb}.
Though apparently these two flows of research had very few points in common so far.

Most of studies on generalized Stieltjes functions centered around their subclass  $\mathcal S_1,$
given by
\begin{equation}\label{stilt_def}
\mathcal S_1:=\{a+ S_1[\mu]: \mu \in \mathcal M_1^+(\R_+),
\,\, a \ge 0\},
\end{equation}
and usually called Stieltjes functions in the literature.
This class along with the closely related class $\mathcal{CBF}$ of so-called complete Bernstein functions
is fundamental in probability theory and spectral theory and their interactions.
Stieltjes and complete Bernstein functions are thoroughly discussed, for instance, in \cite[Chapters 2-8]{Sh}.
For their very recent and interesting generalizations see e.g. \cite{Bonde}, \cite{Kwas}.

The following geometric characterisation of Stieltjes functions is  one of the most useful results of the theory.
It is behind a number of important statements concerning Stieltjes functions and their properties.
\begin{theorem}\label{geometric}
 A function $f : (0,\infty) \to  [0,\infty)$ is Stieltjes
if and only if it is the
restriction to $(0,\infty)$ of an analytic function $F$ on  $\C\setminus (-\infty, 0]$ satisfying
 ${\rm Im}\, F (z)\le  0$ for $z\in \C$ with  ${\rm Im}\, z > 0,$ or equivalently,
 ${\rm Im} \, z \, {\rm Im}\, F (z)\le  0$ for $z \in \C \setminus (-\infty, 0].$
\end{theorem}
See e.g. \cite{Sh} or \cite{Berg2} for a thorough analysis of this result and its relevance
for the theory of Stieltjes functions.

However, in contrast to the case $\alpha=1$, the classes $\mathcal S_\alpha$ of the generalized Stieltjes functions of arbitrary order $\alpha>0,\alpha \neq 1,$ seem to be far from being fully understood,
and essential difficulties  arise in putting a structure into these classes, at least when following intuition
developed for the study of $\mathcal S_1.$
A number of important properties of $\mathcal S_1$ fail for $\mathcal S_\alpha$ with $\alpha \neq 1,$
even after natural reformulations.
In particular, this concerns Theorem \ref{geometric}, where plausible variants
for $\alpha \neq 1$ appear to be false, and the problem of finding an analogue
of Theorem \ref{geometric} for arbitrary $\alpha >0$ was posed  by Sokal in  \cite[p. 184-185]{Sokal}.
So far the research on $\mathcal S_\alpha$  existed  in the form of isolated results scattered over the literature.
Several important properties of $\mathcal S_\alpha$
 can be found in the survey paper  \cite{Karp}.
It is instructive to note that, by e.g. \cite[Theorem 3]{Karp},
one has
$\mathcal{S}_\alpha\,\subset\,\mathcal{S}_\beta, 0<\alpha<\beta,$
so that a function $f \in \mathcal S_\alpha$ admits different Stieltjes representations in $\mathcal S_\beta$ for $\beta >\alpha$
with sometimes rather different properties.
For instance if $f(z)=z^{-1}, z >0,$ then $f=S_1[\delta_0]$ and $f=S_2[\nu],$
where $\nu$ stands for the Lebesgue measure on $\R_+.$

The lack of geometric or any other easy applicable characterisations of $S_\alpha$
makes it, in particular,  nontrivial to decide whether the product  $f_1 f_2$ of  $f_1 \in S_{\alpha_1}$ and $f_2 \in S_{\alpha_2}$
belongs to $S_\gamma$ for some $\gamma >0$ and to determine the optimal $\gamma$ whenever such a $\gamma$ exists.
Nevertheless, Hirschman and Widder proved in \cite[Chapter VII.7.4]{widder}
(see also \cite{HirWi51}) that if $\mu_1 \in \mathcal M^+_{\alpha_1}(\R_+)$
and $\mu_2 \in \mathcal M^+_{\alpha_2}(\R_+),\alpha_1,\alpha_2 >0,$ then there exists $\mu=\mu[\mu_1,\mu_2] \in \mathcal M^+_{\alpha_1+\alpha_2}(\R_+)$
such that
\begin{equation}\label{HWp}
S_{\alpha_1}[\mu_1](z)S_{\alpha_2}[\mu_2](z)=S_{\alpha_1+\alpha_2}[\mu](z), \qquad z >0,
\end{equation}
i.e.
$\mathcal S_{\alpha_1} \mathcal S_{\alpha_2}\subset \mathcal S_{\alpha_1+\alpha_2}$
for all $\alpha_1,\alpha_2>0.$
Observe that $\mu$ in \eqref{HWp} is determined uniquely.
 Moreover, simple examples
 show (see e.g. \eqref{A2} below) that the order $\alpha_1+\alpha_2$ in \eqref{HWp} cannot in general be improved.

The arguments of Hirschman and Widder, addressing in fact a more general situation, were somewhat involved. They relied on
rescaling $S_\alpha[\mu]$ as $F(t)=e^{\alpha t/2}S_\alpha[\mu](e^{t/2}), t \in \R,$  representing $F$ in the form of a convolution transform
\[
F(t)=\int_{\R} G(t-s)\, \nu(ds), \qquad t \in \R,
\]
with $\nu(ds)=e^{-\alpha s/2}\mu(e^{s}ds)$ and $G(t)=2^{-\alpha} \cosh^\alpha(t/2),$
and making use of the representation
\[
G(t)=\frac{1}{2\pi i}
\int_{-i\infty}^{i\infty} e^{t s} \,
\frac{\Gamma(-s+\alpha/2)\Gamma(s+\alpha/2)}{\Gamma(\alpha)}\,ds,
\qquad t\in \R,
\]
where $\Gamma$ is the Gamma function.

Despite the fact that \eqref{HWp} is one of the basic results in the theory of Stieltjes transforms, quite surprisingly,
the measure $\mu$ in \eqref{HWp} and its fine properties remained implicit until now,
and, apart from very specific cases, the form of $\mu$ seemed to be unknown.
To remove this gap, in the general setting of complex Radon measures on $\R_+,$
we prove \eqref{HWp} by providing an explicit formula for $\mu.$ The argument is elementary.
Thus derived $\mu$ is called the \emph{Stieltjes convolution} of $\mu_1$ and $\mu_2$ to emphasize the  similarity of the map $(\mu_1, \mu_2) \to \mu[\mu_1,\mu_2]$ to the usual convolution of measures.
Our formula reveals several interesting properties of $\mu$ showing in particular that $\mu$ does not have  singular continuous component.
Moreover, representing $\mu$ via  $\mu_1$ and $\mu_2$ we obtain a number of submultiplicative inequalities
bounding $\|\mu\|_{\beta_1+\beta_2}$ in terms of the products  $\|\mu_1\|_{\beta_1}  \|\mu\|_{\beta_2}$ for $\beta_j \in [0,\alpha_j], j=1,2.$
These estimates for the size of $\|\mu\|_{\beta_1+\beta_2}$ seem to be completely new.

Next, identifying absolutely continuous measures with their densities, we test our product formula in the classical framework of Stieltjes transforms of functions from
$L^1(\R_+; (1+t)^{-1}),$
and provide  natural conditions on $g_1, g_2 \in L^1(\R_+; (1+t)^{-1})$
to ensure that
the product $S_1[g_1]S_1[g_2]$
is again of the form $S_1[g]$ for some $g \in L^1(\R_+; (1+t)^{-1}).$ This generalizes and extends similar results
in the literature.
Curiously, in the case of Hilbert transforms of $L^p$-functions,
it appeared that our formula is essentially equivalent to the well-known Tricomi identity for Hilbert transforms in its most general formulation.

Although an easy applicable, e.g. geometrical,  characterisation of  $f \in \mathcal S_\beta$ is still out reach, we provide a sufficient condition,
which looks simple and easy to use. Namely, we show that if  $f \in \mathcal S_\alpha,$ and if for a given $\beta \in (0,1]$ one defines $f_{(\beta)}(z)=f(z^\beta), z>0,$
then $f_{(\beta)} \in \mathcal S_{\alpha\beta}.$
So,  in particular,  $ f_{(\beta)} \in \mathcal S_\beta$ if $f \in \mathcal S_1.$ This implication can be considered as a partial answer to Sokal's question mentioned above.
Moreover, employing a different argument, we prove that  if $f \sim (0,\nu)_1,$ then $f_{(\beta)} \sim (0,\mu)_\beta ,$ where $\mu$ is an absolutely continuous
measure given explicitly in terms of $\nu.$ The case $\beta=1/2$ is especially explicit and is a nice illustration of our considerations.
These results
stem from a general theorem
ensuring that for an appropriate $f: (0,\infty) \to (0,\infty)$ and $\gamma=\gamma(f)>0$ one has
$f^\alpha \in \mathcal S_{\gamma \alpha}$ for all $\alpha >0.$
Our version of the product formula  \eqref{HWp}
shows that measures in the Stieltjes representation for $f^\alpha$
do not have singular component and helps to refute several plausible conjectures appearing along the way.
Moreover, we show that the theorem is close to be optimal.

Note that the set $S_{1/2}(\mathcal M^+_0(\R_+))$
was crucial  in the recent study  \cite{Banuelos} of spectral multipliers $\Phi$ on $L^p$-spaces.
Without going into details of \cite{Banuelos}, which rely on a probabilistic background, note that the main statements in \cite{Banuelos} required  $\Phi$  to  have a specific form $S_{1/2}[\mu]$ with $\mu$ being a (bounded) complex measure on $\R_+.$
Our results allow one to produce such multipliers
from a well-understood set of $S_1[\mu], \mu \in \mathcal M^+_0(\R_+),$ and moreover they help to identify
$\Phi,$ dealing with functions themselves rather than with their  Stieltjes representations.
Recall that by Theorem \ref{geometric}
it is comparatively easy to check whether a given function $f$ belong to $\mathcal S_1.$ Then
$f_{(1/2)} \in \mathcal S_{1/2}$ yields a multiplier $\Phi$ needed in \cite[Theorems 1.1-1.3]{Banuelos}.
In addition, we are able to identify the Stieltjes representation of $f_{(1/2)},$
and thus to write $f_{(1/2)}$ in the form used in \cite{Banuelos}.

Finally we study counterparts of \eqref{HWp} in the framework
of generalized Cauchy transforms of complex Radon measures on $\R.$
To put our results into a proper context, we first recall a similar research done for generalized (or, in other terminology, fractional) Cauchy transforms of measures on the unit circle $\mathbb T.$
Let $\mathcal {T}_\alpha$, $\alpha>0,$
be the space of the generalized Cauchy transforms ${\sf C}_\alpha[\mu]$ of order $\alpha$
given by
\begin{equation}\label{MacGG}
{\sf C}_\alpha[\mu](z):=\int_{\mathbb T}\frac{\mu(d\zeta)}{(1-\zeta z)^\alpha},\qquad |z| < 1,
\end{equation}
where $\mu$ runs over the set of probability measures $\mathcal P^+(\mathbb T)$ on $\mathbb T.$
By  one of the fundamental results in the theory of generalized  Cauchy transforms, see e.g. \cite[Theorem 1]{MacGr73} and \cite[p. 22-24]{MCbook},
 if $\mu_j \in \mathcal P^+(\mathbb T)$  and $\alpha_j >0, j=1,2,$
then there exists $\mu \in \mathcal P^+(\mathbb T)$ such that
\begin{equation}\label{prod_cauchy_in}
{\sf C}_{\alpha_1}[\mu_1](z){\sf C}_{\alpha_2}[\mu_2](z)={\sf C}_{\alpha_1+\alpha_2}[\mu](z), \qquad |z|<1.
\end{equation}
In other words, we have
$
\mathcal{T}_{\alpha_1} \mathcal{T}_{\alpha_2}\subset
\mathcal{T}_{\alpha_1+\alpha_2},
$
which is an analogue of the corresponding inclusion for the generalized Stieltjes transforms discussed above.
See \cite{MCbook} for many instances where this result played a role.

In view of \eqref{HWp} and \eqref{prod_cauchy_in}, it is natural to ask whether there exists a relation similar \eqref{prod_cauchy_in}
for generalized Cauchy transforms of measures on $\R.$ Let $\mathcal M^+_\alpha(\R), \alpha \ge 0,$ stand for the set of positive Borel measures $\mu$ on $\R$
satisfying  $(1+|s|)^{-\alpha}\in L^1(\R, \mu).$
For $\mu \in \mathcal M^+_\alpha(\R), \alpha >0,$  define its generalized Cauchy transform $C_\alpha[\mu]$ of order $\alpha$ by
\[
C_\alpha [\mu](z)
:=\int_{\R}\frac{\mu(dt)}{(z+t)^\alpha},
\]
for  $z$ from the upper half-plane $\mathbb C^+.$
Using our technique developed for the generalized Steltjes transforms, we
write down explicitly a ``product'' measure $\mu$ satisfying
\begin{equation}\label{ccauchy}
C_{\alpha_1}[\mu_1](z)C_{\alpha_2}[\mu_2](z)=C_{\alpha_1+\alpha_2}[\mu](z), \qquad z \in \C^+,
\end{equation}
 for  $\mu_1 \in \mathcal M^+_{\alpha_1}(\R)$ and  $\mu_2 \in \mathcal M^+_{\alpha_2}(\R),\alpha_1,\alpha_2>0,$
and even for $\mu_1$ and $\mu_2$ from more general classes of complex Radon measures on $\R.$
 In this case, $\mu$ is not unique,
and we provide a
 $\mu$ preserving good properties of $\mu_1$ and $\mu_2$ such as positivity, integrability or compact support,
and given by a comparatively simple formula.
As a byproduct, as in the case of Stieltjes transforms, the resulting $\mu$ does not have  singular continuous component.
 The task of expressing $\mu$ in terms of $\mu_1$ and $\mu_2$ appears to be rather involved,
so we impose additional assumptions on the size of $\mu_1$ and $\mu_2,$
which hold in several natural situations.
If the size restrictions are dropped, then we also give a formula for $\mu,$
but, in such generality, the formula  becomes rather implicit.
Moreover, it does not necessarily yield a positive $\mu.$
Thus, given $\mu_1\in \mathcal M^+_{\alpha_1}(\R)$ and $\mu_2 \in \mathcal M^+_{\alpha_2}(\R),$
we address the important problem of existence of
$\mu \in \mathcal M^+_{\alpha_1+\alpha_2}(\R)$  such that \eqref{ccauchy} holds.
We prove the existence criteria 
when $\alpha_1=\alpha_2=1$ or $\alpha_1+\alpha_2=1,\alpha_1,\alpha_2>0,$
though the general case remains open.
We also study a related problem of liftings of the generalized Cauchy transforms
from $C_\alpha(\mathcal M^+_\alpha(\mathbb R))$ to $C_\beta(\mathcal M^+_\beta(\mathbb R)), 0<\alpha <\beta,$
and describe several situations when such a lifting is possible.
Despite the generalized Cauchy transforms on $\mathcal P^+ (\mathbb T)$ were studied intensively for some time,
the product formulas and related matters for the generalized Cauchy transforms on $\mathcal M^+_\alpha(\R)$
(and more general classes of measures) appear
probably for the first time.

\section{Notations, preliminaries and conventions}

In this section we collect various notations used throughout the paper,
and clarify parts of our terminology.

To be able to handle "unbounded" complex measures properly, we use the notion of complex Radon measures.
Recall that a bounded linear functional $\mu$ on the Frech\'et space $C_c(\R)$ of continuous functions on $\R$ with compact support is said to be
a complex Radon measure on $\R.$ In other words, a complex Radon measure $\mu$ on $\R$ is a linear functional
 on $C_c(\R)$ such that for every compact $K \subset \R$
the inequality
$|\mu(f)|\le C_K \max_K |f|$ holds for some $C_K>0$  and for all $f \in C_c(\R)$ with support in $K.$
Usually, one uses a more intuitive notation for $\mu$ writing it as an integral: $\mu(f)=\int f d\mu.$
If the context is clear, then the adjective "complex" is usually omitted, and we will follow this rule in the sequel.
Recall that a Radon measure $\mu$ is said to be real if $\mu(f) \in \R$ for all real-valued $f \in C_c(\R),$
and $\mu$ is positive if $\mu(f)\ge 0$ for all $f \in C_c(\R),$ $f \ge 0.$
Note that the space of real Radon measures on $\R$ is a Riesz space with the natural order:
$\mu_1\ge \mu_2$ if $\mu_1-\mu_2$ is a positive Radon measure.
If $\mu$ is a  Radon measure, then there exists the smallest positive Radon measure $|\mu|$ called the variation
(or the absolute  value) of
$\mu$ satisfying $|\mu(f)| \le |\mu|(f)$ for all $f \in C_c(\R),$ $f \ge 0.$
There is an analogue of Jordan's decomposition for
$\mu:$
$$\mu= {\rm Re}\, \mu + i {\rm Im}\, \mu=\mu_1-\mu_2+ i(\mu_3-\mu_4),$$ where $\mu_i, 1 \le i \le 4,$ are positive Radon measures,
such that $|\mu_1-\mu_2|=\mu_1 +\mu_2$ and $|\mu_3-\mu_4|=\mu_3+\mu_4$
(i.e. $\mu_1$ and $\mu_2,$ as well as $\mu_3$ and $\mu_4,$ are mutually singular).
The Radon measures ${\rm Re}\, \mu$ and ${\rm Im}\, \mu$ satisfy $|{\rm Re}\, \mu|\le |\mu|,$ $|{\rm Im}\, \mu|\le |\mu|,$ and $|\mu|\le |{\rm Re}\, \mu|+|{\rm Im}\, \mu|.$
The space $L^1(\R, |\mu|)$ arises as a completion of $C_c(\R)$ in an appropriate semi-norm constructed by means of $\mu,$ and $\mu$ extends to this space accordingly.
For such an extension of  $\mu,$ denoted by the same symbol, and a locally integrable function $g$ (with respect to $\mu$)
 the mapping  $C_c(\R)\ni f \to (g  \mu) (f)=\mu(gf)$  is a Radon measure again, and moreover
$|g  \mu |(dt)=|g(t)| |\mu|(dt).$

Any Radon measure  $\mu$ on $\R$ admits the Lebesgue decomposition $\mu=\mu_a+\mu_d+\mu_s,$
where  $\mu_a,$ $\mu_d,$ and $\mu_s$ stand for locally absolutely continuous
(with respect the Lebesgue measure), purely discrete and
singular continuous part of $\mu,$ respectively. Moreover, the Radon measures $\mu_a,$ $\mu_d$ and $\mu_s$ are defined uniquely, mutually singular, and $|\mu|=|\mu_a|+|\mu_d|+|\mu_s|.$

Observe that positive Radon measures on $\R$ can be identified with locally finite Borel measures on $\R$ by the Riesz representation theorem.
It is also instructive to recall that the space of Radon measures on $\R$ can be considered as the space of functions on $\R$
with locally bounded variation, thus allowing one to use the Lebesgue-Stieltjes integration theory.

In general, using the measure-theoretic terminology, one may consider
 a  Radon measure $\mu$ on $\R$
as a complex Borel measure $\mu$ defined "locally" on $\R$. In other words, if $\mathcal B(\R)$ is the Borel $\sigma$-algebra of $\R,$ then $\mu$ is defined on $\mathcal B(\R)\cap K$ for all compact $K \subset \R,$ takes complex values, and is such that
for any compact $K\subset \R$ the function $B \to \mu(B \cap K), B \in \mathcal B(\R),$ is a complex measure. The need in such objects is clearly justified by a standard example of the Radon measure
${\rm sign}(t) dt,$ which cannot be considered in the framework of complex (or signed) Borel measures on $\mathcal B(\R).$

Finally, note that all of the above considerations are valid with $\R$ replaced by $\R_+,$ with obvious adjustments.

For the theory of Radon measures and their properties one may consult e.g. \cite{Bourb}, \cite[Chapter 7]{Cohn},
\cite[Chapter 4]{Edw}, and \cite[Chapter 4]{Schw}, which sometimes complement each other.
A compact treatment (though without using the Radon measures terminology) in \cite[Chapter 4.16]{HP} is also relevant. Being unable to go into finer details of the theory of Radon measures, we just remark that with certain care, all of the basic facts from the integration theory of positive (Borel) measures can be adapted to this quite a general setting.

For $\alpha \ge 0,$ let ${\mathcal M}_\alpha(\mathbb R_+)$
denote the set of complex Radon measures $\mu$ on $\R_+$
such that
\begin{equation}\label{SSS11}
\|\mu\|_\alpha:=\int_{\R_+}\frac{|\mu|(dt)}{(1+t)^\alpha}<\infty.
\end{equation}
Note that $(\mathcal M(\R_+), \|\cdot\|_\alpha)$ is a linear normed space.
 If  $\mu \in \mathcal M_\alpha(\R_+),$
then $\mu= \mu_{1}-\mu_{2} + i(\mu_{3}-\mu_{4}),$
with $\{ \mu_{i} : 1 \le i \le 4\} \subset \mathcal M^+_{\alpha}(\R_+).$
The normed space $(\mathcal M_\alpha (\R_+),\|\cdot\|_\alpha)$
 can be identified with the dual to
$C_0(\R_+; (1+t)^\alpha):=\{f \in C(\R_+): (1+t)^\alpha f(t)\to 0, \, t \to \infty\}$
equipped with a natural sup-norm,
so  $(\mathcal M_\alpha (\R_+),\|\cdot\|_\alpha)$ is
a Banach space.
The positive part of ${\mathcal M}_\alpha(\mathbb R_+)$
will be denoted by
${\mathcal M}^+_\alpha(\mathbb R_+).$

Similarly, the set of complex Radon measures $\mu$
on $\R$ satisfying
\begin{equation}\label{SSS12}
\|\mu\|_\alpha:=\int_{\R}\frac{|\mu|(dt)}{(1+|t|)^\alpha}<\infty
 \end{equation}
for some $\alpha \ge 0$
will be denoted by $\mathcal M_\alpha (\R),$ and $\mathcal M^+_\alpha(\R)$ will stand for positive Radon measures
in $\mathcal M_\alpha (\R).$
Note that $(\mathcal M_\alpha (\R), \|\cdot\|_\alpha)$ is a Banach space dual to $C_0(\R; (1+|t|)^\alpha).$ Moreover,
${\mathcal M}_\alpha (\mathbb R_+) \subset {\mathcal M}_\alpha(\mathbb R)$
and
${\mathcal M}^+_\alpha (\mathbb R_+) \subset {\mathcal M}^+_\alpha(\mathbb R)$
isometrically, so the norms in \eqref{SSS11} and \eqref{SSS12} agree in a natural sense.

For $\alpha \ge 0$ and $p\ge 1$ let $L^p_{\alpha}(\R_+)$  stand for $L^p(\R_+; (1+t)^{-\alpha}).$ We identify  $f \in L^1_\alpha(\R_+)$
with  $\mu=\mu_f \in \mathcal M_\alpha(\R_+)$ so that  $\mu(dt)=f(t) dt,$
and such an identification yields an isometric embedding of $L^1_\alpha(\R_+)$ into
 $\mathcal M_\alpha(\R_+).$  Similarly,   denoting $L^p(\R; (1+|t|)^{-\alpha}), p \ge 1,$
by $L^p_{\alpha}(\R),$ observe that $L^1_{\alpha}(\R)$ embeds isometrically into $\mathcal M_\alpha(\R).$
For the sequel,  extending $g \in L^p_\alpha(\R_+)$ by zero to $(-\infty,0),$ it will be convenient to identify $g$ with its extension to $\R,$ thus embedding $L^p_\alpha(\R_+)$ into $L^p_\alpha(\R)$ isometrically.

If $\mu \in \mathcal M^+_\alpha(\R_+),$ then its generalized Stieltjes transform of order $\alpha>0$ will be denoted by $S_\alpha[\mu]$, and for $\mu \in \mathcal M_\alpha(\mathbb R)$ we denote its generalized Cauchy transform of order $\alpha>0$ by $C_\alpha[\mu].$
If, moreover, $\mu(dt)=f(t)\,dt$,
and $F$ is one the two transformations above, then we will write
$F[f]$ instead $F[\mu]$.

Furthermore,  we will let $\mathcal S_\alpha$ stand for $\{a+ S_\alpha[\mu]: a\ge 0, \mu \in \mathcal M^+_\alpha(\mathbb R_+)\},$
and if $f \in \mathcal S_\alpha,$   $f=a+ S_\alpha[\mu],$ then we will write $f \sim (a,\mu)_\alpha,$
and call $(a,\mu)_\alpha$ the Stieltjes representation of $f.$
It is essential to note that
$f$ admits an analytic extension to
$\C\setminus (-\infty,0],$
and the Stieltjes representation of $f$ is determined uniquely.
The uniqueness is a direct consequence of a complex inversion theorem for the generalized Stieltjes transforms,
see \cite{Sumner} or \cite{Byrne}, and the property $S_{\alpha}[\mu](z) \to 0$ as $z \to \infty.$
Alternatively, one may use a Laplace representation formula for $f$ (\cite[Theorem 8]{Karp}) and a uniqueness theorem for Laplace
transforms. Though  \cite{Byrne}, \cite{Karp}, and \cite{Sumner} treat only positive measures,
the corresponding reformulations for complex measures are straightforward.

 As usual, $\delta_t, t \in \mathbb R,$  denotes the Dirac delta measure at $t \in \mathbb R,$
 $\chi_{E}$ stands for the characteristic function of a set $E \subset \R,$
 and ${\rm supp} \, \mu$ denotes the support of a measure $\mu.$ 
In the sequel, for  $-\infty \le a<b \le \infty,$ the notation
$
\int_a^b$  is used for $\int_{[a,b]}$ if $a, b \in \mathbb R$,  and for $\int_{(-\infty,b]}$ or
$\int_{[a,\infty)}$ if $a=-\infty$ or $b=\infty,$ respectively.

Throughout the paper, we will use the principal branch of the power function $z \to z ^\alpha, \alpha \in \R,$
given by
\[
z^\alpha=|z|^\alpha e^{i\alpha\arg z},\qquad z\in\C\setminus (-\infty,0],\qquad
\arg(z)\in (-\pi,\pi).
\]

Finally, we let  $\R_+:=[0,\infty)$, $\C^{+}:=\{z\in \C:\,{\rm Im}\,z>0\},$ $\C_+:=\{z \in \C:\, {\rm Re}\, z >0 \},$  $\mathbb T:=\{z\in \C: |z|=1\},$ and $\mathbb D:=\{z \in \mathbb C: |z|<1\}.$

\section{Measures in Hirschman-Widder's representation}\label{st_conv}

In this section, we study the structure and properties of measures $\mu\in \mathcal M_{\alpha_1+\alpha_2}(\R_+)$
arising in Hirschman-Widder's product formula \eqref{HWp}.

The next simple but quite useful result
will be essential for the sequel.
For its proof,
recall that  for all $\alpha_1,\alpha_2>0,$
$-\infty <s<t<\infty,$ and $z\in \C\setminus (-\infty,-s]$  one has
\begin{equation}\label{A2}
\frac{B(\alpha_1,\alpha_2)}{(z+s)^{\alpha_1} (z+t)^{\alpha_2}}=
\frac{1}{(t-s)^{\alpha_1+\alpha_2-1}}
\int_s^t \frac{(\tau-s)^{\alpha_2-1}(t-\tau)^{\alpha_1-1}\,d\tau }{(z+\tau)^{\alpha_1+\alpha_2}},
\end{equation}
where  $B$ is the beta function, defined by
\[
B(\alpha_1,\alpha_2):=\int_{0}^{1} t^{\alpha_1-1}(1-t)^{\alpha_2-1}\, dt, \qquad \alpha_1>0, \, \alpha_2 >0.
\]
See e.g. \cite[p. 301, 2.2.6(2)]{Prudnikov1}.
\begin{theorem}\label{elprod}
Let $\mu_j\in  {\mathcal M}_{\alpha_j}(\R_+)$, $\alpha_j>0,$ $j=1,2.$
Then
\begin{equation}\label{prod_s}
S_{\alpha_1}[\mu_2](z) S_{\alpha_2}[\mu_2](z)=S_{\alpha_1+\alpha_2}[\mu](z),
\qquad z \in \C \setminus (-\infty, 0],
\end{equation}
with the unique $\mu \in  {\mathcal M}_{\alpha_1+\alpha_2}(\R_+)$
defined  by
\begin{equation}\label{m1}
\mu(d\tau)=u(\tau)d\tau+\mu_1(\{\tau\})\mu_2(d\tau),
\end{equation}
and
\begin{align}\label{m0}
B(\alpha_1, \alpha_2) u(\tau)=&
\int_{(\tau,\infty)}\int_{[0,\tau)}
 \frac{(\tau-s)^{\alpha_2-1}(t-\tau)^{\alpha_1-1}}{(t-s)^{\alpha_1+\alpha_2-1}}
\mu_1(ds)\,\mu_2(dt)\\
+&
\int_{(\tau,\infty)} \int_{[0,\tau)}
\frac{(\tau-s)^{\alpha_1-1}(t-\tau)^{\alpha_2-1}}{(t-s)^{\alpha_1+\alpha_2-1}}
\mu_2(ds)\,\mu_1(dt)
\notag
\end{align}
for all almost all $\tau \ge 0.$ Moreover,
 \begin{equation}\label{ineq}
\|\mu \|_{\alpha_1+\alpha_2} \le \|\mu_1\|_{\alpha_1}\|\mu_2\|_{\alpha_2},
\end{equation}
with equality sign if $\mu_j \in \mathcal M^+_{\alpha_j}(\R_+), j=1,2.$
\end{theorem}

\begin{proof}
First assume, in addition, that $\mu_j \in \mathcal M^+_{\alpha_j}(\R_+), j=1,2,$ and let $z>0$ be fixed.
Then using the integrability assumptions on $\mu_1$ and $\mu_2$, write
\[
S[\mu_1](z) S[\mu_2](z)=\int_0^\infty\int_0^\infty
\frac{\mu_1(ds)\mu_2(dt)}{(z+s)^{\alpha_1}
(z+t)^{\alpha_2}}=g_1(z)+g_2(z)+g_0(z),
\]
where
\begin{align*}
g_1(z)=&\int_{(0,\infty)}\int_{[0,t)}
\frac{\mu_1(ds)\mu_2(dt)}{(z+s)^{\alpha_1}
(z+t)^{\alpha_2}},\\
g_2(z)=&\int_0^\infty\int_{(t,\infty)}
\frac{\mu_1(ds)\mu_2(dt)}{(z+s)^{\alpha_1}
(z+t)^{\alpha_2}}=
\int_{(0,\infty)}\int_{[0,t)}
\frac{\mu_2(ds)\mu_1(dt)}{(z+s)^{\alpha_2}
(z+t)^{\alpha_1}},\\
g_0(z)=&\int_0^\infty
\frac{\mu_1(\{t\})\mu_2(dt)}{(z+t)^{\alpha_1+\alpha_2}}=
\int_0^\infty\frac{\mu_2(\{t\})\mu_1(dt)}{(z+t)^{\alpha_1+\alpha_2}}.
\end{align*}
Clearly, all of the integrals above are finite, and
 the function $[0, \infty)\ni t \to \mu_2(\{t\})$ is Borel since its support is countable.
 Next, by (\ref{A2}), positivity of integrands,
and Fubini's theorem,
\begin{align*}
&B(\alpha_1,\alpha_2)g_1(z)\\
=&
\int_{(0,\infty)}\int_{[0,t)}
\int_{(s,t)}\frac{(\tau-s)^{\alpha_2-1}(t-\tau)^{\alpha_1-1}\,d\tau }{(z+\tau)^{\alpha_1+\alpha_2}}\,
\frac{\mu_1(ds)\mu_2(dt)}{(t-s)^{\alpha_1+\alpha_2-1}}\\
=&
\int_0^\infty\int_{(0,t)}
\int_{[0,\tau)}\frac{(\tau-s)^{\alpha_2-1}(t-\tau)^{\alpha_1-1} }
{(t-s)^{\alpha_1+\alpha_2-1}(z+\tau)^{\alpha_1+\alpha_2}}\,
\mu_1(ds)\,d\tau \,\mu_2(dt)\\
=&
\int_0^\infty\int_{(\tau,\infty)}
\int_{[0,\tau)}\frac{(\tau-s)^{\alpha_2-1}(t-\tau)^{\alpha_1-1} }{(t-s)^{\alpha_1+\alpha_2-1}}\,
\mu_1(ds)\,\mu_2(dt)\frac{d\tau}{(z+\tau)^{\alpha_1+\alpha_2}}.
\end{align*}
Similarly,
\begin{align*}
&B(\alpha_1,\alpha_2)g_2(z)\\
=&
\int_{0}^{\infty}\int_{(\tau,\infty)} \int_{[0,\tau)}
\frac{(\tau-s)^{\alpha_1-1}(t-\tau)^{\alpha_2-1}}{(t-s)^{\alpha_1+\alpha_2-1}}
\mu_2(ds)\,\mu_1(dt)\frac{d\tau}{(z+\tau)^{\alpha_1+\alpha_2}}.
\end{align*}

So,  if $z>0,$ then
\begin{equation}\label{pr11}
S_{\alpha_1}[\mu_1](z)S_{\alpha_2}[\mu_2](z)=\int_{0}^\infty \frac{\mu(d\tau)}{(z+\tau)^{\alpha_1+\alpha_2}},
\end{equation}
where $\mu=\mu[\mu_1,\mu_2] \in \mathcal M^+_{\alpha_1+\alpha_2}(\R_+)$ is given by (\ref{m1}) and (\ref{m0}).
(Note that the measurability of $u$ is also guaranteed by Fubini's theorem.)
By analytic continuation, this formula extends to all $z \in \C \setminus (-\infty,0].$

To prove \eqref{ineq} under our additional assumption note that
\begin{equation}\label{equal}
\int_0^\infty \frac{\mu[\mu_1,\mu_2](dt)}{(z+t)^{\alpha_1+\alpha_2}}
=\int_0^\infty \frac{\mu_1(dt)}{(z+t)^{\alpha_1}}\,
\int_0^\infty \frac{\mu_2(dt)}{(z+t)^{\alpha_2}},\qquad z>0,
\end{equation}
so letting $z=1$ in \eqref{equal} we obtain \eqref{ineq}.

In the general case, we apply the first part of the proof to positive measures $|\mu_1|$ and $|\mu_2|$
and, using Fubini's theorem again, arrive at \eqref{m1} and \eqref{m0}.

Furthermore, if $\mu=\mu[\mu_1,\mu_2]$ is defined by \eqref{m1} and \eqref{m0}, then
\begin{equation}\label{positiv}
|\mu[\mu_1, \mu_2]| \le \mu [|\mu_1|, |\mu_2|].
\end{equation}
Hence, in view of \eqref{equal},
\begin{align}\label{positivv}
\int_0^\infty \frac{|\mu[\mu_1,\mu_2]|(dt)}{(1+t)^{\alpha_1+\alpha_2}}\le&
\int_0^\infty \frac{\mu[|\mu_1|,|\mu_2|](dt)}{(1+t)^{\alpha_1+\alpha_2}}\\
=&
\int_0^\infty \frac{|\mu_1|(dt)}{(1+t)^{\alpha_1}}\,
\int_0^\infty \frac{|\mu_2|(dt)}{(1+t)^{\alpha_2}},\notag
\end{align}
which is equivalent to \eqref{ineq}.
\end{proof}
\begin{remark}\label{rema}
Observe that by \eqref{positivv} for almost every $\tau>0,$
the double integrals in \eqref{m0} converge absolutely 
and thus the integration limits in these integrals can be interchanged.
\end{remark}

Given $\mu_1 \in \mathcal M_{\alpha_1}(\R_+)$ and $\mu_2 \in \mathcal M_{\alpha_2}(\R_+),\alpha_1, \alpha_2>0,$ the measure $\mu \in \mathcal M_{\alpha_1+\alpha_2}(\R_+)$  defined by \eqref{m1} and \eqref{m0}
will be called the \emph{Stieltjes convolution} of $\mu_1$  and $\mu_2$ and denoted by $\mu_1\otimes_{\alpha_1,\alpha_2}\mu_2.$
Theorem \ref{elprod} says that
\begin{equation}\label{prodd}
S_{\alpha_1}[\mu_1] S_{\alpha_2}[\mu_2]=S_{\alpha_1+\alpha_2}[\mu_1 \otimes_{\alpha_1,\alpha_2} \mu_2],
\end{equation}
so that with an appropriate choice of orders
the Stieltjes transform maps the Stieltjes convolution of measures into the product of their Stieltjes transforms.
This fact partially explains our terminology, see also Remark \ref{gener} below. There are other, rather simple algebraic properties of the Stieltjes convolution
resembling the usual convolution of measures on $\R_+$, see e.g. \cite{Heymann},  \cite{Miana}, \cite{Schwartz},  \cite{SrivT}, and \cite{Yak} for some of their variants in particular situations and a similar terminology (see also Remark \ref{St_term} below).
However, we decided to avoid a discussion of these properties, since they are not relevant for this paper.

At the same time, the Stieltjes convolution of measures has some unexpected
features.
The next direct corollary reveals a specific structure of representing measures
$\mu_1\otimes_{\alpha_1, \alpha_2}\mu_2$ arising
from the product of the generalized Stieltjes transforms. It emphasizes, in particular,
the fact that the Stieltjes convolution of measures does not have singular continuous part.

\begin{corollary}\label{AbsC}
For $\alpha_j >0, j=1,2$ let $\mu_j\in {\mathcal M}_{\alpha_j}(\R_+)$, and  $\mu=\mu_1\otimes_{\alpha_1,\alpha_2}\mu_2.$
Then the following hold.
\begin{itemize}
\item [(i)] If $\mu_1$ and $\mu_2$ are positive, then $\mu$ is positive too.
\item [(ii)] If
$
\supp \mu_j\subset [a,b], j=1,2,
$
then
$
\supp \mu \subset [a,b]
$
as well.
\item [(iii)]  $\mu$ has no singular continuous part.
\item [(iv)]  $\mu$ is locally absolutely continuous if and only if
$\mu_1$ and $\mu_2$ have disjoint sets of atoms.
\item [(v)] the discrete part $\mu_{d}$ of $\mu$ is of the form
\[
\mu_{d}(d\tau)=\mu_1(\{\tau\})\mu_2(d\tau)=\mu_2(\{\tau\})\mu_1(d\tau).
\]
As consequence, $\mu$ is purely discrete if and only if $\mu_1=c_1 \delta_t$ and $ \mu_2=c_2 \delta_t$
for some $c_1,c_2 \in \mathbb C$ and $t \in \R_+.$
\end{itemize}
\end{corollary}
The proof of Corollary \ref{AbsC} is straightforward, and is therefore omitted.
\begin{remark}\label{commut}
Curiously, a similar effect of absence of singular continuous component appears
for the so-called free multiplicative convolution
of Borel probability measures on $\R_+$ and for the additive convolution of such measures on $\R$ (when neither of measures is a point mass),
see \cite{Ji} and \cite{Belin} respectively. Though the reasons for this phenomena lie probably much deeper.
\end{remark}

To gain an intuition behind the Stieltjes convolution, consider the following illustrative example.

\begin{example}\label{deltames}
Let $\alpha_j >0, j=1,2.$ If $\mu_1 \in \mathcal M_{\alpha_1}(\R_+)$ and $\mu_2=\delta_a$ for some $a \ge 0,$
then using \eqref{m1} and (\ref{m0}),
we have
\[
\mu_1 \otimes_{\alpha_1,\alpha_2} \mu_2 = B(\alpha_1, \alpha_2)u + \mu(\{a\})\delta_a,
\]
where
\[
u(\tau)=\begin{cases} (a-\tau)^{\alpha_1-1}
 \int_{[0,\tau)}
\frac{(\tau-s)^{\alpha_2-1}}{(a-s)^{{\alpha_1+\alpha_2}-1}}
\mu_1(ds),&  \qquad \tau \in (0,a),\\
(\tau-a)^{\alpha_1-1}\int_{(\tau,\infty)}
 \frac{(s-\tau)^{\alpha_2-1}}{(s-a)^{{\alpha_1+\alpha_2}-1}}\,\mu_1(ds),&  \qquad \tau >a,
\end{cases}
\]
if $a >0,$ and
\[u(\tau)=\tau^{\alpha_1-1}\int_{(\tau,\infty)}
{(s-\tau)^{\alpha_2-1}}{s^{1-\alpha_1-\alpha_2}}\,\mu_1(ds),\qquad  \tau >0,\]
if $a=0.$
In particular, setting $\alpha_1=\alpha_2=1$, we have
\[
\mu_1 \otimes_{1,1} \mu_2 =u+ \mu_1(\{a\})\delta_a
\]
with
\[
u(\tau)=\chi(a-\tau)
 \int_{[0,\tau)}
\frac{\mu(ds)}{a-s}
+\chi(\tau-a)\int_\tau^\infty
 \frac{\mu(ds)}{s-a}.
\]
If $0< a<b <\infty,$
$\mu_1=\delta_ a$ and $\mu_2=\delta_b,$ then $\mu_1,\mu_2 \in \mathcal M_\alpha(\R_+)$ for all $\alpha >0,$
and we have
\[
[\delta_b \otimes_{1,1} \delta_a] (dt)=\frac{1}{b-a}\chi(t-a)\chi(b-t)\, dt= \frac{1}{b-a}\chi_{[a,b]}(t)\, dt, \qquad t >0.
\]

\end{example}

\section{Inequalities for the Stieltjes convolution of measures}\label{st_conv_in}

In this section, we shed a light on the Stieltjes convolution and  generalize  the estimate \eqref{ineq}
by measuring the right-hand side of \eqref{ineq} in stronger norms.

Observe that
if
$
\mu_j \in \mathcal M_0(\R_+), j=1,2,
$
and $\mu=\mu_1\otimes_{\alpha_1,\alpha_2}\mu_2$,
then
 by the dominated convergence  theorem,
\begin{align}
\|\mu\|_0
\le& \lim_{s\to\infty}\,
\int_0^\infty \frac{s^{\alpha_1}|\mu_1|(dt)}{(s+t)^{\alpha_1}}\,
\int_0^\infty \frac{s^{\alpha_2}|\mu_2|(dt)}{(s+t)^{\alpha_2}}\label {ineq1}\\
=&
\int_0^\infty |\mu_1|(dt)\,
\int_0^\infty |\mu_2|(dt)=\|\mu_1\|_0 \|\mu_2\|_0.\notag
\end{align}
The inequalities \eqref{ineq} and \eqref{ineq1} suggest that
 the submultiplicative property $\|\mu\|_{\beta_1+\beta_2} \le \|\mu_1\|_{\beta_1} \|\mu_2\|_{\beta_2}$
extrapolates to the whole of scale $\mathcal M_{\beta_1}(\R_+)\times \mathcal M_{\beta_2}(\R_+),$  $(\beta_1,\beta_2) \in [0,\alpha_1] \times [0,\alpha_2].$
Below we prove that it is so up to certain constants,
apart from the boundary case when one of the numbers $\beta_1$ and $\beta_2$ equals zero.
In the latter case, we show that a (necessarily) weaker inequality holds.

We will need the next simple lemma.
\begin{lemma}\label{elem}
Let  $\alpha> 0$ and let
\begin{equation}\label{element}
F_{\alpha}(t): =
\int_1^\infty\frac{ds}{s (s+t)^\alpha},
\qquad t>0.
\end{equation}
Then there exist $\tilde c_{\alpha}, c_{\alpha} >0$ such that
\begin{equation}\label{log11}
c_\alpha \frac{\log(t+e)}{(1+t)^{\alpha}}   \le F_{\alpha}(t)  \le \tilde c_\alpha \frac{\log(t+e)}{(1+t)^{\alpha}},   \qquad
t>0.
\end{equation}
\end{lemma}
\begin{proof}
Fix $t >0.$ Integrating by parts, we obtain
\[
t^\alpha F_{\alpha}(t)=\int_{1/t}^\infty\frac{ds}{s(s+1)^\alpha}=\frac{\log t}{(1/t+1)^\alpha}+
\alpha\int_{1/t}^\infty\frac{\log s\,ds}{(s+1)^{\alpha+1}},
\]
and the estimate \eqref{log11} follows.
\end{proof}

\subsection{The numbers $\beta_1$ and $\beta_2$ are separated from zero}

First we provide bounds for  $\|\mu\|_{\beta_1+\beta_2}$ in terms of
$\|\mu_1\|_{\beta_1}$ and $\|\mu_2\|_{\beta_2},$ when $\beta_j\in (0,\alpha_j],$
$\alpha_j >0,$ $j=1,2.$

Our arguments will rely on the following  two identities.
If $\alpha_1,\alpha_2>0$ and $0<s<t$,  then setting $z=1$ in \eqref{A2},
we obtain
\begin{equation}\label{CD1}
\frac{1}{(t-s)^{\alpha_1+\alpha_2-1}}\int_s^t
\frac{(\tau-s)^{\alpha_2-1}(t-\tau)^{\alpha_1-1}d\tau}
{(1+\tau)^{\alpha_1+\alpha_2}}
=\frac{B(\alpha_1,\alpha_2)}{(1+s)^{\alpha_1}(1+t)^{\alpha_2}}.
\end{equation}
Moreover,
multiplying both sides of  \eqref{A2} by
$(z+s)^{\alpha_1}(z+t)^{\alpha_2}$ and passing to the limit as $z \to \infty$ we obtain
\begin{equation}\label{CD2}
\frac{1}{(t-s)^{\alpha_1+\alpha_2-1}}
\int_s^t
(\tau-s)^{\alpha_2-1}(t-\tau)^{\alpha_1-1}d\tau
=B(\alpha_1,\alpha_2),
\end{equation}
cf. \cite[p. 298, 2.2.5(1)]{Prudnikov1}.

Fix $\alpha_j>0$ and $\beta_j\in (0,\alpha_j], j=1,2,$ and define
\begin{equation}
J (s,t):=\int_s^t
\frac{(\tau-s)^{\alpha_1-1}(t-\tau)^{\alpha_2-1}d\tau}
{(1+\tau)^{\beta_1+\beta_2}},\qquad 0<s<t.
\end{equation}

\begin{lemma}\label{NewL}
Let
$\beta_j\in (0,\alpha_j)$ and  $\gamma_j=\alpha_j-\beta_j,  j=1,2.$
Then for all $0<s<t,$
\begin{equation}\label{nu0}
J(s,t)\le \frac{\beta_1^{\beta_1} \beta_2^{\beta_2} B(\gamma_1,\gamma_2)}{(\beta_1+\beta_2)^{\beta_1+\beta_2}}\cdot \frac{(t-s)^{\alpha_1+\alpha_2-1}}{(1+s)^{\beta_2}(1+t)^{\beta_1}},
\end{equation}
and
\begin{equation}\label{nu01}
J(s,t)\le \frac{\gamma_1^{\gamma_1}\gamma_2^{\gamma_2} B(\beta_1,\beta_2)}{(\gamma_1+\gamma_2)^{\gamma_1+\gamma_2}}\cdot \frac{(t-s)^{\alpha_1+\alpha_2-1}}{(1+s)^{\beta_2}(1+t)^{\beta_1}}.
\end{equation}
\end{lemma}

\begin{proof}
Let $\nu\in (0,1)$ and $0<s<t$ be fixed. Then by H\"older's inequality,
\begin{align}
J(s,t)
=&\int_s^t
\frac{(\tau-s)^{\beta_1-\nu}(t-\tau)^{\beta_2-\nu}}
{(1+\tau)^{\beta_1+\beta_2}}(\tau-s)^{\gamma_1-(1-\nu)}(t-\tau)^{\gamma_2-(1-\nu)}\,d\tau\label{jst} \\
\le& \left(\int_s^t
\frac{(\tau-s)^{\beta_1/\nu-1}(t-\tau)^{\beta_2/\nu-1}}
{(1+\tau)^{\beta_1/\nu+\beta_2/\nu}}\, d\tau\right)^\nu 
\notag \\
\times& \left(\int_s^t
(\tau-s)^{\gamma_1/(1-\nu)-1}(t-\tau)^{\gamma_2/(1-\nu)-1}\,d\tau\right)^{1-\nu}.
\notag
\end{align}
Using (\ref{CD1}) and  (\ref{CD2}) to transform the last two terms in \eqref{jst},
we obtain
\begin{equation}\label{RR1}
J(s,t)\le
B_\nu\frac{(t-s)^{\alpha_1+\alpha_2-1}}{(1+s)^{\beta_2}(1+t)^{\beta_1}},
\end{equation}
where
\[
B_\nu=\left(B\left(\frac{\beta_1}{\nu},\frac{\beta_2}{\nu}\right)\right)^\nu \left(B\left(\frac{\gamma_1}{1-\nu},\frac{\gamma_2}{1-\nu}\right) \right)^{1-\nu}.
\]

Next, employing the relation $B(s,t)=\Gamma(t)\Gamma(s)/\Gamma(t+s), t, s>0,$ and Stirling's formula, note that
\begin{equation}\label{asbeta}
\lim_{s,t \to \infty} \frac {(s+t)^{s+t-\frac{1}{2}} B(s,t)}{\sqrt {2 \pi} s^{s-\frac{1}{2}} t^{t-\frac{1}{2}}}=1.
\end{equation}
Thus, in view of \eqref{asbeta},
\begin{equation}\label{betanu}
\lim_{\nu\to 0}\,B_\nu=\frac{\beta_1^{\beta_1} \beta_2^{\beta_2}}{(\beta_1+\beta_2)^{\beta_1+\beta_2}} B(\gamma_1,\gamma_2).
\end{equation}
So letting $\nu \to 0$ in \eqref{RR1},  the estimate (\ref{nu0})
follows.

On the other hand, since similarly to \eqref{betanu},
\[
\lim_{\nu\to 1}\,B_\nu=\frac{\gamma_1^{\gamma_1}\gamma_2^{\gamma_2}}{(\gamma_1+\gamma_2)^{\gamma_1+\gamma_2}}B(\beta_1,\beta_2),
\]
we let $\nu \to 1$ in  (\ref{RR1}) and obtain
(\ref{nu01}).
\end{proof}
Next, for $\alpha_1,\alpha_2 >0$ and $\beta_1 \in [0,\alpha_1],$ $ \beta_2 \in [0,\alpha_2],$
we  estimate the size of $\|\mu_1\otimes_{\alpha_1,\alpha_2}\mu_2\|_{\beta_1+\beta_2}.$
To simplify our presentation, we separate the cases when $\beta_j \in (0,\alpha_j), j=1,2,$
and when either $\beta_1=\alpha_1$ or $\beta_2=\alpha_2.$ 
The first case is covered by the next result.
\begin{theorem}\label{11S}
For $\alpha_j >0$ let
$
\beta_j\in (0,\alpha_j)$ and $\gamma_j:=\alpha_j-\beta_j, j=1,2.$
If $\mu_j \in \mathcal M_{\beta_j}(\R_+),  j=1,2,$
and $\mu=\mu_1\otimes_{\alpha_1,\alpha_2}\mu_2,$
then $ \mu \in \mathcal M_{\beta_1+\beta_2}(\R_+)$ and
\begin{equation}\label{BB12}
\|\mu\|_{\beta_1+\beta_2}\le \min(A_1,A_2)\|\mu\|_{\beta_1} \|\mu_2\|_{\beta_2},
\end{equation}
where
\begin{equation}\label{a1a2}
A_1:=\frac{\beta_1^{\beta_1} \beta_2^{\beta_2}}{(\beta_1+\beta_2)^{\beta_1+\beta_2}} \cdot \frac{B(\gamma_1,\gamma_2)}{B(\alpha_1,\alpha_2)}\quad \text{and} \quad
A_2:=\frac{\gamma_1^{\gamma_1}\gamma_2^{\gamma_2}}{(\gamma_1+\gamma_2)^{\gamma_1+\gamma_2}}\cdot \frac{B(\beta_1,\beta_2)}{B(\alpha_1,\alpha_2)}.
\end{equation}
\end{theorem}

\begin{proof}
By Theorem  \ref{elprod}, the Stieltjes convolution $\mu$ is given by \eqref{m1} and \eqref{m0}.
Assume first that $\mu_1$ and $\mu_2$ are positive, so that $\mu$ and $u$ in \eqref{m1} are positive.
Then, by (\ref{m0}), Fubini's theorem and Lemma \ref{NewL},
we have
\begin{align*}
\int_0^\infty \frac{u(\tau)\,d\tau}{(1+\tau)^{\beta_1+\beta_2}}
\le& \frac{1}{B(\alpha_1,\alpha_2)}\int_{(0,\infty)} \int_{[0,t)}
J(s,t)\frac{\mu_1(ds)\, \mu_2(dt)}{(t-s)^{\alpha_1+\alpha_2-1}}\\
+&\frac{1}{B(\alpha_1,\alpha_2)}\int_{(0,\infty)} \int_{[0,t)}
J(s,t)\frac{\mu_2(ds)\,\mu_1(dt)}{(t-s)^{\alpha_1+\alpha_2-1}}\\
\le& A_j\int_{(0,\infty)} \int_{[0,t)}
\frac{\mu_1(ds)\, \mu_2(dt)}{(1+s)^{\beta_1}(1+t)^{\beta_2}}\\
+&A_j\int_{(0,\infty)} \int_{[0,t)}
\frac{\mu_2(ds)\, \mu_1(dt)}{(1+s)^{\beta_2}(1+t)^{\beta_1}},
\end{align*}
where $A_j$ stands for either $A_1$ or $A_2.$
Furthermore, note that
\begin{equation}\label{aj}
A_j \ge 1, \qquad j=1,2.
\end{equation}
Indeed, setting $t=s+1$ in \eqref{RR1} we infer that
\[
\frac{B(\alpha_1,\alpha_2)}{(s+2)^{\beta_1+\beta_2}}\le J(s,s+1)\le \frac{B_\nu}{(s+1)^{\beta_1+\beta_2}},
\]
so
\[
B_\nu \ge B(\alpha_1,\alpha_2) \frac{(s+1)^{\beta_1+\beta_2}}{(s+2)^{\beta_1+\beta_2}}, \qquad s >0.
\]
Passing to the limit in the above inequality as $s \to \infty$ and recalling the definition of $A_j,$
we deduce \eqref{aj}.

Therefore, taking into account \eqref{aj}
and (\ref{m1}),
we conclude that
\begin{align*}
&\|\mu\|_{\beta_1+\beta_2}
\le A_j \int_0^\infty \int_{[0,t)}
\frac{\mu_1(ds)\, \mu_2(dt)}{(1+s)^{\beta_1}(1+t)^{\beta_2}}\\
+&A_j\int_0^\infty \int_{(t,\infty)}
\frac{\mu_1(ds)\, \mu_2(dt)}{(1+s)^{\beta_1}(1+t)^{\beta_2}}
+\int_0^\infty\frac{\mu_1 \{\tau\} \, \mu_2(d\tau)}{(1+\tau)^{\beta_1+\beta_2}}\\
\le& A_j\|\mu_1\|_{\beta_1} \|\mu_2\|_{\beta_2},
\end{align*}
and the statement follows under our positivity assumption.
In the general case, it suffices to use \eqref{positiv} and to apply the first part of the proof.
\end{proof}
Now we consider the second case where it is clearly enough to
assume that $\beta_2=\alpha_2$ and $\beta_1 \in (0,\alpha_1).$
\begin{corollary}\label{11SS}
For $\alpha_1,\alpha_2 >0$ let $\beta_2=\alpha_2$ and
$
\beta_1\in (0,\alpha_1).$
If $\mu_j \in \mathcal M_{\beta_j}(\R_+),  j=1,2,$
and $\mu=\mu_1\otimes_{\alpha_1,\alpha_2}\mu_2,$
then $ \mu \in \mathcal M_{\beta_1+\beta_2}(\R_+)$ and
\begin{equation}\label{BB12B}
\|\mu\|_{\beta_1+\beta_2}\le\frac{B(\beta_1,\beta_2)}{B(\alpha_1,\alpha_2)} \|\mu\|_{\beta_1} \|\mu_2\|_{\beta_2}.
\end{equation}
\end{corollary}
\begin{proof}
Note that if $(\beta_1,\beta_2) \in (0,\alpha_1)\times(0,\alpha_2)$ and $A_1$ and $A_2$ are defined in \eqref{a1a2}, then
\[
\lim_{\beta_2 \to \alpha_2}\,A_1=\infty \quad \text{and} \quad
\lim_{\beta_2\to \alpha_2}\,A_2= \frac{B(\beta_1,\alpha_2)}{B(\alpha_1,\alpha_2)}.
\]
Moreover, $\|\mu\|_{\beta_2} \to \|\mu\|_{\alpha_2}$ and $\|\mu\|_{\beta_1+\beta_2} \to \|\mu\|_{\beta_1+\alpha_2}$
as $\beta_2 \to \alpha_2.$
Thus, passing to the limit in \eqref{BB12}, we obtain \eqref{BB12B}.
\end{proof}
\begin{remark}\label{RemAB}
Observe that
\[
\lim_{(\beta_1,\beta_2)\to (0,0)}\,A_1=1 \quad \text{and}
\lim_{(\beta_1, \beta_2)\to (\alpha_1, \alpha_2)}\,A_2=1,
\]
so  passing to the limit in \eqref{BB12} as $(\beta_1,\beta_2)\to (0,0)$ or as $(\beta_1,\beta_2)\to (\alpha_1,\alpha_2)$ we
obtain either  \eqref{ineq1}  or \eqref{ineq}, respectively.
\end{remark}

It is instructive to note that the constant $\min(A_1, A_2)$ in \eqref{BB12} cannot in general be replaced by $1,$ and in fact the optimal constant in \eqref{BB12} can be as large as one pleases. To see this consider the following example.
\begin{example}\label{constt}
Let $a>0,$  and let
$\alpha_1=\alpha_2=1$ and $\mu_1=\delta_0, \mu_2=\delta_a.$
Then, in view of Example \ref{deltames}, one has
\[
\mu=\mu_1\otimes_{1,1}\mu_2=a^{-1}{\chi_{[0,a]}}.
\]
Moreover,
\[
\|\mu\|_{1}=\frac{1}{a}\int_0^a\frac{dt}{1+t}=\frac{\log(1+a)}{a},
\]
and for any $\alpha>0,$
\[
\|\mu_1\|_\alpha=1\qquad \text{and} \qquad \|\mu_2\|_\alpha=\frac{1}{(1+a)^\alpha}.
\]
Thus, if $\delta \in (0,1),$ then there is $C(\delta)>0$ such that
\[
\|\mu\|_1\le C(\delta)\|\mu_1\|_{1-\delta} \|\mu_2\|_\delta,
\]
hence
\[
\frac{\log(1+a)}{a}\le \frac{C(\delta)}{(1+a)^\delta}.
\]
Letting now $\delta\to 1,$ we infer that
\[
\lim_{\delta\to 1}\,C(\delta)=\infty.
\]
\end{example}

The constants given in Theorem \ref{11S} are not optimal and can be improved
in several situations of interest. In particular, using \eqref{RR1}, one can prove that if
$\alpha_1=\alpha_2=1, $ $\beta_1=\beta_2=\beta \in (0,1),$ and $\mu_1, \mu_2$ and $\mu$ are as in Theorem \ref{11S}, then
$\|\mu\|_{2\beta} \le \|\mu_1\|_{\beta}\|\mu_2\|_{\beta}$ for every $\beta \in (0,1).$  Since, in this paper, we are not interested in the best constants,
we omit a discussion of further details.

\subsection{The boundary case: one of the numbers $\beta_1$ and $\beta_2$ equals zero}
In this subsection we show that in the boundary case when  $\beta_1 \in (0, \alpha_1]$ and $\beta_2=0$
the inequality \eqref{BB12} holds up to small perturbations of $\alpha_1.$
(The case when $\beta_2 \in (0,\alpha_2]$ and $\beta_1=0$ is clearly analogous.)
On the other hand, we prove that  our bounds are optimal with respect to
$\beta_1$ and $\beta_2$,
and in this case \eqref{BB12}  does not, in general, hold. For $\mu \in \mathcal M_{\alpha}(\R_+), \alpha >0,$ 
define
\begin{equation}\label{Con1}
\|\mu\|_{\alpha,{\rm log}}:=\int_0^\infty \frac{\log(t+e)\, |\mu|(dt)}{(1+t)^{\alpha}},
\end{equation}
where the integral can be infinite.
\begin{theorem}\label{EL1L}
For $\alpha_1,\alpha_2>0$ let $\mu_1 \in \mathcal M_{\alpha_1}(\R_+),$ $\mu_2 \in \mathcal M_0(\R_+),$
and $\mu=\mu_1 \otimes_{\alpha_1,\alpha_2}\mu_2.$
\begin{itemize}
\item [(i)] If
$\psi: [0,\infty)\to [1,\infty)$ is
an increasing function satisfying
\begin{equation}\label{Bpsi}
B_\psi:=\int_1^\infty\frac{dt}{t\psi(t)}<\infty,
\end{equation}
then there is $C=C(\alpha_1,\alpha_2,\psi)>0$ such that
\begin{equation}\label{LLQ1DopA}
\int_0^\infty \frac{|\mu|(dt)}{\psi(t)(1+t)^{\alpha_1}}
\le C
\|\mu_1\|_{\alpha_1} \|\mu_2\|_0.
\end{equation}
\item [(ii)]
If $\|\mu_1\|_{\alpha_1,{\rm log}}<\infty,$
then
\begin{equation}\label{LLQ1}
\|\mu\|_{\alpha_1}\le \frac{\alpha_2^{-1}+\tilde c_{\alpha_1}}{B(\alpha_1,\alpha_2)}
\|\mu_1\|_{\alpha_1{\rm ,log}} \|\mu_2\|_0,
\end{equation}
with  $\tilde c_{\alpha_1}$ as in Lemma \ref{elem}.
\end{itemize}
\end{theorem}

\begin{proof}
Taking into account \eqref{positiv}, without loss of generality, we may suppose that $\mu_1$ and $\mu_2$ are positive, so that $\mu$ is positive as well.

To prove (i), note that if  $\alpha>\nu>0$ and $t\ge 2,$ then
\[
 \int_1^\infty \frac{s^{\nu-1}\,ds}{\psi(s)(s+t)^{\alpha}}
\ge \frac{1}{\psi(t)}\int_{t/2}^t\frac{s^{\nu-1}\,ds}{(s+t)^{\alpha}}
=\frac{1}{\psi(t)t^{\alpha-\nu}}
\int_{1/2}^1\frac{s^{\nu-1}\,ds}{(s+1)^{\alpha}}.
\]
So, there exists $D_{\nu,\alpha}>0$ such that
\begin{equation}\label{RRA}
 \int_1^\infty \frac{s^{\nu-1}\,ds}{\psi(s)(s+t)^{\alpha}}
\ge
\frac{D_{\nu,\alpha}}{\psi(t)(1+t)^{\alpha-\nu}},\qquad  t>0.
\end{equation}

So, using  (\ref{RRA}), \eqref{Bpsi} and Fubini's theorem,
we infer that
\begin{align*}
D_{\alpha_2,\alpha_1+\alpha_2} &\int_0^\infty \frac{\mu(dt)}{\psi(t)(1+t)^{\alpha_1}}\le
\int_0^\infty \int_1^\infty \frac{s^{\alpha_2-1}\,ds}{\psi(s)(s+t)^{\alpha_1+\alpha_2}}
\mu(dt)\\
=&\int_1^\infty \int_0^\infty \frac{\mu_1(dt)}{s\psi(s)(s+t)^{\alpha_1}}
\int_0^\infty \frac{s^{\alpha_2}\mu_2(dt)}{(s+t)^{\alpha_2}}ds\\
\le&\|\mu_2\|_0 \int_0^\infty \int_1^\infty \frac{ds}{s\psi(s)(1+t)^{\alpha_1}}
\mu_1(dt)\\
=& B_\psi\|\mu_1\|_{\alpha_1} \|\mu_2\|_0,
\end{align*}
hence (\ref{LLQ1DopA}) is true.

Let us now prove (ii).  Recall that
\begin{equation}\label{BB}
\int_0^\infty\frac{s^{\alpha_2-1}\,ds}{(s+t+1)^{\alpha_1+\alpha_2}}=
\frac{B(\alpha_1,\alpha_2)}{(t+1)^{\alpha_1}},\quad t>0,
\end{equation}
see e.g. \cite[p. 298, no.24]{Prudnikov1}.

If  (\ref{Con1}) holds, then
by  (\ref{BB}),
\eqref{log11} and Fubini's theorem
 we have
\begin{align*}
\label{Prev}
B(\alpha_1,\alpha_2)\|\mu\|_{\alpha_1}
=&\int_0^\infty \int_0^\infty \frac{s^{\alpha_2-1}\,ds}{(s+t+1)^{\alpha_1+\alpha_2}}\mu(dt)\\
\le& \int_0^1 s^{\alpha_2-1}\,ds\int_0^\infty \frac{\mu(dt)}{(t+1)^{\alpha_1+\alpha_2}}\\
+&\int_1^\infty s^{-1}\int_0^\infty \frac{\mu_1(dt)}{(s+t+1)^{\alpha_1}}
\int_0^\infty \frac{s^{\alpha_2}\mu_2(dt)}{(s+t+1)^{\alpha_2}}ds\\
\le& \alpha_2^{-1}\|\mu_1\|_{\alpha_1} \|\mu_2\|_{\alpha_2}
+\|\mu_2\|_0 \int_0^\infty F_{\alpha_1}(t)\mu_1(dt)\\
\le& (\alpha_2^{-1}+ \tilde c_{\alpha_1})\|\mu_1\|_{\alpha_1,{\rm log}} \|\mu_2\|_0,
\end{align*}
where  $F_{\alpha_1}$ is given by \eqref{element}, hence (\ref{LLQ1}) follows.
\end{proof}

Thus we arrive at the following immediate consequence of Theorem \ref{EL1L}.
Since only one of the parameters $\beta_1$ and $\beta_2$ is non-zero,
we discard the $(\beta_1,\beta_2)$-notation and use instead a parameter $\beta.$
\begin{corollary}\label{corr}
Let $\alpha_j >0, j=1,2.$ If $\mu_1 \in \mathcal M_{\alpha_1}(\R_+),$ $\mu_2 \in \mathcal M_0(\R_+),$
and $\mu=\mu_1\otimes_{\alpha_1,\alpha_2}\mu_2,$ then there exists
$C=C(\alpha_1,\alpha_2)>0$ such that
\begin{equation}\label{corr_mes}
\|\mu\|_{\beta}\le
\begin{cases} C \|\mu_1\|_{\alpha_1} \, \|\mu_2\|_0,& \quad \text{if} \,\, \beta >\alpha_1,\\
 C \|\mu_1\|_{\beta} \, \|\mu_2\|_0, & \quad \text{if} \,\, 0< \beta < \alpha_1\,\, {and} \,\, \|\mu_1\|_\beta <\infty.
\end{cases}
\end{equation}
\end{corollary}

It appears that \eqref{corr_mes} is close to be optimal with respect to the size
of $\mu_1,$ and one cannot in general replace $\|\mu_1\|_{\alpha_1, {\rm log}}$ by $\|\mu\|_{\alpha_1}.$
\begin{theorem}\label{L0}
Let
$\mu_j \in \mathcal M^+_{\alpha_j}(\R_+), \alpha_j>0, j=1,2,$
and $\mu=\mu_1 \otimes_{\alpha_1,\alpha_2}\mu_2.$
Then
\begin{equation}\label{Nau}
\|\mu_1\|_{\alpha_1,{\rm log}} \,  \|\mu_2\|_{\alpha_2}\le
\frac{2^{\alpha_1+\alpha_2}\|\mu\|_{\alpha_1}}{{c}_{\alpha_1} B(\alpha_1,\alpha_2)},
\end{equation}
where $c_{\alpha_1}$ is as in Lemma \ref{elem}.
\end{theorem}

\begin{proof}
By   (\ref{log11}),
noting that
\[
\frac{s}{s+t+1}\ge \frac{1}{2(t+1)},\qquad s\ge 1,\quad t>0,
\]
and using \eqref{BB}, we have
\begin{align*}
B(\alpha_1,\alpha_2)\|\mu\|_{\alpha_1}
\ge& \int_0^\infty \int_1^\infty \frac{s^{\alpha_2-1}\,ds}{(s+t+1)^{\alpha_1+\alpha_2}}\mu(dt)\\
\ge& \int_1^\infty s^{-1}\int_0^\infty \frac{\mu_1(dt)}{(s+t+1)^{\alpha_1}}
\int_0^\infty \frac{s^{\alpha_2}\mu_2(dt)}{(s+t+1)^{\alpha_2}}ds\\
\ge& \frac{\|\mu_2\|_{\alpha_2}}{2^{\alpha_1+\alpha_2}}\int_0^\infty F_{\alpha_1}(t)\,\mu_1(dt)
\ge \frac{{c}_{\alpha_1}}{2^{\alpha_1+\alpha_2}}\|\mu_1\|_{\alpha_1,{\rm log}} \, \|\mu_2\|_{\alpha_2},
\end{align*}
i.e. (\ref{Nau}) holds.
\end{proof}
The following corollary of Theorem \ref{L0} is straightforward.
\begin{corollary}
Under the assumptions of Theorem \ref{L0},
\begin{equation}\label{equiv}
\|\mu\|_{\alpha_1}<\infty\quad  \Longrightarrow \quad
\|\mu_1\|_{\alpha_1,{\rm log}}<\infty,
\end{equation}
and the implication in \eqref{equiv} becomes equivalence if  $\mu_2 \in \mathcal M^+_0(\R_+).$
\end{corollary}

For a concrete example showing that \eqref{corr_mes}
may fail if $\beta=\alpha_1,$  consider, for instance,
\[
\mu_1(dt)=\frac{(1+t)^{\alpha_1-1}}{\log^2(t+e)}\, dt \qquad \text{and}\qquad  \mu_2(dt)=\delta_0(dt).
\]

\begin{remark}\label{gener}
In the context of Sections  \ref{st_conv} and \ref{st_conv_in},
let us consider the
products $f_1 f_2,$
where
$$
f_j=a_j+S_{\alpha_j}[\mu_j], \,\, a_j \in \mathbb C, \,\, \alpha_j >0, \quad \text{and} \quad
\mu_j\in \mathcal{M}_{\alpha_j}(\R_+), \,\, j=1,2.
$$
Note that if  $a_j \ge 0$ and $\mu_j \in \mathcal M^+_{\alpha_j}(\R_+),j=1,2,$ then  $f_1$ and $f_2$ are generalized Stieltjes.

Recall that  by e.g. \cite[Theorem 3]{Karp}
if $\beta >\alpha >0$ and  $\nu \in \mathcal M_\alpha^+(\R_+),$ then $S_\alpha [\nu]=S_\beta [\mu]$ for  $\mu \in\mathcal M^+_\beta(\R_+)$
defined by
\begin{equation}\label{up}
\mu(dt)=\left(
(B(\alpha, \beta-\alpha))^{-1} \int_{[0,t)} \frac{\nu(ds)}{(t-s)^{\alpha+1-\beta}}\right)\, dt.
\end{equation}
This fact generalizes to the case when $\nu \in \mathcal M_\alpha(\R_+)$ and  $\mu \in\mathcal M_\beta(\R_+)$ by decomposing $\nu$ and $\mu$  into linear combinations of four
summands from $\mathcal M^+_\alpha (\R_+)$ and $\mathcal M^+_\beta(\R_+),$ respectively.
So, employing  \eqref{up} in such a general setting,
we have
\begin{align}
B(\alpha_1,\alpha_2) &(f_1(z)f_2(z)-S_{\alpha_1}[\mu_1](z) S_{\alpha_2}[\mu_2](z))
=a_1 a_2 B(\alpha_1, \alpha_2)\label{general} \\
+&a_1
\int_{0}^\infty \left(\int_{[0,\tau)}(\tau-t)^{\alpha_1-1}\,\mu_2(dt)\right)\,\frac{d\tau}{(z+\tau)^{\alpha_1+\alpha_2}}\notag \\
+&a_2
\int_{0}^\infty \left(\int_{[0,\tau)}(\tau-t)^{\alpha_2-1}\,\mu_1(dt)\right)\,\frac{d\tau}{(z+\tau)^{\alpha_1+\alpha_2}}\notag \\
=&S_{\alpha_1+\alpha_2}[u](z), \qquad z >0,\notag
\end{align}
where $u$ is
given by
\begin{equation*}
B(\alpha_1,\alpha_2)u(\tau)= a_1\int_{[0,\tau)}(\tau-t)^{\alpha_1-1}\,\mu_2(dt) +a_2 \int_{[0,\tau)}(\tau-t)^{\alpha_2-1}\,\mu_1(dt),
\end{equation*}
and $u \in  L^1_{\alpha_1+\alpha_2}(\R_+)$ by Fubini's theorem.
By Theorem \ref{elprod},
$S_{\alpha_1}[\mu_1] S_{\alpha_2}[\mu_2]=S_{\alpha_1+\alpha_2}[\mu]$ with $\mu=\mu_1 \otimes_{\alpha_1,\alpha_2},$ so
\begin{equation}\label{sing_const}
f_1 f_2 =a_1 a_2 + S_{\alpha_1+\alpha_2}[u]+S_{\alpha_1+\alpha_2}[\mu]=a_1a_2 + S_{\alpha_1+\alpha_2} [u+\mu].
\end{equation}
The term $S_{\alpha_1+\alpha_2} [u]$ is much simpler than $S_{\alpha_1+\alpha_2} [\mu].$
So the analogues  of statements obtained in Sections  \ref{st_conv}  and \ref{st_conv_in}
can be formulated for $S_{\alpha_1+\alpha_2} [\mu + u]$
as well, covering, in particular, the situation when $f_1$ and $f_2$ are generalized Stieltjes functions.
However, to not obscure our presentation with technical details,
we have decided to consider only the case $a_1=a_2=0$ and to deal solely with the generalized Stieltjes transforms.
\end{remark}

\section{The Stieltjes convolution, Hlibert transform and Tricomi's identity}\label{st_conv1}

While our studies in Sections \ref{st_conv} and \ref{st_conv_in} concerned the Stieltjes convolution on $\mathcal M_{\alpha_1}(\R_+)\times  \mathcal M_{\alpha_2}(\R_+)$ for arbitrary $\alpha_1,\alpha_2>0,$
they have several curious applications in the setting of classical
 Stieltjes transforms, i.e. when  $\alpha_1=\alpha_2=1$ and $\mu_j(dt)=f_j(t)\, dt$ with $f_j \in L^1_{1}(\R_+),
j=1,2.$ In this section, we will investigate this situation  in some more details
and reveal its close relations to  Hilbert transforms and Tricomi's identity.
These considerations will be based on Theorem \ref{elprod} yielding the explicit expression  \eqref{m0} for the Stieltjes convolution of $\mu_1$ and $\mu_2$.
Note that there is an illuminating probabilistic interpretation of $\mu_1\otimes_{1,1}\mu_2,$  $\mu_1,\mu_2 \in \mathcal M^+_0(\R_+),$ as the distribution of a random variable uniformly distributed between two independent random variables, see \cite{Asch}. See also \cite{Johnson} for a similar interpretation and \cite{Soltani} for additional references.

Thus we deal with the classical Stieltjes transform $S_1$ defined on $L^1_1(\R_+)$ by
\[
S_1[g](z):=\int_0^\infty \frac{g(t)\, dt}{z+t}, \qquad g \in  L^1_1(\R_+), \quad z\in \C\setminus
 (-\infty,0].
\]
Given $g_1,g_2 \in  L^1_1(\R_+),$ it is of interest to know 
when $S_1[g_1]S_1[g_2]$ is again of the form $S_1[g]$ for some $g \in  L^1_1(\R_+).$
Using Theorem \ref{elprod} and the Stieltjes convolution  $g_1\otimes_{1,1} g_2$, we proceed with providing a partial answer to this question in terms of Hilbert transforms of $g_1$ and $g_2$.
To this aim,  note that if
$g_1, g_2 \in L^1_1(\R_+)$, then Theorem \ref{elprod} and Remark \ref{rema} imply that
$g_1 \otimes_{1,1} g_2 \in L^1_2(\R_+)$
is given by
\begin{equation}\label{deriv}
[g_1 \otimes_{1,1} g_2](\tau)= \int_{0}^\tau\int_{\tau}^\infty
\frac{g_2(s)\, ds}{s-t}\,g_1(t)\, dt
+\int_{0}^\tau \int_{\tau}^\infty  \frac{g_1(s)\, ds}{s-t} \, g_2(t)\, dt,
\end{equation}
for almost every $\tau \in \R_+,$ and
\[
\|g_1 \otimes_{1,1} g_2\|_{L^1_2(\R_+)}
 \le
\|g_1\|_{L^1_1(\R_+)}  \|g_2\|_{L^1_1(\R_+)}.
\]

The identity \eqref{deriv} suggests to employ Hilbert transforms. Recall that for  $g \in L^1_1(\R)$ its Hilbert transform $H[g]$ is defined by
\begin{equation}\label{hilbert}
H[g](\tau):= \text{p.v.}\, \frac{1}{\pi}\int_{\R} \frac{g(t)\, dt}{t-\tau}=\frac{1}{\pi}\lim_{\epsilon \to 0}\int_{|t-\tau|>\epsilon} \frac{g(t)\, dt}{t-\tau},
\end{equation}
where the limit exists for almost all $\tau >0$ by e.g. \cite[p. 348-349]{Havin}.
We will need a classical boundedness property of  the Hilbert transforms
on weighted $L^p$-spaces saying that for  $ p \in (1,\infty)$
and  $\alpha\in [0,1)$ there exists $c_{p,\alpha}>0$ such that
\begin{equation}\label{Hww}
\|H[g]\|_{L^p_\alpha(\R)}\le c_{p,\alpha}
\|g\|_{L^p_\alpha(\R)}, \qquad g \in L^p_\alpha(\R),
\end{equation}
see e.g. \cite[Example 9.1.7 and Theorem 9.4.6]{Grafakos1}.
(Alternatively, one may consult 
 \cite[Theorem 3]{Koizumi} for a direct argument.)

\begin{theorem}\label{Stiv}
Let
$
g_j\in L^{p_j}_{\beta_j}(\R_{+}),
$
with
\begin{equation}\label{Ca1}
p_j\in (1,\infty),\qquad \frac{1}{p_1}+\frac{1}{p_2}\le 1,\qquad \beta_j\in [0,1),\quad j=1,2.
\end{equation}
Then $g_1$, $g_2$, $g_2 H[g_1],$ and $g_1 H[g_2]$ belong to $ L^1_1(\R_+),$
and
\begin{equation}\label{m22}
S_1[g_1] S_1[g_2]=S_1\bigl[g_2 H[g_1]+g_1 H[g_2]\bigr].
\end{equation}
\end{theorem}

\begin{proof}
By
H\"older's inequality, one  has
$
g_j\in L^{1}_1(\R_{+}), j=1,2.
$
Moreover,  in view of \eqref{Hww}, we have $ g_1 H[g_2], g_2 H[g_1] \in  L^{1}_1(\R_{+})$ as well.
Indeed, let $\gamma  \in (0,\infty]$ be such that
\[
\frac{1}{p_1}+\frac{1}{p_2}+\frac{1}{{\gamma}}=1.
\]
Then, by H\"older's inequality and (\ref{Hww}), 
\[
\|g_1 H[g_2]\|_{L^1_1(\R_{+})}
\le
c_{p_2,\beta_2} C_{\gamma} \|g_1\|_{L^{p_1}_{\beta_1}(\R_{+})}
\|g_2\|_{L^{p_2}_{\beta_2}(\R_{+})},
\]
where
$C_\gamma=1$ if $\gamma=\infty$, and
\[
C_{\gamma}=\left(\int_0^\infty \frac{dt}{(1+t)^{(1-\beta_1/p_1-\beta_2/p_2){\gamma}}}\right)^{1/{\gamma}}<\infty \qquad \text{if}\, \, \gamma \in (0,\infty).
\]
The estimate for $\|g_1 H[g_2]\|_{L^1_1(\R_{+})}$ is completely analogous,
so that if 
\[g:=g_2 H[g_1]+g_1 H[g_2],\]
then
 $g \in L^1_1(\R_+).$

If $\tau >0,$ then taking into account \eqref{Hww} and applying
Parseval's identity for Hilbert transforms (see e.g. \cite[Eq. (4.176)]{King}) to
\[
f_1:=\chi_{(0,\tau)}g_1\in L^p(\R_{+}),\qquad p:=p_1,
\]
and
\[
f_2:=\chi_{(0,\tau)}g_2\in L^q(\R_{+}),\qquad q:=\frac{p_1}{p_1-1},
\]
we conclude that
\begin{align}\label{identity}
0=&\pi \int_0^\infty f_1(t) H[f_2](t)\,dt + \pi \int_0^\infty H[f_1](t)f_2(t)\,dt\\
=& \int_0^\tau \int_0^\tau \frac{g_2(s)\, ds}{s-t}\, g_1(t)\, dt+ \int_0^\tau \int_0^\tau \frac{g_1(s)\, ds}{s-t}\, g_2(t)\, dt,\notag
\end{align}
where the inner integrals are understood in the principal value sense.
Next, adding the right-hand sides of \eqref{deriv} and \eqref{identity}, we can write
\begin{equation}\label{denform}
(g_1\otimes_{1,1} g_2) (\tau)
=\int_0^\tau g(t)\, dt
\end{equation}
for almost every $\tau >0.$ Since $g \in L^1_1(\R_+),$ we have
\begin{equation}\label{cesaro}
\lim_{\tau \to \infty}\frac{1}{\tau} \int_0^\tau g(t)\, dt =0.
\end{equation}
Thus, employing \eqref{denform}, integrating by parts and taking into account \eqref{cesaro},
we infer that
\begin{equation*}
S_1[g_1](z) S_1 [g_2](z)=-\int_0^\infty (g_1\otimes_{1,1} g_2) (\tau)\,  d(z+\tau)^{-1}=S_1[g](z)
\end{equation*}
for all $z >0,$ i.e. \eqref{m22} holds.
\end{proof}

Observe that $g_2H[g_1]+g_1 H[g_2]$ may, in general, change its sign
even if $g_j\ge 0$, $j=1,2$.

\begin{remark}\label{St_term} It seems that the identity (\ref{m22})
appeared for the first time in \cite[$\S\,15$]{Yak1} (see also \cite{Yak}),
where it was proved
for a certain class of functions  associated with a Mellin transform.
Later the result was extended to $L^p$-spaces in \cite{SrivT},
where it was assumed that
\[
g_j\in L^{p_j}(\R_{+}),\quad p_j\in (1,\infty),\quad j=1,2,\quad
\frac{1}{p_1}+\frac{1}{p_2}<1.
\]
Theorem \ref{Stiv} is a  substantial generalization of this result
based on a direct use of the convolution formula \eqref{m1} and \eqref{m0}.
Note that
 the term "Stieltjes convolution"
was employed in the literature to denote $g_1H[g_2]$ + $g_2H[g_1]$ rather than $g_1\otimes_{1,1}g_2.$ In view of the material
presented in Sections \ref{st_conv} and \ref{st_conv_in}, we  believe that our terminology is more natural and revealing.
\end{remark}
\medskip

The appearance of Hilbert transforms in the right-hand side of \eqref{m22}
suggests  to look for a version of \eqref{m22} concerning  functions defined
on the whole real line. On this way, we show that  the convolution formula (\ref{m22})
for  $g_j\in L^{p_j}(\R_{+})$, $p_j\in (1,\infty)$, $j=1,2$, $1/p_1+1/p_2\le 1,$
is equivalent  to the next well-known version of Poincar\'e-Bertrand's formula:
\begin{equation}\label{Carton}
{H}[f_1](t) {H}[f_2](t)-f_1(t) f_2(t)={H}\bigl[f_1  {H}[f_2]+ f_2 {H}[f_1]\bigr](t)
\end{equation}
 for $f_1\in L^{p_1}(\R)$ and $f_2\in L^{p_2}(\R)$ and almost all  $t \in \mathbb R.$
See \cite[Chapters 2.13, 4.16 and 4.23]{King} and \cite{Okada} for a general discussion of \eqref{Carton} and related identities.
Apparently, the identity \eqref{Carton} was first proved by Tricomi  for the case $1/p_1+1/p_2< 1,$ see  \cite[Theorem IV]{Tricomi}.
It was then extended  in \cite[Theorem  \ref{elprod}]{Carton} and \cite[Theorem]{Love} to include a more general assumption  $1/p_1+1/p_2\le 1.$
The approach in the second paper was rather technical, while the first paper offered a comparatively simple proof and gave
several interesting consequences.
As a byproduct of our technique,  we prove a slightly more general version of \eqref{Carton}
valid under the assumptions of Theorem \ref{Stiv}.

To prove the equivalence of Theorem \ref{Stiv} and \eqref{Carton} for  $f_j\in L^{p_j}_{\beta_j}(\R)$, with $\beta_j$ and $p_j$ satisfying \eqref{Ca1}, $j=1,2,$
we first note that  one implication is rather direct.
 Indeed, assume that \eqref{Carton} holds for $f_j\in L^{p_j}_{\beta_j}(\R), j=1,2.$
By restricting (\ref{Carton}) to $t \in (-\infty,0),$ replacing in this identity $f_1$ with  $f_1 \chi_{\R_+}$
 and $f_2$ with $f_2 \chi_{\R_+},$ and setting $t=-s,$
we obtain  (\ref{m22}). So, in fact, the result in \cite{SrivT} is an easy corollary of a version of Tricomi's identity.

We proceed with deriving the other implication.
To this aim, recall the classical Sokhotski-Plemelj jump formula for Hilbert transforms:
\begin{equation}\label{Ss}
\lim_{\epsilon\to +0}\,
\frac{1}{\pi}S_1[f](-t\mp i\epsilon)=\pm if(t)+{H}[f](t),\qquad \text{a.e.} \quad t>0.
\end{equation}
for all $f \in L^p(\R_+), p \in (1,\infty),$ see e.g. \cite[p. 348-349]{Havin} or  \cite[Theorem 5.30]{Rosen}.
\begin{theorem}\label{TS1}
Let  $f_j \in L^{p_j}_{\beta_j}(\R), j=1,2,$ be such that \eqref{Ca1} holds.
Then $f_1$, $f_2$,  $f_2 H[f_1],$ and $f_1 H[f_2]$ belong to $L^1_1(\R),$
and $f_1$ and $f_2$ satisfy the Tricomi identity (\ref{Carton}).
\end{theorem}

\begin{proof}
As in the proof of Theorem \ref{Stiv}, the fact that $f_1, f_2$ and $f_2 H[f_1]+f_1 H[f_2] $ are in $ L^1_1(\R)$ is a direct corollary of \eqref{Hww} and H\"older's inequality.

Thus, the main issue is to transfer  Theorem \ref{Stiv} to the setting of the whole of $\R.$
This will be done by
 shifting  $f_1$ and $f_2$ backwards and observing that
via density arguments it suffices to prove  (\ref{Carton}) for $f_1$ and $f_2$ with compact support.
Let $a\in (0,\infty)$ be fixed and $\supp(f_j)\subset [-a,a], j=1,2$.
Define
\[
{f}_{j,a} (t):=f_j(t-a),\qquad \text{a.e.}\,\, t>0,\quad j=1,2.
\]
Then, by Theorem \ref{Stiv}, we have
\begin{equation}\label{S0}
\frac{1}{\pi}S_1[{f}_{1,a}](z) \frac{1}{\pi}S_1[{f}_{2,a}](z)
=\frac{1}{\pi}S_1 \bigl[{f}_{1,a}\, {H}[f_{2,a}]+ {f}_{2,a}\, {H}[{f}_{1,a})]\bigr](z)
\end{equation}
for all $z\in \C\setminus (-\infty,0].$ In particular, setting $z=t,$
$t>0,$ in (\ref{S0}), we obtain
\[
{H}[{f}_{1,a}](-t)  H [{f}_{2,a}](-t)
={H}\bigl[{f}_{1,a}\, {H} [{f}_{2,a}]+{f}_{2,a}\, {H}[ {f}_{1,a}]\bigr](-t).
\]
This means that (\ref{Carton}) holds for $f_1$ and $f_2$ and $t<-a$.

Next, setting $z=-t-i\epsilon$ with $t,\epsilon>0$ in (\ref{S0}),  letting $\epsilon\to 0$,
and taking into account (\ref{Ss}), we conclude that
\begin{align}
(i{f}_{1,a}(t)+&{H}[f_{1,a}](t))
(i{f}_{2,a}(t)+{H}[f_{2,a}](t)) \label{eq1} \\
=&i{f}_{1,a}(t) {H}[f_{2,a}](t)+
i{f}_{2,a}(t) {H}[f_{1,a}](t)\notag \\+&
{H}\bigl[{f}_{1,a}\,
{H}[{f}_{2,a}]+{f}_{2,a}\, {H}[f_{1,a}]\bigr](t)\notag
\end{align}
for almost all $t >0.$
Similarly, applying the same argument for $z=-t+i\epsilon,$ we have
\begin{align}
(-if_{1,a}(t)+&{H}[f_{1,a}](t))
(-if_{2,a}(t)+{H}[f_{2,a}](t))\label{eq2}\\
=&-if_{1,a}(t) {H}[f_{2,a}](t)-
if_{2,a}(t) {H}[f_{1,a}](t)\notag \\+&
{H}\bigl[f_{1,a}\,
{H}[f_{2,a}]+f_{2,a}\, {H}[f_{1,a}] \bigr](t),\notag
\end{align}
for almost all $t >0.$

So, summing up \eqref{eq1} and \eqref{eq2},
and returning to $f_1$ and $f_2,$ we infer that for almost every $t>-a:$
\begin{equation}\label{Gv}
{H}[f_1](t) {H}[f_2](t)
-f_1(t) f_2(t)
={H}\bigl [f_1
{H}[f_2]+f_2 {H}[f_1]\bigr ](t).
\end{equation}

Since the choice of $a$ was arbitrary, this implies the claim.
\end{proof}
Note that quite an interesting application of Tricomi's identity to the study of nonlinear PDE
was found recently in \cite{Elgindi}.

\section{Representability problem for generalized Stieltjes transforms}\label{repres}

As we discussed in the introduction, finding  a convenient criterion
for representability of a function as a generalized Stieltjes transform of order $\alpha>0, \alpha \neq 1,$
is a difficult and challenging problem.
In this section, among other things, we show
that if $f \in \mathcal S_\alpha$ and $f_\beta(z):=f(z^\beta), \beta  \in (0,1],$ then
$f_\beta \in \mathcal S_{\alpha \beta},$ so that $f_\beta \in \mathcal S_\beta$ if $f \in \mathcal S_1.$
 Since the class $\mathcal S_1$ is well-understood, our results lead
to a transparent description of a substantial subset of $\mathcal S_\alpha,$
and clarify the representability problem to some extent.

Our construction will be realized in two ways. One of them will be based on an abstract argument
and the other will rely on a series of explicit transformations, thus yielding a bit more eventually.
On the other hand, the ways are not independent, and the second one uses an intuition provided
by the first, soft approach.
In these studies, Theorem \ref{elprod} will play a substantial role.

Recall from \eqref{stilt_def} that $\mathcal S_1= \mathbb R_+ + S_1(\mathcal M_1^+(\R_+)),$
and that $\mathcal S_1$  can be neatly characterized
by Theorem \ref{geometric}.
It will useful to note that if $f \in \mathcal S_1, f \sim (a,\mu)_1,$ then by the monotone convergence theorem,
\begin{equation}\label{limit_s}
\exists \lim_{s \to \infty} s f(s) =\begin{cases} \mu(\R_+), &\,\, \text{if} \,\,  a=0\,\,  \text{and}\,\,  \mu\,\, \text{is finite},\\
\infty,& \, \,\, \text{otherwise.}
\end{cases}
\end{equation}
The class $\mathcal S_1$ is closely related to the  class of so-called complete Bernstein functions, denoted by  $\mathcal{CBF}$.
One may define $f \in \mathcal{CBF} $ as $f:  (0,\infty)\to [0,\infty)$ such that $s^{-1}f \in \mathcal S_1,$
so that
\begin{equation}\label{cbf}
\mathcal {CBF}:=\{as+ s S_1[\mu](s): a \ge 0, \mu \in \mathcal M_1^+(\R_+)\}.
\end{equation}
Recall that  $f \in \mathcal{CBF}, f \ne  0,$ if and only if $1/f \in \mathcal S_1,$ see \cite[Theorem 7.3]{Sh}.

If $f \in \mathcal{CBF}$ or $f \in \mathcal{S}_1,$ then $f$ extends analytically to
$\C \setminus (-\infty,0],$ and we then identify $f$ with its analytic extension.
Moreover, if $f \in \mathcal{CBF},$ then $f$ extends continuously to $\C \setminus (-\infty,0).$

Our arguments will rely on specific properties of the subclass of $\mathcal{CBF}$ consisting of Thorin-Bernstein functions $\mathcal {TBF}.$ Recall that $\mathcal{TBF}$ can be described by the right-hand side of \eqref{cbf}
with $\mu$, being in addition,   absolutely continuous on $(0,\infty)$ such that $\mu(dt)=t^{-1}w(t)dt,$ and $w$ is non-decreasing on $(0,\infty),$
see \cite[Theorem 8.2, (v)]{Sh}.
The next two statements on Thorin-Bernstein functions  can be found in \cite[Theorem 8.2, (iii)]{Sh} and \cite[Proposition 8.7]{Sh}.
For short-hand,
if
 $f:(0,\infty) \to (0,\infty)$ is differentiable, then we let
\begin{equation}\label{phi}
\Psi[f](s):=-(\log f)'(s)=-\frac{f'(s)}{f(s)}, \qquad  s>0.
\end{equation}
Note that if  $\Psi[f](s)\ge 0$ for all $ s>0,$ and $\Psi[f] \in C(0,\infty),$ then $f' \in C(0,\infty)$
and $f'(s)\le 0$ for every $s>0,$ so that there exists (possibly infinite) $f(0+)=\lim_{s \to 0}f(s).$
\begin{theorem}\label{thorin}
\begin{itemize}
\item [(i)] If $g: (0,\infty) \to (0,\infty),$ then $g \in \mathcal {TBF}$ if and only if
$g$ is of the form
\begin{equation}\label{A1}
g(s)=a + bs+\int_{(0,\infty)} \log\left(1+\frac{s}{t}\right)\,\nu(dt), \qquad s >0,
\end{equation}
where $a,b \ge 0$ and $\nu$ is  a unique  positive Borel measure on $(0,\infty)$ satisfying
\begin{equation}\label{B1}
\int_{(0,1)} |\log t|\,\nu(dt)+\int_1^\infty \frac{\nu(dt)}{t}<\infty.
\end{equation}
\item [(ii)] If  $f: (0,\infty) \to (0,\infty)$ is differentiable, then $f=e^{-g}$ with $g \in \mathcal{TBF}$ if and only
$\Psi[f] \in \mathcal S_1$ and $f(0+)\le 1.$ 
\end{itemize}
\end{theorem}
See \cite[Chapter 8]{Sh} for more information on the properties of $\mathcal{TBF}$ and relevance of $\mathcal{TBF}$ in probability theory.

We start with  a general result describing a class of functions $f$ on $(0,\infty)$
such that for an appropriate $\gamma>0$ one has $f^\alpha \in S_{\gamma \alpha}$ for all $\alpha >0.$
 Its proof is based on an approximation trick involving the Krein-Milman theorem
and used already in the literature, see e.g. the proofs of \cite[Theorem 3.2]{Faraut} or
\cite[Theorem 2]{MacGr73}.
We will need  the following classical fact, separated for ease of reference.
\begin{lemma}\label{F}
Let $(f_n)_{n \ge 1}\subset \mathcal{S}_\alpha$, $\alpha>0$. If for every $s >0$
there exists $f(s):=\lim_{n\to\infty}\,f_n(s),$
then
$f\in \mathcal{S}_\alpha,$ and moreover $\lim_{n\to\infty}\,f^{(k)}_n(s)=f^{(k)}(s),$ $s>0,$
for every $k \in \mathbb N.$
\end{lemma}
The proof of $f \in \mathcal S_\alpha$ is analogous to the proof of
\cite[Theorem 2.2, (iii)]{Sh}, where $\alpha=1$,
and is therefore omitted. The last claim follows either from \cite[Corollary 1.7]{Sh}
or by adjusting the proof of \cite[Theorem 2.2, (iii)]{Sh} accordingly.
\begin{theorem}\label{1L}
Let $f:(0,\infty)\to (0,\infty)$ be differentiable.
Assume that
\begin{equation}\label{S}
\Psi[f] \in \mathcal S_1
\end{equation}
and
\begin{equation}\label{S1}
\gamma:=\lim_{s \to\infty}\,s\Psi[f](s)<\infty.
\end{equation}
Then for any $\alpha>0,$
\begin{equation}\label{beta}
f^\alpha\in \mathcal{S}_{\gamma \alpha}.
\end{equation}
Moreover, if $f^\alpha \sim (a, \mu)_{\gamma \alpha},$ then $\mu$ does not have singular continuous part.
\end{theorem}

\begin{proof}
We first prove \eqref{beta}. Let $\alpha>0$ and $\epsilon > 0$ be fixed, and
\[
f_\epsilon(s):=f(s+\epsilon), \qquad s >0.
\]
Then, by Theorem \ref{thorin},(ii),
there exists $g=g_\epsilon \in \mathcal{TBF}$ such that
\begin{equation}\label{T}
f_\epsilon(s)=f_\epsilon(0)e^{-g(s)},\qquad s \ge 0,
\end{equation}
and, in view of Theorem \ref{thorin},(i),
$g$ is of the form \eqref{A1} with
$a=a_\epsilon=0,$ $ b=b_\epsilon \ge 0,$ and $\nu=\nu_\epsilon$ being a positive Borel measure on $(0,\infty)$ satisfying \eqref{B1}.
So, $\Psi[f_\epsilon] \in \mathcal{S}_1$ and
\begin{equation}\label{D}
\Psi[f_\epsilon](s)=
g'(s)=b+\int_{(0,\infty)} \frac{\nu(dt)}{s+t},
\end{equation}
for every $s >0.$  Hence, by \eqref{limit_s}, $b=0$ and
\begin{equation}\label{Fatou}
 \int_{(0,\infty)} \nu(dt)= \gamma.
\end{equation}
We have
\begin{equation}\label{SM}
f_\epsilon(s)=
f_\epsilon (0) \exp\left(-\int_{(0,\infty)} \log\left(1+\frac{s}{t}\right)\,\nu(dt)\right),\qquad s>0.
\end{equation}
Fix
 $n\in \mathbb  N,$ let
\[
f_{\epsilon, n}(s):=\exp\left(-\int_{1/n}^n \log\left(1+\frac{s}{t}\right)\,\nu(dt)\right), \qquad s >0,
\]
and note that $\gamma >0$ by \eqref{limit_s}.
Using the Krein-Milman theorem (see e.g. \cite[Theorem 8.14]{Simon}) and \eqref{Fatou},
we can approximate the restriction of  $\gamma^{-1} \nu$ to $[1/n,n]$ by finite convex combinations of delta measures
in the weak$^*$ topology of $(C([1/n, n]))^*,$
so that
\[\int_{1/n}^n \log\left(1+\frac{s}{t}\right)\,\nu(dt)=\lim_{j \to \infty}\sum_{i=1}^{N_j} a_{ij} \log\left(1+\frac{s}{t_{ij}}\right), \qquad s >0,
\]
for some $t_{ij} \subset [1/n,n]$ and $a_{i,j} \ge 0$ with
$\sum_{i=1}^{N_j} a_{ij}=\gamma$ for all $j\in \mathbb N.$
Thus, for every $s >0,$
\[
(f_{\epsilon, n}(s))^\alpha= \lim_{j \to \infty}\prod_{i=1}^{N_j}  \left(1+\frac{s}{t_{ij}}\right)^{-\alpha a_{ij}}.
\]
Setting
\[
G_j(s)= \prod_{i=1}^{N_j} \left(1+\frac{s}{t_{ij}}\right)^{-\alpha a_{ij}}, \qquad j \in \N, \, s >0,
\]
and invoking  Theorem \ref{elprod} (or the less explicit result by Hirsmann-Widder \cite{widder}, mentioned in the introduction), we have $G_j \in  \mathcal{S}_{\gamma \alpha}$ for all $j \in \N.$ Hence  Lemma \ref{F} implies that $(f_{\epsilon, n})^\alpha \in \mathcal{S}_{\gamma \alpha}$ as well.

Now letting $n \to \infty$ and using the monotone convergence theorem, we infer that
\[
(f_{\epsilon}(s))^\alpha=\lim_{n\to\infty}\,(f_{\epsilon, n}(s))^\alpha, \qquad s >0,
\]
so by Lemma \ref{F} again,
\begin{equation}\label{FF}
(f_{\epsilon})^\alpha \in \mathcal{S}_{\gamma\alpha}.
\end{equation}
Finally, letting $\epsilon \to 0,$ we note that
\[
(f(s))^\alpha=\lim_{\epsilon\to 0}\,(f(s+\epsilon))^\alpha
=
\lim_{\epsilon\to 0}\, (f_{\epsilon}(s))^\alpha,
\]
for all $s >0.$ So   (\ref{FF}) and Lemma \ref{F} imply that
$
f^\alpha \in \mathcal{S}_{\gamma\alpha},
$
and the proof of \eqref{beta} is complete.

Let $f^\alpha \sim (a, \mu)_{\gamma \alpha}.$ To prove the last claim, note first that in view of  \eqref{beta} one clearly has
 $f^{\alpha/2}\in \mathcal{S}_{\alpha \gamma/2}.$
Applying the representation \eqref{sing_const}
to the product $f^\alpha$   of the generalized Stieltjes functions  $f^{\alpha/2}$ and $f^{\alpha/2}$
and using Corollary \ref{AbsC} and \eqref{general},  we conclude that  $\mu$  has no singular continuous part.
\end{proof}
\begin{remark}\label{log}
Observe that by Theorem \ref{1L}
if $f$ satisfies (\ref{S}) and
(\ref{S1}) with $\gamma=1,$ then
$f\in \mathcal{S}_1$. However, the opposite implication does not hold, i.e. neither \eqref{S} nor \eqref{S1} follow from $f\in \mathcal{S}_1$.
Indeed, considering first  \eqref{S},
 let
$
f(s):=(\log(s+1))^{-1}, s>0,
$
and observe that $f \in \mathcal{S}_1$ by, for example, Theorem \ref{geometric}
(or \cite[p.9]{Sh}).
Then
$$
\Psi[f](s)=((s+1)\log(s+1))^{-1},
$$
and $\Psi[f] \not \in \mathcal{S}_1$ by \eqref{limit_s}.

Turning now to \eqref{S1}, define
$$
f(s):=(s+1)^{-1}+(s+2)^{-1}, \qquad s>0.
$$
Then
$$
\Psi[f](s)=
(s+1)^{-1}+(s+2)^{-1}-2(2s+3)^{-1}
$$
and $\Psi[f]\not\in \mathcal{S}_1.$
One may just note  that $\Psi[f]=S_1[\mu]$ for a \emph{signed} measure  $\mu \in \mathcal M_1(\R_+),$ and refer to a uniqueness theorem for Stieltjes transforms.
\end{remark}

Thus, by Remark \ref{log}, the assumption $f \in \mathcal S_1$ is strictly weaker
than the assumptions in Theorem \ref{1L} with $\gamma=1$, and it is natural to ask whether
Theorem \ref{1L} holds if one only assumes that $f \in \mathcal S_1.$
More precisely, one can ask whether the following implication is true:
\begin{equation}\label{Q1}
f\in \mathcal{S}_1
\Longrightarrow
f^\alpha\in \mathcal{S}_\alpha,\quad \alpha\in (0,1).
\end{equation}
It is instructive to note
the  reformulation of \eqref{Q1}
via analytic extensions of functions from $\mathcal S_1$:
\begin{equation}\label{Q2}
g\in \mathcal{S}_1,\quad
g(\Sigma_\pi)\in \overline{\Sigma}_{\pi\alpha}
\Longrightarrow
g\in \mathcal{S}_\alpha,\quad \alpha\in (0,1),
\end{equation}
where ${\Sigma}_{\theta}=\{z\in \C: |{\rm arg}\, (z)|<\theta\}, \theta \in (0,\pi].$
To see that \eqref{Q1} and \eqref{Q2} are equivalent, let first \eqref{Q1} hold.  If $\alpha\in (0,1),$
$
g\in \mathcal{S}_1$ and
$
g(\Sigma_\pi)\in \overline{\Sigma}_{\pi\alpha},
$
then by Theorem \ref{geometric} we have
$g^{1/\alpha}\in\mathcal{S}_1,$
and in view of \eqref{Q1},
\[
g=(g^{1/\alpha})^\alpha\in \mathcal{S}_\alpha,
\]
so that \eqref{Q2} is true.
On the other hand, if \eqref{Q2} holds and
$f\in \mathcal{S}_1$, then using Theorem \ref{geometric} we infer that
\[
f^\alpha\in \mathcal{S}_1 \qquad \text{and} \qquad
f^\alpha (\Sigma_\pi)\in \overline{\Sigma}_{\pi\alpha},
\]
hence
$f^\alpha \in \mathcal{S}_\alpha$ by  (\ref{Q2}),
and \eqref{Q1} holds as well.

Using Theorem \ref{elprod}, we show that
\eqref{Q1}, and thus \eqref{Q2},
fail dramatically.
\begin{theorem}\label{SSS}
Let $f\in \mathcal{S}_1,$ $f\sim (a,\mu)_1,$
and let
$\mu$ satisfy one of the
following two conditions:
\begin{itemize}
\item [(i)] $\mu$ has
a nonzero continuous singular part;

\item [(ii)] $\mu$ is purely discrete and its set of atoms contains at least two points.
\end{itemize}
Then for each $n\ge 2$, $n\in \N,$ there exists
$\epsilon_n\in (0,1/n)$ such that
\[
f^{\alpha}\not\in \mathcal{S}_\alpha,\quad
|\alpha-1/n|<\epsilon_n.
\]
In particular,
$
f^{1/n}\not\in \mathcal{S}_{1/n}.
$
\end{theorem}
\begin{proof}
Fix $n \in \mathbb N, n \ge 2,$
and, arguing as in the proof of Theorem \ref{1L},
write $ f=(f^{1/n})^n.$
If $f$ satisfies either (i) or (ii), then
using either Corollary \ref{AbsC}, (iii) or Corollary \ref{AbsC}, (v), respectively, along with \eqref{general},
we conclude that $f^{1/n} \not \in \mathcal S_{1/n}.$

Next recall that by \cite[Theorem 2]{Sokal}, one has $f \in \mathcal{S}_\alpha$ if and only if
\[
(-1)^n \sum_{j=0}^n \binom{k}{j} \frac{\Gamma(n+k+\alpha)}{\Gamma(n+j +\alpha)}s^j f^{(n+j)}(s) \ge 0, \qquad s >0,
\]
for all integers $n, k \ge 0.$
Hence, if
$f_m\in \mathcal{S}_{\alpha_m}$, $ m \ge 1,$ and there exist
\[
\alpha=\lim_{m\to\infty}\,\alpha_m \qquad \text{and}\qquad
g(s)=\lim_{m\to\infty}\,f_m(s),\quad s>0,
\]
then, taking into account Lemma \ref{F}, we infer that
$
g\in \mathcal{S}_\alpha.
$
In other words, if $f\in \mathcal{S}_1,$ then the set
$
\{\alpha\in [0,1]:\,f^\alpha\in \mathcal{S}_\alpha\}
$
is closed.
Since
for $n \in \mathbb N, n \ge 2,$ we have $f^{1/n} \not \in \mathcal S_1,$
it then follows that there exists
$\epsilon_n\in (0,1/n)$ such that
$
f^{\alpha}\not\in \mathcal{S}_\alpha
$
for all $\alpha$ with $|\alpha-1/n|<\epsilon_n,$
as required.
\end{proof}

Theorem \ref{1L} allows us to construct a variety of
functions representable as the generalized Stieltjes transforms.
In particular, the function constructed below will be crucial in the proof
of Theorem \ref{Comp}, one of the main results of this section.
\begin{example}\label{Ex1}
For
 fixed $t\ge 0$ and
 $\beta\in (0,1],$ define
\begin{equation}\label{kbeta}
K_{\beta, t}(s):=\frac{1}{s^\beta+t},\qquad s>0.
\end{equation}
We have
\[
\Psi \left[K_{\beta, t}\right](s)=\frac{\beta s^{\beta-1}}{s^\beta+t},
\]
and  (\ref{S1}) holds with $\gamma=\beta$.
Moreover,
\[
\frac{1}{\Psi \left[K_{\beta, t}\right](s)}=\beta^{-1}(s+t s^{1-\beta})\in \mathcal{CBF},
\]
hence  $\Psi[K_{\beta, t}]\in \mathcal{S}_1$.
So, by Theorem \ref{1L}, for every $\alpha>0,$
\[
(K_{\beta,t})^\alpha
\in \mathcal{S}_{\alpha \beta},
\]
and, in particular, $K_{\beta, t}\in \mathcal{S}_\beta.$ If $K_{\beta, t} \sim (a,\nu)_\beta,$
then using
\[
\lim_{s\to\infty}\,s^\beta K_{\beta, t}(s)=1,
\]
and the monotone convergence theorem, we conclude that  $a=0$ and $\nu \in \mathcal M^+_0(\R_+)$ with
$\|\nu\|_{0}=1.$
\end{example}

For $\beta \in (0,1]$ and $f : (0,\infty)\to [0,\infty)$ let
\begin{equation}\label{defbeta}
f_{(\beta)}(s):=f(s^\beta), \qquad s>0.
\end{equation}

In view of Theorem \ref{geometric} the next result can be considered as a geometric condition describing a substantial subset of
$\mathcal S_\alpha,$ and thus filling partially a gap in the theory of generalized Stieltjes functions indicated by Sokal in \cite[p. 184-185]{Sokal}.
\begin{theorem}\label{Comp}
Let $f\in \mathcal{S}_\alpha$, $\alpha>0,$ and $\beta \in (0,1].$
Then
\[
f_{(\beta)}
\in \mathcal{S}_{\alpha\beta}.
\]
In particular, $f_{(\beta)}\in \mathcal{S}_\beta$ if $f\in \mathcal{S}_1$.
\end{theorem}
\begin{remark}\label{ban_rem}
Note for the sequel that similarly to \eqref{limit_s}  if $f \in \mathcal S_\beta$ and $\lim_{s \to \infty}s^\beta f(s)$ is finite, then by the monotone convergence theorem  $f=S_\beta[\mu]$ with $\mu \in \mathcal M^+_0(\R_+).$ 
Thus, Theorem \ref{Comp} implies that if $f \in \mathcal S_1$ is such that $\lim_{s \to \infty} s f(s) <\infty,$ then $f_{(\beta)}=S_\beta[\nu]$  with 
$\nu \in \mathcal M^+_0(\R_+).$ 
\end{remark}
\begin{proof}
By assumption, $f \sim (a,\nu)_\alpha,$
so
\[
f_{(\beta)}(s)=a+\int_0^\infty \frac{\nu(dt)}{(s^{\beta}+t)^\alpha}=a+\int_0^\infty (K_{\beta, t}(s))^\alpha {\nu(dt)}, \qquad s >0,
\]
where $a \ge 0$ and $\nu$
is a positive Borel measure on $[0,\infty)$ satisfying
\begin{equation}\label{al}
\int_0^\infty \frac{\nu(dt)}{(1+t)^\alpha}<\infty.
\end{equation}
For $n \in \mathbb N$ let
\[
f_{(\beta,n)}(s):=a+\int_0^n (K_{\beta, t}(s))^\alpha \nu(dt).
\]
By Example \ref{Ex1},  for every fixed $t > 0$ we have
$
(K_{\beta, t})^\alpha\in \mathcal{S}_{\alpha\beta}.
$
So using (\ref{al}) and the Krein-Milman theorem and arguing as in the proof of Theorem  \ref{1L},
we infer that  $f_{(\beta,n)} \in \mathcal{S}_{\alpha\beta}.$
Since by
the monotone convergence theorem,
\[
f_{(\beta)}(s)=\lim_{n\to\infty}\,f_{(\beta,n)}(s), \qquad s >0,
\]
we conclude that $f_{(\beta)}\in  \mathcal{S}_{\alpha\beta}$, and the statement follows.
\end{proof}
The next corollary of Theorem \ref{Comp} concerning complex measures is direct.
\begin{corollary}\label{Ban}
Let $f = a+S_1[\mu],$ where $\mu \in \mathcal M_0(\R_+), a \in \mathbb C.$ Then there exists $\nu \in \mathcal M_{0}(\R_+)$ such that
$f_{(1/2)}=a+ S_{1/2}[\nu]$
\end{corollary}
To prove the statement it  suffices to write $\mu=\mu^+_1-\mu_1^{-}+i(\mu^+_2-\mu^{-}_2)$ with
 $\mu^{\pm}_{j} \in \mathcal M^+_0(\R_+), j=1,2,$
and to apply Theorem \ref{Comp} to each of the four transforms $S_1[\mu^{\pm}_{j}], j=1,2.$

Corollary \ref{Ban} (along with Remark \ref{ban_rem}) can be used to identify spectral multipliers $\Phi$
relevant for the main results in \cite{Banuelos}.
For more on that see Remark \ref{beta1} and Example \ref{Elog}.
\begin{remark}
Theorem \ref{Comp} is essentially sharp.
If e.g. $f(s)=s^{-1},s>0,$ then $f_{(\beta)}(s)=s^{-1/\beta} \in \mathcal S_{\beta},$ and $f_{(\beta)} \not \in \mathcal S_\alpha$ for any $\alpha <\beta.$
(For the latter claim, see e.g. \cite[Corollary 2, p. 63]{Karp}.)

Next,
consider
\begin{equation}\label{examf}
f(s)=2\int_1^2\frac{dt}{(t+s)^3}+\int_1^2\frac{dt}{(t+s)^3}, \qquad s >0,
\end{equation}
studied already in \cite[Remark 10]{Karp}. Since $f \in \mathcal S_3,$  Theorem \ref{Comp} implies that
$
f_{(1/2)}\in \mathcal{S}_{3/2}.
$
In fact, more is true.
By a direct computation (see again \cite[Example 10]{Karp}), the analytic extension of $f$ to $\mathbb C^+,$
denoted by the same symbol, satisfies
\[
{\rm Im}\,f(z)\in \mathbb C \setminus \mathbb C^+, \qquad z\in \mathbb C^+.
\]
From this and Theorem \ref{geometric},
we infer that
$
f_{(1/2)}\in \mathcal{S}_1\subset \mathcal S_{3/2}.
$
On the other hand,
\[
f(s)=\frac{1}{(s+1)^2}-\frac{1}{2(s+2)^2},
\]
so $f \not \in \mathcal S_2$ since the representing measure for $f$ is signed.

Thus, in general, for $\beta>0$ and $\alpha\in (0,1),$ the implication
\[
f\in \mathcal{S}_\beta,\;
f_{(\alpha)}\in \mathcal{S}_\gamma,\;\, \gamma<\beta\alpha\Longrightarrow
f\in \mathcal{S}_{\gamma/\alpha},
\]
is not valid. (For $f$ given by \eqref{examf} we have
$\beta=3$, $\alpha=1/2$, $\gamma=1,$
and $\gamma/\alpha=2$.)
\end{remark}

Let us illustrate Theorems \ref{1L} and \ref{Comp} by the following example.
\begin{example}\label{Ex2}
Let
\[
f(s):=\frac{\log s}{s-1}, \qquad s >0.
\]
Then
\[
\Psi[f](s)=\frac{1}{s-1}-\frac{1}{s\log s},
\]
for all $s>0,$ so the assumption  (\ref{S1}) holds for $\gamma=1$.
Moreover, we have
\[
\Psi[f](s)=\frac{q(1)-q(s)}{s-1},\qquad q(1)=1,
\]
with $q(s)=(sf(s))^{-1}=(s-1)(s\log s)^{-1}.$ Since $f \in \mathcal S_1$ by \cite[Example 2.9]{GHT}, we also have
$1/f \in \mathcal {CBF},$ and then $q \in \mathcal{S}_1.$
Hence  $\Psi[f]\in \mathcal{S}_1$  by \cite[Theorem 7.17]{Sh} (see also \cite[Theorem 2]{Nak}).
So,   Theorem \ref{1L} implies
 that for any $\alpha>0,$
\[
\left(\frac{\log s}{s-1}\right)^\alpha\in \mathcal{S}_\alpha.
\]
Moreover,
by Theorem \ref{Comp},  for  all $\beta\in (0,1]$
and $\alpha>0,$
\begin{equation}\label{examp}
\left(\frac{\log s}{s^\beta-1}\right)^\alpha\in \mathcal{S}_{\beta \alpha }.
\end{equation}

Observe that since
\[
\lim_{s\to\infty}\,f(s)=0
\qquad  \text{and} \qquad
\lim_{s\to\infty}\,s f(s)=\infty,
\]
one has $f\sim (0, \nu)_1,$ but
the measure $\nu$ is not finite on $[0,\infty)$ in contrast to the
situation considered in Example \ref{Ex1}.
\end{example}

Now we turn to our second approach to the Stieltjes representability problem leading to an explicit representation of $f_{(\beta)}$
and being alternative to Theorem \ref{Comp}.
However, this approach uses an intuition provided by Theorem \ref{Comp}, and it is independent of Theorem \ref{Comp} under a thin disguise.
 If $f \in \mathcal S_1, f \sim (a, \mu)_1,$ then by Theorem \ref{Comp} we have $f_{(\beta)} \in S_\beta, f_{(\beta)} \sim (a,\nu)_\beta,$ for every $\beta \in (0,1],$
and it is natural to try to determine $\nu$ in terms of $\mu$.
This will be the content of Theorem \ref{CorM} below.

We will need the following statement, which is a corollary of \cite[Theorem 3]{Karp},
see also \cite[Lemma 2.1]{Samko}.

\begin{lemma}\label{L1K}
Let $g \in \mathcal S_1$ be such that
\begin{equation}\label{2I}
g(s)=\int_0^\infty \frac{v(t)\,dt}{s+t},\qquad s>0,
\end{equation}
for some positive $v \in L^1_1(\R_+),$ absolutely continuous on $[0,a]$ for every $a>0.$
If  moreover $g\in \mathcal{S}_\beta, \beta \in (0,1),$ then $v_\beta$  given by
\begin{equation}\label{AI}
v_\beta(t)=\frac{d}{dt}\int_0^t\frac{v(r)\,dr}{(t-r)^{1-\beta}},\qquad \mbox{a.e}\;\;t>0,
\end{equation}
 is well-defined, $v_\beta \in L^1_\beta(\R_+), v_\beta \ge 0,$
and
\begin{equation}\label{E1}
g(s)=\int_{0}^\infty\frac{v_\beta(t)\,dt}{(s+t)^\beta},\qquad s>0.
\end{equation}
\end{lemma}

\begin{lemma}\label{Mbeta}
For $\beta \in (0,1)$ let $K_{\beta}:=K_{\beta, 1}$ be given by \eqref{kbeta}.
Then $K_{\beta} \in  \mathcal{S}_\beta,$
and
\begin{equation}\label{RRa}
K_\beta(s)=\int_0^\infty \frac{\varphi_\beta(t)\,dt}{(s+t)^\beta}, \qquad s >0,
\end{equation}
where a positive $\varphi_\beta \in L^1_\beta(\R_+)$ is defined for all $t>0$ by
\begin{equation}\label{RR}
\varphi_\beta(t)
:= \frac{2\beta\sin(\pi\beta)}{\pi}
\int_0^1\frac
{t^{2\beta-1}(1+\cos(\pi\beta)t^\beta r^\beta)r^\beta\,dr}{(t^{2\beta}r^{2\beta}+2\cos(\pi\beta)t^\beta r^\beta+1)^2
(1-s)^{1-\beta}}.
\end{equation}
\end{lemma}

\begin{proof}
Let $\beta \in (0,1)$ be fixed.
First, we infer by Example \ref{Ex1} that $K_{\beta} \in \mathcal S_\beta,$ so
$K_{\beta}$ admits the representation \eqref{RRa} with
a positive $\varphi_\beta \in L^1_\beta(\R_+),$ and
we have  only to determine $\varphi_\beta$ explicitly.
To this aim, recall that $K_{\beta} \in \mathcal S_1,$ and
\begin{equation}\label{Cl}
K_{\beta}(s)=\int_0^\infty \frac{\psi_\beta(r)\,dr}{s+r},\qquad s>0,
\end{equation}
where
\[
\psi_\beta(r):= 
\frac{\sin(\pi\beta) r^\beta}
{\pi(r^{2\beta}+2\cos(\pi\beta)r^\beta+1)},\qquad r>0,
\]
and  $\psi_\beta \in L^1_1(\R_+).$ See e.g. \cite[Theorem 2]{Kato}, where
\eqref{Cl} was deduced in a more general operator context, or \cite[Example 4.1.3]{Mart1}.
Using now Lemma \ref{L1K}, for every $t >0,$ we have
\begin{align*}
\varphi_\beta(t)=&
\frac{\sin(\pi\beta)}{\pi}
\frac{d}{dt}\int_0^t\frac
{r^\beta\,dr}{(r^{2\beta}+2\cos(\pi\beta)r^\beta+1)
(t-r)^{1-\beta}}\\
=&
\frac{\sin(\pi\beta)}{\pi}
\frac{d}{dt}\int_0^1\frac
{t^{2\beta}r^\beta\,dr}{(t^{2\beta}r^{2\beta}+2\cos(\pi\beta)t^\beta r^\beta+1)
(1-r)^{1-\beta}},
\end{align*}
and
\begin{align*}
(&t^{2\beta}r^{2\beta}+2\cos(\pi\beta)t^\beta r^\beta+1)^2
\frac{d}{dt}\left(\frac
{t^{2\beta}}{(t^{2\beta}r^{2\beta}+2\cos(\pi\beta)t^\beta r^\beta+1)}\right)\\
=&2\beta t^{2\beta-1}(1+\cos(\pi\beta)t^\beta r^\beta).
\end{align*}
Thus
\[
\varphi_\beta(t)=\frac{2\beta\sin(\pi\beta)}{\pi}
\int_0^1\frac
{t^{2\beta-1}(1+\cos(\pi\beta)t^\beta r^\beta)s^\beta\,dr}{(t^{2\beta}r^{2\beta}+2\cos(\pi\beta)t^\beta r^\beta+1)^2
(1-r)^{1-\beta}}, \quad t >0,
\]
so
(\ref{RRa}) and  (\ref{RR}) hold.
\end{proof}

It is crucial to note that to write down \eqref{RRa} we used the fact that $K_\beta \in \mathcal S_\beta$
proved in Example \ref{Ex1}. Thus, our argument is not independent of Example \ref{Ex1}.
However, once we have guessed \eqref{RRa}, we can verify \eqref{RRa} directly and do not need to invoke
any additional results.
At the same time, we are not able to directly check the \emph{positivity} of $\varphi_\beta$ for $\beta \in (1/2,1)$, and this part of Lemma \ref{L1K} depends on Example \ref{Ex1}, and so on Theorem \ref{1L}.
See also Remark \ref{complex} below.

While  the formulas \eqref{RRa} and \eqref{RR}  for $K_\beta$ are somewhat cumbersome, the case $\beta=1/2$ is exceptional as the next example shows.
\begin{example}\label{1/2A}
If $\beta=1/2,$ then  by (\ref{RRa}), \eqref{RR}
and \cite[p. 302, no.9]{Prudnikov1}, we have
\[
\varphi_{1/2}(t)
=\frac{1}{\pi t}\int_0^t
\frac{r^{1/2}\,dr}{(r+1)^2(t-r)^{1/2}}
=\frac{1}{2(t+1)^{3/2}}
\]
for all $t>0,$ hence
\[
\frac{1}{s^{1/2}+1}=\frac{1}{2}\int_0^\infty\frac{dt}{(s+t)^{1/2}(t+1)^{3/2}},\qquad s>0.
\]
\end{example}

Now, given $f \in \mathcal S_1,$ we provide an explicit Stieltjes representation for $f_{(\beta)}.$

\begin{theorem}\label{CorM}
Let $f \in \mathcal S_1, f \sim (a,\mu)_1,$
and let $f_{(\beta)}$ be defined  by \eqref{defbeta}.
Then $f_{(\beta)} \in \mathcal S_\beta$ for every $\beta\in (0,1),$
and
\begin{equation}\label{psi1}
f_{(\beta)}(s)=a+\mu(\{0\})s^{-\beta}+\int_0^\infty
\frac{\psi_\beta(r)\,dr}{(s+r)^\beta},\qquad s>0,
\end{equation}
where $\psi_\beta \in L^1_\beta(\R_+)$ is given by
\begin{equation}\label{psi}
\psi_\beta(r):=
\int_{(0,\infty)}\varphi_\beta \left(\frac{r}{\tau^{1/\beta}}\right)\,
\frac{\mu(d\tau)}{\tau^{1/\beta}},\qquad r>0,
\end{equation}
and $\varphi_\beta$ is defined
by   \eqref{RR}.
\end{theorem}
\begin{proof}
It suffices to consider only the case when $f\sim (0,\mu)_1$ and $\mu(\{0\})=0.$

By Lemma \ref{Mbeta}, for every  $\tau>0,$ we have
\begin{equation}\label{GS}
\frac{1}{s^\beta+\tau}=\frac{1}{\tau((s/\tau^{1/\beta})^\beta +1)}
=\int_0^\infty\frac{\varphi_\beta(t)\,dt}{(s+\tau^{1/\beta}t)^\beta},
\end{equation}
where $\varphi_\beta$ is given by \eqref{RRa}.
Then, taking into account the positivity of integrands and using Fubini's theorem,
\begin{align*}
f_{(\beta)}(s)=&\int_0^\infty \frac{\mu(d\tau)}{s^\beta+\tau}=
\int_0^\infty
\left(
\int_0^\infty\frac{\varphi_\beta(t)\,dt}{(s+\tau^{1/\beta}t)^\beta}\right)\,
\mu(d\tau)\\
=&\int_0^\infty
\left(
\int_0^\infty\frac{\varphi_\beta(r/\tau^{1/\beta})\,dr}{(s+r)^\beta}\right)\,
\frac{\mu(d\tau)}{\tau^{1/\beta}}\\
=&\int_0^\infty\left(
\int_0^\infty \varphi_\beta \left(\frac{r}{\tau^{1/\beta}}\right)\,
\frac{\mu(d\tau)}{\tau^{1/\beta}}\right)
\frac{dr}{(s+r)^\beta}, \qquad s >0,
\end{align*}
which is equivalent to (\ref{psi1}) and (\ref{psi}).
\end{proof}

Despite  a fully expanded formula \eqref{psi1} looks quite "heavy", it can be used successfully in particular
situations, and the following example illustrates this point.

\begin{remark}\label{beta1}
Under the assumptions of Theorem  \ref{CorM},
if  $\beta=1/2,$ then
\[
\varphi_{1/2}(t)=\frac{1}{2(t+1)^{3/2}}, \qquad t >0,
\]
so that the corresponding density $\psi_{1/2}$ in the Stieltjes representation of $f_{1/2}$ is given by
\begin{equation*}
\psi_{1/2}(r)=
\int_0^\infty\varphi_{1/2}\left(\frac{r}{\tau^2} \right)\,
\frac{\mu(d\tau)}{\tau^2}=\frac{1}{2}\int_0^\infty
\frac{\tau\,\mu(d\tau)}{(r+\tau^2)^{3/2}}, \qquad r>0,
\end{equation*}
cf. Example \ref{1/2A}.
Thus, if $\beta=1/2,$ then Theorem \ref{CorM} takes a transparent form
and, being combined with Remark \ref{ban_rem}, yields explicit formulas for multipliers $\Phi$ in \cite[Theorems 1.1-1.3]{Banuelos} given there in the form $S_{1/2}[\nu]$ with $\nu \in \mathcal M_0(\R_+)$
(or an equivalent form).
\end{remark}

The next example illustrates this a bit further.
\begin{example}\label{Elog}
Let
\[
f(s)=\int_0^1\frac{dt}{s+t}=\log(1+1/s), \qquad s >0.
\]
Then $f \in \mathcal S_1$, and by (\ref{psi}) we have
\[
f_{(1/2)}(s)=\log(1+1/\sqrt{s})=
\int_0^\infty\frac{\psi_{1/2}(r)\,dr}{(s+r)^{1/2}}, \qquad s >0,
\]
where
\begin{align*}
\psi_{1/2}(r)=&
\int_0^\infty\varphi_{1/2}\left(\frac{r}{\tau^2} \right)\,
\frac{\mu(d\tau)}{\tau^2}
=\frac{1}{2}\int_0^1
\frac{d\tau}{(1+r/\tau^2)^{3/2}\tau^2}\\
=&\frac{1}{4}\int_0^1
\frac{dt}{(t+r)^{3/2}}=\frac{1}{2}\left(\frac{1}{\sqrt{r}}-\frac{1}{\sqrt{1+r}}\right).
\end{align*}

So,
\[
\log(1+1/\sqrt{s})=
\frac{1}{2}\int_0^\infty
\left(\frac{1}{\sqrt{r}}-\frac{1}{\sqrt{1+r}}\right)
\frac{dr}{(s+r)^{1/2}}, \qquad s >0.
\]

Alternatively, using Theorem \ref{geometric}, it is easy to verify that $f \in \mathcal S_1.$
Then $f_{(1/2)} \in \mathcal S_{1/2}$ by Theorem \ref{CorM}, and its representing measure is finite
by Remark \ref{ban_rem}. This kind of arguments can be used to produce spectral multipliers $\Phi$ in
 \cite{Banuelos} starting from a given function from $\mathcal S_1$ rather than from its representation as a generalized  Stieltjes transform.
\end{example}

\begin{remark}\label{complex}
Once Theorem \ref{CorM} is obtained, one may replace the assumption $f \in \mathcal S_1, f \sim (0,\mu)_1,$
by the assumption that $f=S_1[\mu]$ for, in general,  a \emph{complex} Radon measure $\mu \in \mathcal M_1(\R_+).$
A direct verification using Fubini's theorem shows that the formulas \eqref{psi1} and \eqref{psi} remain valid.
However, in this case, the formulas appear as a black box, without a natural explanation on how they could be obtained. Moreover, Theorem \ref{CorM} relies on Theorem \ref{Comp} as far as positivity of the representing measures is concerned.
\end{remark}

\section{Product formulas for generalized Cauchy transforms and related matters}\label{cauchy_sec}
\subsection{General product formulas}

In this section, being motivated by the property \eqref{prod_cauchy_in}, we investigate product formulas in the setting of the generalized Cauchy transforms on $\R.$
Recall that for $\mu \in \mathcal M_\alpha(\R)$ its generalized Cauchy transform  of order $\alpha>0$ is defined as
\begin{equation}\label{ReprCC}
C_\alpha[\mu](z):=\int_{\R}\frac{\mu(dt)}{(z+t)^\alpha},\qquad z\in \C^{+},
\end{equation}
where $\C^+$ stands for the upper half-plane $\{z\in \C: {\rm Im}\, z >0\}.$

Given $\mu_j\in \mathcal{M}_{\alpha_j}(\R), \alpha_j >0, j=1,2,$ we use our ideas
from the preceding sections and obtain
product formulas of the form
\begin{equation}\label{cauchy_id}
C_{\alpha_1}[\mu_1](z) C_{\alpha_2}[\mu_2](z) =C_{\alpha_1+\alpha_2}[\mu](z), \qquad z \in \C^+,
\end{equation}
 with an emphasis on the case when
$\mu_j  \in  {\mathcal M}^+_{\alpha_j}(\R), j=1,2.$
It is essential to observe that given the right-hand side of \eqref{ReprCC} its  representation as a generalized Cauchy transform is not unique.
For example, we have
\[
\int_{\R}\frac{t^k\,dt}{(z+t)^{\beta}}\equiv 0,\qquad z\in \C^{+},\quad
k \in \mathbb N,\quad \beta>k+1.
\]
This fact to leads to non-uniqueness of $\mu$ in \eqref{cauchy_id}.
In Theorems \ref{elprodC} and \ref{compact}, we provide explicit, comparatively simple  formulas for $\mu$ resembling \eqref{m1} and \eqref{m0} and having similar useful features.
These formulas
are obtained  under additional assumptions on the size of $\mu_1$ and $\mu_2,$ which
include several situations of  interest.
(As far as $\mu$ is not unique, here we avoid using
the convolution notation from the preceding sections.)
At the same time, we
produce a $\mu \in \mathcal M_{\alpha_1+\alpha_2}(\R)$ satisfying
\eqref{ReprCC}  without any additional restrictions on $\mu_1 \in \mathcal M_{\alpha_1}(\mathbb R)$ and $\mu_2 \in \mathcal M_{\alpha_2} (\R).$
However, in this case, our formula for $\mu$ appears to be rather complicated.
In such a general framework, it is not always possible to choose a positive $\mu$ in \eqref{cauchy_id}
even if $\mu_j \in \mathcal M^+_{\alpha_j}(\R), j=1,2,$
and we explore this issue thoroughly.
For  $\alpha_1=\alpha_2=1$, we characterize the existence of such a $\mu$ by a simple integrability condition
on $\mu_1$ and $\mu_2.$ Moreover, we prove that the choice of positive $\mu$
is possible if  $\alpha_1+\alpha_2=1.$

To formulate our first result, we define the functional $J_{\alpha_1,\alpha_2}: \mathcal M_{\alpha_1}(\R)\times \mathcal M_{\alpha_2} (\R)\to \R_+ \cup \{\infty\}$
by
\[
J_{\alpha_1,\alpha_2}[\mu_1,\mu_2]:=\int_{\R}\int_{(-\infty,t)}
\int_{(s,t)}\frac{(\tau-s)^{\alpha_2-1}(t-\tau)^{\alpha_1-1}\,d\tau }{(1+|\tau|)^{\alpha_1+\alpha_2}}\,
\frac{|\mu_1| (ds)\, |\mu_2| dt)}{(t-s)^{\alpha_1+\alpha_2-1}}.
\]
\begin{theorem}\label{elprodC}
Let $\mu_j \in  {\mathcal M}_{\alpha_j} (\R), \alpha_j>0, j=1,2,$
and  suppose that
\begin{equation}\label{C1H}
J_{\alpha_1,\alpha_2}[\mu_1,\mu_2]<\infty \qquad \text{and} \qquad
J_{\alpha_2,\alpha_1}[\mu_2,\mu_1]<\infty.
\end{equation}
Then \eqref{cauchy_id} holds
with $\mu \in  {\mathcal M}_{\alpha_1+\alpha_2}(\R)$  given by
\begin{equation}\label{m1C}
\mu(d\tau):=u(\tau)d\tau+\mu_1(\{\tau\})\mu_2(d\tau)
\end{equation}
and
\begin{align}\label{m0C}
B(\alpha_1, \alpha_2) u(\tau):=&\int_{(\tau,\infty)}\int_{(-\infty,\tau)}
 \frac{(\tau-s)^{\alpha_2-1}(t-\tau)^{\alpha_1-1}}{(t-s)^{\alpha_1+\alpha_2-1}}
\mu_1(ds)\,\mu_2(dt)\\
+&\int_{(\tau,\infty)} \int_{(-\infty,\tau)}
\frac{(\tau-s)^{\alpha_1-1}(t-\tau)^{\alpha_2-1}}{(t-s)^{\alpha_1+\alpha_2-1}}
\mu_2(ds)\,\mu_1(dt)\notag
\end{align}
\end{theorem}
for almost all $\tau \in \mathbb R.$ Moreover, such a $\mu$ is posiive if $\mu_1$ and $\mu_2$ are positive.
\begin{proof}
The proof is similar to the proof
of Theorem \ref{elprod}.
Let  $z\in \C^+$ be fixed. Write
\[
C_{\alpha_1}[\mu_1](z) C_{\alpha_2}[\mu_2](z)=\int_{\R}\int_{\R}
\frac{\mu_1(ds)\, \mu_2(dt)}{(z+s)^{\alpha_1}
(z+t)^{\alpha_2}}:=g_1(z)+g_2(z)+g_0(z),
\]
where
\begin{align*}
g_1(z):=&\int_{\R}\int_{(-\infty,t)}
\frac{\mu_1(ds)\, \mu_2(dt)}{(z+s)^{\alpha_1}
(z+t)^{\alpha_2}},\\
g_2(z):=&\int_{\R}\int_{(t,\infty)}
\frac{\mu_1(ds)\, \mu_2(dt)}{(z+s)^{\alpha_1}
(z+t)^{\alpha_2}}=
\int_{\R}\int_{(-\infty,t)}
\frac{\mu_2(ds)\, \mu_1(dt)}{(z+s)^{\alpha_2}
(z+t)^{\alpha_1}},\\
g_0(z):=&\int_{\R}
\frac{\mu_1(\{t\})\, \mu_2(dt)}{(z+t)^{\alpha_1+\alpha_2}}=
\int_{\R}\frac{\mu_2(\{t\})\, \mu_1(dt)}{(z+t)^{\alpha_1+\alpha_2}}.
\end{align*}

By (\ref{A2}),
\[
B(\alpha_1,\alpha_2) g_1(z)=
\int_{\R}\int_{(-\infty,t)}
\int_{(s,t)}\frac{(\tau-s)^{\alpha_2-1}(t-\tau)^{\alpha_1-1}\,d\tau }{(z+\tau)^{\alpha_1+\alpha_2}}\,
\frac{\mu_1(ds)\, \mu_2(dt)}{(t-s)^{\alpha_1+\alpha_2-1}}.
\]
Then, using  (\ref{C1H}),
the elementary inequality
\begin{equation}\label{inneq}
|z+\tau|\ge c(1+|\tau|),
\qquad \tau\in \R,
\end{equation}
for some $c=c(z) >0,$
and Fubini's theorem,  we infer that
\begin{align*}
&B(\alpha_1, \alpha_2) g_1(z)\\
=&
\int_{\R}\int_{(-\infty,t)}
\int_{(-\infty,\tau)}\frac{(\tau-s)^{\alpha_2-1}(t-\tau)^{\alpha_1-1} }
{(t-s)^{{\alpha_1+\alpha_2}-1}(z+\tau)^{\alpha_1+\alpha_2}}\,
\mu_1(ds)\,d\tau \,\mu_2(dt)\\
=&
\int_{\R}\int_{(\tau,\infty)}
\int_{(-\infty,\tau)}\frac{(\tau-s)^{\alpha_2-1}(t-\tau)^{\alpha_1-1} }{(t-s)^{\alpha_1+\alpha_2-1}}\,
\frac{\mu_1(ds)\,\mu_2(dt)\,d\tau}{(z+\tau)^{\alpha_1+\alpha_2}}\\
=&\int_{\R}\frac{u_1(\tau)\, d\tau}{(z+\tau)^{\alpha_1+\alpha_2}},
\end{align*}
where $u_1 \in L^1_{\alpha_1+\alpha_2}(\R).$
Similarly,
\begin{align*}
B(\alpha_1, \alpha_2) g_2(z)
=&
\int_{\R}
\int_{(\tau,\infty)} \int_{(-\infty,\tau)}
\frac{(\tau-s)^{\alpha_1-1}(t-\tau)^{\alpha_2-1}}{(t-s)^{\alpha_1+\alpha_2-1}}
\frac{\mu_2(ds)\, \mu_1(dt)\,d\tau}{(z+\tau)^{\alpha_1+\alpha_2}}\\
=&\int_{\R}\frac{u_2(\tau)\, d\tau}{(z+\tau)^{\alpha_1+\alpha_2}},
\end{align*}
with  $u_2 \in L^1_{\alpha_1+\alpha_2}(\R).$
Hence \eqref{cauchy_id}  holds with $\mu$ given by  (\ref{m1C}) and (\ref{m0C}), where $u=u_1+u_2$
belongs to $L^1_{\alpha_1+\alpha_2}(\R).$ Moreover, it is direct to check that for $\mu_d(dt)=\mu_1(\{t\})\mu_2(dt)$
one has $\mu_d \in \mathcal M_{\alpha_1+\alpha_2}(\R)$ and $\|\mu_d\|_{\alpha_1+\alpha_2}\le \|\mu_1\|_{\alpha_1} \|\mu_2\|_{\alpha_2}.$
The claim on positivity of $\mu$ for $\mu_j \in \mathcal M^+_{\alpha_j}(\R), j=1,2,$
is straightforward.
\end{proof}
\begin{remark}\label{rem_good}
Note that, as in the situation of Corollary \ref{AbsC},  $\mu$ does not have singular continuous component.
In general, under the assumptions of Theorem \ref{elprodC}, one may formulate  an analogue of Corollary \ref{AbsC}, but we omit easy details.
\end{remark}

Now we consider several natural cases where one can get rid of the assumption \eqref{C1H}.
In particular, we include Theorem \ref{elprod} into a more general context of Cauchy transforms,
see Corollary \ref{CorC11}, b) below.
\begin{corollary}\label{CorC11}
Let $\mu_j \in {\mathcal M}_{\alpha_j}(\R),\alpha_j>0, j=1,2,$ be such that
one of the following conditions holds:
\begin{itemize}
\item [a)]  One has $\mu_j \in  {\mathcal M}_0(\mathbb R), j=1,2;$
\item [b)] There exists $a\in \R$ such that
$\supp\,\mu_j\subset [a,\infty), j=1,2;$
\item [c)] There exists $a\in \R$ such that
$\supp\,\mu_j\subset (-\infty,a], j=1,2;$
\item
[d)]\ One of the measures $\mu_1$ or $\mu_2$
has compact support.
\end{itemize}
If  $\mu$ is defined by  \eqref{m1C} and \eqref{m0C},
then \eqref{cauchy_id} holds.
\end{corollary}

\begin{proof}
To prove the assertion,  we  verify the assumption (\ref{C1H}) of Theorem \ref{elprodC} in each of the cases a)-d).

a) Using  (\ref{CD2}), we obtain
\begin{align*}
&B(\alpha_1,\alpha_2) J_{\alpha_1,\alpha_2}[\mu_1,\mu_2]\\
\le& B(\alpha_1,\alpha_2)\int_{\R}
\int_{(-\infty,t)}
\int_{(s,t)}(\tau-s)^{\alpha_2-1}(t-\tau)^{\alpha_1-1}\,d\tau \,
\frac{|\mu_1|(ds)|\mu_2|(dt)}{(t-s)^{{\alpha_1+\alpha_2}-1}}\\
=&
 \int_{\R}\int_{(-\infty,t)}|\mu_1|(ds)|\mu_2|(dt)\le
\|\mu_1\|_{ 0} \|\mu_2\|_{0}.
\end{align*}
Similarly,
\[
B(\alpha_1, \alpha_2) J_{\alpha_2, \alpha_1}[\mu_2,\mu_1]\le
\|\mu_1\|_{ 0} \|\mu_2\|_{0}.
\]

b) By considering appropriate shifts of measures $\mu_1$ and $\mu_2,$ we may assume that $a=0.$
Using (\ref{A2}), we infer that
\begin{align*}
&B(\alpha_1,\alpha_2) J_{\alpha_1,\alpha_2}[\mu_1,\mu_2]\\
\le& B(\alpha_1,\alpha_2)
\int_{0}^\infty\int_{[0,t)}
\int_{(s,t)}\frac{(\tau-s)^{\alpha_2-1}(t-\tau)^{\alpha_1-1}\,d\tau }{(1+\tau)^{\alpha_1+\alpha_2}}\,
\frac{|\mu_1|(ds)|\mu_2|(dt)}{(t-s)^{\alpha_1+\alpha_2-1}}\\
=&
\int_{0}^\infty\int_{[0,t)}
\,\frac{|\mu_1|(ds)|\mu_2|(dt)}{(1+s)^{\alpha_1}
(1+t)^{\alpha_2}}
\le
\|\mu_1\|_{\alpha_1}\|\mu_2\|_{\alpha_2}<\infty.
\end{align*}
By symmetry, we also have
\[
B(\alpha_1,\alpha_2) J_{\alpha_2,\alpha_1}[\mu_2,\mu_1]\le
\|\mu_1\|_{\alpha_1}
\|\mu_2\|_{\alpha_2}<\infty.
\]

c) We proceed along the lines of the proof of b). Assuming $a=0,$
from (\ref{A2}) it follows that
\begin{align*}
&\frac{1}{(t-s)^{\alpha_1+\alpha_2-1}}
\int_s^t\frac{(\tau-s)^{\alpha_2-1}(t-\tau)^{\alpha_1-1}\,d\tau}
{(1-\tau)^{\alpha_1+\alpha_2}}\\
=&\frac{1}{(t-s)^{\alpha_1+\alpha_2-1}}
\int_{-t}^{-s}\frac{(-s-\tau)^{\alpha_2-1}(\tau+t)^{\alpha_1-1}\,d\tau}
{(1+\tau)^{\alpha_1+\alpha_2}}\\
=&\frac{B(\alpha_1,\alpha_2)}{(1-s)^{\alpha_1} (1-t)^{\alpha_2}},\qquad
s<t<0.
\end{align*}
Then
\begin{align*}
&B(\alpha_1,\alpha_2)J_{\alpha_1,\alpha_2}[\mu_1,\mu_2]\\
&\le
B(\alpha_1,\alpha_2)\int_{(-\infty,0]}\int_{(-\infty,t)}
\int_{(s,t)}\frac{(\tau-s)^{\alpha_2-1}(t-\tau)^{\alpha_1}\,d\tau }{(1-\tau)^{\alpha_1+\alpha_2}}\,
\frac{|\mu_1|(ds)|\mu_2|(dt)}{(t-s)^{\alpha_1+\alpha_2-1}}\\
=&
\int_{(-\infty,0]}\int_{(-\infty,t)}
\frac{|\mu_1|(ds)|\mu_2|(dt)}{(1-s)^{\alpha_1} (1-t)^{\alpha_2}}
\le
\|\mu_1\|_{\alpha_1}
\|\mu_2\|_{\alpha_2}<\infty.
\end{align*}
By symmetry, as above,
\[
B(\alpha_1,\alpha_2) J_{\alpha_2,\alpha_1}[\mu_2,\mu_1]
\le
\|\mu_1\|_{\alpha_1}
\|\mu_2\|_{\alpha_2}<\infty.
\]

d)  Assume that $\supp\,\mu_1\subset [-a,a]$ for some $a>0$.
Let
$\mu_2=\mu_{2,1}+\mu_{2,2}$, where
\[
\supp \mu_{2,1}\in (-\infty,0]\subset (-\infty,a]\quad \text{and} \quad
\supp \mu_{2,2}\in [0,\infty)\subset [-a,\infty).
\]
Then, by the arguments in b) and c),
\[
J_{\alpha_1,\alpha_2}[\mu_1,\mu_2]
\le J_{\alpha_1,\alpha_2}[\mu_1,\mu_{2,1}]+
J_{\alpha_1,\alpha_2}[\mu_1,\mu_{2,2}]<\infty,
\]
and
\[
J_{\alpha_2,\alpha_1}[\mu_2,\mu_1]
\le J_{\alpha_2,\alpha_1}[\mu_{2,1},\mu_1]+
J_{\alpha_2,\alpha_1}[\mu_{2,2},\mu_1]<\infty.
\]

If $\supp\,\mu_2\subset [-a,a]$, $a>0,$ then the considerations are completely analogous.
\end{proof}

The next result is similar in spirit to Corollary \ref{CorC11}. On the other hand, to make the relation \eqref{cauchy_id} valid for $\mu$ defined by a concrete formula, it offers  a compromise
between our assumptions on the supports on $\mu_1$ and $\mu_2$ and on their size at zero and at infinity.
\begin{theorem}\label{compact}
Let $\mu_j \in  {\mathcal M}_{\alpha_j}(\mathbb R), \alpha_j \ge 1, j=1,2,$
satisfy $\supp\,\mu_1\subset (-\infty,a)$ and $ \supp\,\mu_2\subset (a,\infty)$
for some $a \in \mathbb R.$
Then
\[
C_{\alpha_1}[\mu_1](z) C_{\alpha_2}[\mu_2](z)=C_{\alpha_1+\alpha_2}[u](z), \qquad z \in \C^+,
\]
where $u \in L^1_{\alpha_1+\alpha_2}(\R)$
is defined  by
\begin{equation}\label{Com0C}
B(\alpha_1,\alpha_2)u(\tau):=\begin{cases} -\int_{(a,\infty)}\int_{(\tau,a)}
 \frac{(\tau-s)^{\alpha_2-1}(t-\tau)^{\alpha_1-1}}
{(t-s)^{\alpha_1+\alpha_2-1}}
\,\mu_1(ds)\,\mu_2(dt),& \quad \tau<a,\\
 -\int_{(a,\tau)}\int_{(-\infty,a)}
 \frac{(\tau-s)^{\alpha_2-1}(t-\tau)^{\alpha_1-1}}
{(t-s)^{\alpha_1+\alpha_2-1}}
\,\mu_1(ds)\,\mu_2(dt),& \quad \tau>a,
\end{cases}
\end{equation}
with
\begin{align*}
(\tau-s)^{\alpha_2-1}=e^{i\pi(\alpha_2-1)}(s-\tau)^{\alpha_2-1}, \qquad \tau<s,\\
(t-\tau)^{\alpha_1-1}=e^{i\pi(\alpha_1-1)}(\tau-t)^{\alpha_1-1}, \qquad \tau>t.
\end{align*}
\end{theorem}

\begin{proof}
Let $z\in \C^{+}$ be fixed. Without loss of generality, we may assume that $a=0.$
We have
\[
C_{\alpha_1}[\mu_1](z)C_{\alpha_2}[\mu_2](z)=
\int_{(0,\infty)}\int_{(-\infty,0)}
\frac{\mu_1(ds)\,\mu_2(dt)}{(z+s)^{\alpha_1}(z+t)^{\alpha_2}}.
\]
Note that
\[
0=\lim_{\epsilon\to 0+}\,\int_{\R+i\epsilon}\frac{(\tau-s)^{\alpha_2-1}(t-\tau)^{\alpha_1-1}\,
d\tau }{(z+\tau)^{\alpha_1+\alpha_2}}
=\int_{\R}\frac{(\tau-s)^{\alpha_2-1}(t-\tau)^{\alpha_1-1}\,d\tau }{(z+\tau)^{\alpha_1+\alpha_2}}.
\]

Then, by (\ref{A2}),
for any
$s<t,$
\[
\frac{B(\alpha_1,\alpha_2)}{(z+s)^{\alpha_1} (z+t)^{\alpha_2}}
=-\frac{1}{(t-s)^{\alpha_1+\alpha_2-1}}
\left\{\int_{-\infty}^s+\int_t^\infty\right\} \frac{(\tau-s)^{\alpha_2-1}(t-\tau)^{\alpha_1-1}\,d\tau }{(z+\tau)^{\alpha_1+\alpha_2}}.
\]
So, we can formally write
\begin{equation}\label{CinA}
B(\alpha_1,\alpha_2)C_{\alpha_1}[\mu_1](z) C_{\alpha_2}[\mu_2](z):=I_1(z)+I_2(z),
\end{equation}
where
\begin{align*}
I_1(z):=-\int_{(0,\infty)}\int_{(-\infty,0)}
\int_{-\infty}^s \frac{(\tau-s)^{\alpha_2-1}(t-\tau)^{\alpha_1-1}\,d\tau }
{(t-s)^{\alpha_1+\alpha_2-1}(z+\tau)^{\alpha_1+\alpha_2}}
\,\mu_1(ds)\,\mu_2(dt),\\
I_2(z):=-\int_{(0,\infty)}\int_{(-\infty,0)}
\int_t^\infty \frac{(\tau-s)^{\alpha_2-1}(t-\tau)^{\alpha_1-1}\,d\tau }
{(t-s)^{\alpha_1+\alpha_2-1}(z+\tau)^{\alpha_1+\alpha_2}}
\,\mu_1(ds)\,\mu_2(dt).
\end{align*}
To prove the theorem it suffices to show that the (triple) integrals $I_1$ and $I_2$
in (\ref{CinA}) are absolutely convergent.
Taking into account \eqref{inneq},
we have
\begin{align}\label{iii}
&\int_{(0,\infty)}\int_{(-\infty,0)}
\int_{-\infty}^s \frac{|\tau-s|^{\alpha_2-1}|t-\tau|^{\alpha_1-1}\,d\tau }{|t-s|^{\alpha_1+\alpha_2-1}|z+\tau|^{\alpha_1+\alpha_2}}
\, |\mu_1|(ds)\, |\mu_2|(dt)\\
\le& \frac{1}{c^{\alpha_1+\alpha_2}} \int_{(0,\infty)}\int_{(-\infty,0)}
\int_{-\infty}^s
\frac{|\tau-s|^{\alpha_2-1}|t-\tau|^{\alpha_1-1}\,d\tau}
{|t-s|^{\alpha_1+\alpha_2-1}|\tau|^{\alpha_1+\alpha_2}}\,
|\mu_1|(ds) \, |\mu_2|(dt)\notag \\
\le& \frac{1}{c^{\alpha_1+\alpha_2}} \int_0^\infty\int_{-\infty}^0
\int_{|s|}^\infty
\frac{(\tau-|s|)^{\alpha_2-1}(t+\tau)^{\alpha_1-1}d\tau}
{\tau^{\alpha_1+\alpha_2}}
\frac{|\mu_1|(ds)\, |\mu_2|(dt)}{(t+|s|)^{\alpha_1+\alpha_2-1}}.\notag
\end{align}
To estimate the latter integral, observe that
\begin{align*}
&2^{1-\alpha_1}\int_{|s|}^\infty
\frac{(\tau-|s|)^{\alpha_2-1}(t+\tau)^{\alpha_1-1}d\tau}
{\tau^{\alpha_1+\alpha_2}}\\
\le& t^{\alpha_1-1}\int_{|s|}^\infty
\frac{(\tau-|s|)^{\alpha_2-1}d\tau}
{\tau^{\alpha_1+\alpha_2}}
+\int_{|s|}^\infty
\frac{(\tau-|s|)^{\alpha_2-1}d\tau}
{\tau^{\alpha_2+1}}\\
=& t^{\alpha_1-1}\int_0^\infty
\frac{\tau^{\alpha_2-1}d\tau}
{(\tau+|s|)^{\alpha_1+\alpha_2}}
+\int_0^\infty
\frac{\tau^{\alpha_2-1}d\tau}
{(\tau+|s|)^{\alpha_2+1}}\\
=& \frac{t^{\alpha_1-1}}{|s|^{\alpha_1}}\int_0^\infty
\frac{\tau^{\alpha_2-1}d\tau}
{(\tau+1)^{\alpha_1+\alpha_2}}
+\frac{1}{|s|}\int_0^\infty
\frac{\tau^{\alpha_2-1}d\tau}
{(\tau+1)^{\alpha_2+1}}\\
\le& \left(\frac{t^{\alpha_1-1}}{|s|^{\alpha_1}}+\frac{1}{|s|}\right)
\int_0^\infty
\frac{\tau^{\alpha_2-1}d\tau}
{(\tau+1)^{\alpha_2+1}}\le \frac{t^{\alpha_1-1}}{|s|^{\alpha_1}}+\frac{1}{|s|},
\end{align*}
where we used that
\[
\int_0^\infty
\frac{\tau^{\alpha_2-1}d\tau}
{(\tau+1)^{\alpha_2+1}}=\frac{1}{\alpha_2}\le 1
\]
(see e.g. \cite[p. 300, no.19]{Prudnikov1}).
Therefore, from \eqref{iii} it follows that
\begin{align*}
&2^{1-\alpha_1}\int_{(0,\infty)}\int_{(-\infty,0)}
\int_{-\infty}^s
\frac{|\tau-s|^{\alpha_2-1}|t-\tau|^{\alpha_1-1}\,d\tau}
{|t-s|^{\alpha_1+\alpha_2-1}|\tau|^{\alpha_1+\alpha_2}}\,
|\mu_1|(ds)\, |\mu_2|(dt) \\
&\le\frac{1}{c^{\alpha_1+\alpha_2}} \int_{(0,\infty)}\int_{(-\infty,0)}\left(
\frac{t^{\alpha_1-1}}{|s|^{\alpha_1}}
+\frac{1}{|s|}
\right)
\frac{
|\mu_1|(ds)\, |\mu_2|(dt)}{(t+|s|)^{\alpha_1+\alpha_2-1}}\\
&\le \frac{2}{c^{\alpha_1+\alpha_2}}  \int_{(0,\infty)}\int_{(-\infty,0)}
\frac{
|\mu_1|(ds)\, |\mu_2|(dt)}{|s|^{\alpha_1}t^{\alpha_2}}<\infty.
\end{align*}

Similarly,
\begin{align*}
&2^{1-\alpha_2}\int_{(0,\infty)}\int_{(-\infty,0)}
\int_t^\infty \frac{|\tau-s|^{\alpha_2-1}|t-\tau|^{\alpha_1-1}\,d\tau }{|t-s|^{{\alpha_1+\alpha_2}-1}|z+\tau|^{\alpha_1+\alpha_2}}
\,|\mu_1|(ds)\,|\mu_2|(dt)\\
\le& \frac{2}{c^{\alpha_1+\alpha_2}}\int_{(0,\infty)}\int_{(-\infty,0)}
\,\frac{|\mu_1|(ds)\,|\mu_2|(dt)}{|s|^{\alpha_1}t^{\alpha_2}}<\infty.
\end{align*}
Thus,
 by Fubini's theorem,
\begin{align*}
I_1(z)=&
-\int_{-\infty}^0 \int_{(0,\infty)}\int_{(\tau,0)}
 \frac{(\tau-s)^{\alpha_2-1}(t-\tau)^{\alpha_1-1}}
{(t-s)^{\alpha_1+\alpha_2-1}}
\, \frac{\mu_1(ds)\, \mu_2(dt)d\tau}{(z+\tau)^{\alpha_1+\alpha_2}}\\
=&\int_{\R}\frac{u_1(\tau)\, d\tau}{(z+\tau)^{\alpha_1+\alpha_2}},
\end{align*}
where $u_1 \in L^1_{\alpha_1+\alpha_2}(\R),$
and similarly
\begin{align*}
I_2(z)=&\int_0^\infty  \int_{(0,\tau)}\int_{(-\infty,0)}
 \frac{(\tau-s)^{\alpha_2-1}(t-\tau)^{\alpha_1-1}}
{(t-s)^{\alpha_1+\alpha_2-1}}
\,\frac{\mu_1(ds)\, \mu_2(dt) d\tau}{(z+\tau)^{\alpha_1+\alpha_2}}\\
=&\int_{\R}\frac{u_2(\tau)\, d\tau}{(z+\tau)^{\alpha_1+\alpha_2}},
\end{align*}
with $u_2 \in L^1_{\alpha_1+\alpha_2}(\R).$
Setting $u=u_1+u_2,$  the statement follows from (\ref{CinA}).

\end{proof}
Remark that even  when the assumptions of  Theorems \ref{elprodC} and \ref{compact} are satisfied simultaneously these results  yield, in general, different  measures
$\mu,$ and the one originating from Theorem \ref{compact} may not be positive for positive $\mu_1$ and $\mu_2.$
In general, if $\mu_1$ and $\mu_2$ are positive, then under the assumptions of Theorem \ref{compact}, one cannot choose a positive $\mu$ satisfying \eqref{cauchy_id}. See Theorem \ref{G1} and Example \ref{ExC} below.

Finally, we prove a version of \eqref{cauchy_id}
without
any additional assumptions on $\mu_j \in \mathcal M_{\alpha_j}(\R), j=1,2$.
 We will rely on an idea from \cite{MacGr73} used to show \eqref{prod_cauchy_in}.
While \eqref{A2} was convenient to deal with the products of the generalized Stieltjes transforms,
a direct application of \eqref{A2} to
 the products of the generalized Cauchy transforms leads to convergence problems.
So we recast \eqref{A2} into a form convenient for our current purposes.
Such an approach
seems to be more transparent and revealing than transforming \eqref{prod_cauchy_in}
to \eqref{cauchy_id} via an appropriate conformal mapping.
On the other hand, we are not able to avoid the use
of conformal mappings completely.

We  proceed with several preparations.
Setting $z=1$ and
$
\tau=(s-t)r+rt, \, r\in (0,1),
$
in \eqref{A2},
we conclude that for all  $\alpha_1,\alpha_2>0$, $a>0$ and $t>s>-1,$
\begin{align}\label{RRRG}
\frac{1}{(a(1+s))^{\alpha_1} (a(1+t))^{\alpha_2}}
=&(B(\alpha_1,\alpha_2))^{-1}\int_0^1\frac{r^{\alpha_1-1}(1-r)^{\alpha_2-1}\,dr}
{(a(1+c(r;s,t))^{\alpha_1+\alpha_2}}\\
=&\int_{0}^{1}\frac{\nu(dr)}
{(a(1+c(r;s,t))^{\alpha_1+\alpha_2}},\notag
\end{align}
where
$$c(r;s,t)=rs+(1-r)t, \qquad r \in (0,1),$$
denotes the convex combination of $s$ and $t,$
and the probability measure
$\nu=\nu_{\alpha_1,\alpha_2}$ on $[0,1]$ is given by $$\nu(dr):=(B(\alpha_1,\alpha_2))^{-1} r^{{\alpha_1}-1}(1-r)^{{\alpha_2}-1}\,dr.$$
Note that, moreover,  by analytic continuation,
(\ref{RRRG}) holds for
$
a\in \C_+$ and $s,t$ from the unit disc $\mathbb D.$

We will also need the Cayley transform $\omega(z):=\frac{z-i}{z+i}$
mapping $\overline{\C^+}$ onto $\overline{{\mathbb D}}\setminus \{1\}$
homeomorphically,
 so that $\omega(\R) = \mathbb T \setminus \{1\}$,
with  $\omega^{-1}: \overline{{\mathbb D}}\setminus\{1\}\mapsto \overline{\mathbb C^+},$
$\omega^{-1}(z)=i\frac{1+z}{1-z}.$
  Recall that we deal with the principal branch of  $z \to z^\alpha$ with the cut along
$(-\infty,0].$

\begin{lemma}\label{GrNew}
Let $\alpha_1,\alpha_2>0.$ If $s,t \in \R, s \neq t,$ and $z\in \C^{+},$ then
\begin{equation}\label{PmainN}
\frac{(i+s)^{\alpha_1}(i+t)^{\alpha_2}}{(z+s)^{\alpha_1}(z+t)^{\alpha_2}}
=
\int_{\R} \frac{G(\tau;s,t)
\,d\tau}{(z+\tau)^{\alpha_1+\alpha_2}},
\end{equation}
where
\begin{equation}\label{AL}
G(\tau;s,t):=
\frac{(i+\tau)^{\alpha_1+\alpha_2}}{\pi}
\int_0^1 {\rm Im}\left(\frac{1}{\tau-\omega^{-1}(c(r;\omega(s),\omega(t)))}\right)d\nu(r),
\end{equation}
for all $\tau \in \mathbb R,$ and $G(\cdot; s,t) \in L^1_{\alpha_1+\alpha_2}(\R)$ with $\|G(\cdot; s,t)\|_{\alpha_1+\alpha_2}\le 1.$
\end{lemma}

\begin{proof}
Let $s,t \in \R, s \neq t,$ and $z\in \C^{+}$ be fixed.
Note that by a simple calculation, for all
$\lambda \in \overline{\C^+},$
\begin{equation}\label{ElemG}
\frac{i+\lambda}{z+\lambda}=\frac{1}{a(z)(1-\omega(z)\omega(\lambda))},
\end{equation}
where for short-hand $a(z):=(2i)^{-1}(i+z)$, so that $a(z) \in \mathbb C_+.$

Therefore,
\[
\frac{(i+s)^{\alpha_1}}{(z+s)^{\alpha_1}} \cdot \frac{(i+t)^{\alpha_2}}{(z+t)^{\alpha_2}}
=\frac{1}{(a(z)(1-\omega(z)\omega(s)))^{\alpha_1}} \cdot \frac{1}{
(a(z)(1-\omega(z)\omega(t)))^{\alpha_2}},
\]
and in view of (\ref{RRRG}),
\begin{equation}\label{PL1}
\frac{(i+s)^{\alpha_1}(i+t)^{\alpha_2}}{(z+s)^{\alpha_1}(z+t)^{\alpha_2}}
=\int_0^1\frac{\nu(dr)}
{(a(z)(1-\omega(z) c(r;\omega(s),\omega(t)))^{\alpha_1+\alpha_2}}.
\end{equation}
On the other hand, by (\ref{ElemG}) again,
\[
\frac{1}{a(z)(1-\omega(z) c(r;\omega(s),\omega(t)))}=\frac{\omega^{-1}(c(r;\omega(s),\omega(t)))+i}
{\omega^{-1}(c(r;\omega(s),\omega(t)))+z}.
\]
Hence, taking into account that $s\not=t$, $s,t\in \R,$ and using Poisson's integral formula for $\C^+,$
we infer that
\begin{align*}\label{MgG}
(a(z)(1-\omega(z)c(r;\omega(s),& \omega(t))))^{-(\alpha_1+\alpha_2))}\\
=&\frac{1}{\pi}\int_{\R}{\rm Im}\left(\frac{1}{\tau-\omega^{-1}(c(r;\omega(s),\omega(t)))}\right)
\frac{(i+\tau)^{\alpha_1+\alpha_2}}{(z+\tau)^{\alpha_1+\alpha_2}}\,d\tau.
\end{align*}
By combining  (\ref{PL1}) and the above formula,
we obtain \eqref{PmainN}.

Moreover,
\begin{equation*}\label{PEst}
\int_{\R} \frac{|G(\tau;s,t)|\,d\tau}{(1+|\tau|)^{\alpha_1+\alpha_2}}
\le \frac{1}{\pi}
\int_0^1 \int_{\R}{\rm Im}\left(\frac{1}{\tau-\omega^{-1}(c(r;\omega(s),\omega(t)))}\right)d\tau \, \nu(dr)
=1,
\end{equation*}
which finishes the proof.
\end{proof}

Now we can state our product formula for the generalized Cauchy transforms of arbitrary $\mu_j \in \mathcal M_{\alpha_j}(\R), \alpha_j >0, j=1,2.$
\begin{theorem}\label{StMc}
Let $\mu_j\in\mathcal{M}_{\alpha_j}(\R)$, $\alpha_j>0$, $j=1,2$.
Then \eqref{cauchy_id} holds
with $\mu \in \mathcal M_{\alpha_1+\alpha_2}(\R)$
defined by
\begin{equation}\label{Pw}
\mu(d\tau)=u(\tau)\,d\tau+\mu_1(\{\tau\})\mu_2(d\tau)
\end{equation}
where
\[
u(\tau)
:=\int_{\R}\int_{\R,\,s\not=t}
G(\tau;t,s)
\frac{\mu_1(ds)\mu_2(dt)}{(i+s)^{\alpha_1}(i+t)^{\alpha_2}},
\]
and $G$ is given by Lemma \ref{GrNew}.
\end{theorem}

\begin{proof}
First note that by \eqref{AL}, \eqref{inneq} and Fubini's theorem,
\begin{align*}
\int_{\R}\frac{|u(\tau)|\,d\tau}{(1+|\tau|)^{\alpha_1+\alpha_2}}
\le& (\sqrt{2})^{\alpha_1+\alpha_2}
\int_{\R}\int_{\R}\int_{\R,\,t\not=s}
\frac{|G(\tau;s,t)|
|\mu_1|(ds)\, |\mu_2| (dt)\, d\tau}{(1+|s|)^{\alpha_1}(1+|t|)^{\alpha_2}(1+|\tau|)^{\alpha_1+\alpha_2}}\\
\le& (\sqrt{2})^{\alpha_1+\alpha_2}\|\mu_1\|_{\alpha_1}\|\mu_2\|_{\alpha_2},
\end{align*}
so that $u \in L^1_{\alpha_1+\alpha_2}(\R).$
Let $z \in \C^+$ be fixed.
Write formally
\[
C_{\alpha_1}[\mu_1](z)C_{\alpha_2}[\mu_2](z)
=\int_{\R}\int_{\R,\,s\not=t}
\frac{\mu_1(ds)\mu_2(dt)}{(z+s)^{\alpha_1}(z+t)^{\alpha_2}}+\int_{\R}
\frac{\mu_1(\{\tau\})\, \mu_2(d\tau)}{(z+\tau)^{\alpha_1+\alpha_2}}.
\]
By Lemma \ref{GrNew} and Fubini's theorem again,
\begin{align*}
&
\int_{\R}\int_{\R,\,t\not=s}
\frac{\mu_1(ds)\, \mu_2(dt)}{(z+t)^{\alpha_1}(z+s)^{\alpha_2}}\\
=&\int_{\R}\int_{\R}\int_{\R,\,t\not=s}
G(\tau;s,t)
\frac{\mu_1(ds)\mu_2(dt)}{(s+i)^{\alpha_1}(i+t)^{\alpha_2}}\frac{d\tau}{(z+\tau)^{\alpha_1+\alpha_2}}\\
=&\int_{\R}\frac{u(\tau)\, d\tau}{(z+\tau)^{\alpha_1+\alpha_2}}.
\end{align*}

Moreover, for $\mu_d(d\tau)=\mu_1(\{\tau\})\mu_2(d\tau)$  it is easy to check that $\mu_d \in \mathcal M_{\alpha_1+\alpha_2}(\R),$ and $\|\mu_d\|_{\alpha_1+\alpha_2}\le \|\mu_1\|_{\alpha_1}\|\mu_2\|_{\alpha_2}.$
\end{proof}

Note that as in Theorems \ref{elprodC} and \ref{compact},
$\mu$ does not have singular continuous component.
However,
in contrast to Theorem \ref{elprod} dealing with Stieltjes transforms,
the nature of $\mu$ in \eqref{Pw}
 is rather implicit.
Moreover, as we show below, the positivity of $\mu_1$ and $\mu_2$ is not, in general, inherited by $\mu.$
As far as  Theorem \ref{elprodC} yields positive $\mu$ for positive $\mu_1$ and $\mu_2,$
this shows that the assumptions \eqref{C1H} are, in general, necessary for its conclusion,
and it particular for the validity of the formula for $\mu$ given by \eqref{m1C} and \eqref{m0C}.

\subsection{Positivity of representing measures for products of Cauchy transforms, and related matters}

We proceed with deriving a criterion for $\mu$ in \eqref{cauchy_id} to be positive
if $\alpha_1=\alpha_2=1$ and $\mu_j \in \mathcal M^+_1(\R), j=1,2.$ The statement will help us
us to produce various examples of $\mu_1$ and $\mu_2$ from $\mathcal M^+_1(\R)$ not allowing for positive $\mu$ in \eqref{cauchy_id}.
First, we need to separate a simple lemma.
\begin{lemma}\label{J}
Let $\mu\in \mathcal{M}_2^{+}(\R)$.
If $C_2[\mu](z)=0$ for all $z\in \C^{+},$
then
$
\mu(dt)=c\,dt
$
for some constant $c\ge 0$.
\end{lemma}

\begin{proof}
By assumption,
\begin{equation*}
\int_{\R} \frac{\mu(dt)}{(t+z)^2}=0,\qquad z\in \C^{+}.
\end{equation*}
Passing to conjugates and using the positivity of $\mu,$
we have
\begin{equation}\label{yy}
\int_{\R} \frac{\mu(dt)}{(t-z)^2}=0, \qquad z \in \C^+.
\end{equation}
Let
\[
F(z):=\int_{\R}\left(\frac{1}{t-z}-\frac{t}{1+t^2}\right)\,\mu(dt)=
\int_{\R}\frac{(1+tz)\,\mu(dt)}{(t-z)(1+t^2)}, \qquad z \in \C^+,
\]
be the Nevanlinna-Pick function associated to $\mu.$
(See e.g. \cite[p. 84]{Rosen} for the notion
of Nevanlinna-Pick functions.)
Then $F$ is analytic on $\C^+,$ and ${\rm Im}\, F(z)\ge 0$ for all $z \in \C^+.$
Differentiating $F$ and using \eqref{yy}, we infer that
$F'(z)=0$ for all $z\in \C^{+}.$
Hence there exists $a \in \C$ with ${\rm Im}\, a \ge 0$ such that
$
F(z)=a.
$
 Then, by the inversion formula for Nevanlinna-Pick functions, see e.g. \cite[Theorem 5.4]{Rosen},
$\mu(dt)=({\rm Im} \, a/\pi)\, dt=c\,dt.$
\end{proof}

\begin{theorem}\label{G1}
Let $\mu_j\in \mathcal{M}_1^{+}(\R)$, $j=1,2,$ satisfy
\[
\supp\,\mu_1\subset (-\infty,a) \qquad \text{and} \qquad
\supp\,\mu_2\subset (a,\infty)
\]
for some $a \in \R.$
Then there exists
$\mu\in \mathcal{M}_2^{+}(\R)$ such that \eqref{cauchy_id} holds
if and only if
\begin{equation}\label{gv1}
\int_{(a,\infty)}\int_{(-\infty,a)}
\frac{\mu_1(ds)\mu_2(dt)}{t-s}<\infty.
\end{equation}
\end{theorem}

\begin{proof}
Without loss of generality we may assume that $a=0.$
By Theorem \ref{compact},
we have
\[
C_1[\mu_1](z)C_1[\mu_2](z)=C_2[u](z),\qquad z\in \C^{+},
\]
where the \emph{negative} function $u\in L^1_2(\R)$ is
given by
\begin{equation}\label{AA}
u(\tau)=\begin{cases}-\int_{(0,\infty)}\int_{(\tau,0)}
\frac{\mu_1(ds)\mu_2(dt)}{t-s},& \quad \tau<0,\\
-\int_{(0,\tau)}\int_{(-\infty,0)}
\frac{\mu_1(ds)\mu_2(dt)}{t-s},& \quad \tau>0.
\end{cases}
\end{equation}

Suppose that there exists
$\mu\in \mathcal{M}_2^{+}(\R)$ such that \eqref{cauchy_id} holds.
Letting then
\[
\nu(d\tau):=|u(\tau)|\,d\tau+\mu(d\tau),
\]
we conclude that $\nu \in \mathcal{M}_2^{+}(\R)$ and $C_2[\nu](z)=0, z \in \C_+.$ 
Hence by Lemma  \ref{J} there exists
 $c\ge 0$ such that
$
\nu(d\tau)=c\, d\tau.
$
Since $\mu$ is positive, from the definition of $\nu$ it follows  
that
\[
|u(\tau)|=\lim_{s \to 0}\frac{1}{s}\int_{\tau}^{\tau+s} |u(r)|\, dr
\le \limsup_{s \to 0}
\frac{1}{s}\int_{\tau}^{\tau+s} \nu (dr) = c,
\]
for almost all $\tau \in \R,$
and in view of \eqref{AA},
\[
\int_{(0,\tau)}\int_{(-\infty,0)}
\frac{\mu_1(ds)\mu_2(dt)}{t-s}\le c \qquad \text{for a.e.}\, \tau>0.
\]
Thus, by applying Fatou's Lemma, we infer that (\ref{gv1}) (with $a=0$) is true.

Conversely, let now (\ref{gv1}) hold. Denote by $R(\mu_1, \mu_2)$ the left-hand side of \eqref{gv1} and define
\[
\mu(d\tau):=[u(\tau)+R(\mu_1,\mu_2)]\,d\tau,
\]
where $u$ is given  by (\ref{AA}).
Then, taking into account
the elementary equality
\[
\int_{\R}\frac{d\tau}{(z+\tau)^2}=0,\qquad z\in \C^{+},
\]
we conclude
that $\mu\in \mathcal{M}_2^{+}(\R),$ and (\ref{cauchy_id}) holds
as well. (Alternatively, to arrive at \eqref{cauchy_id}, one may note that \eqref{gv1} implies \eqref{C1H}
and  refer to Theorem \ref{elprodC}.)
\end{proof}

We illustrate Theorem \ref{G1} with a simple concrete example, though other more general examples can be constructed easily.
\begin{example}\label{ExC}
Let $\mu_1$ and $\mu_2$ from $\mathcal M^+_1(\R)$ be defined by
\[\mu_1(dt)= \chi_{(-\infty, -1]}(t)\frac{dt}{|t|^{1/2}} \qquad \text{and}\qquad \mu_2(dt)=\chi_{[1,\infty)}(t)\frac{dt}{t^{1/2}}.\]
Then we have
\[
\int_{(0,\infty)}\int_{(-\infty,0)}
\frac{\mu_1(ds)\mu_2(dt)}{t-s}=\int_1^\infty\int_1^\infty
\frac{ds\,dt}{t^{1/2}s^{1/2}(t+s)}=\infty,
\]
and therefore Theorem \ref{G1} with $a=0$ implies that there is no $\mu\in \mathcal{M}_2^{+}(\R)$
satisfying \eqref{cauchy_id}.
\end{example}

On the other hand,
if $\alpha_1+\alpha_2=1, \alpha_j >0,$ and $ \mu_ j \in \mathcal M^+_{\alpha_j}(\R), j=1,2,$
then
under mild additional assumptions on $\mu_1$ and $\mu_2,$
one can always  find a $\mu \in \mathcal M^+_1(\R)$ satisfying
$C_{\alpha_1}[\mu_1]C_{\alpha_2}[\mu_2]=C_{1}[\mu].$
To show this, recall that by the well-known Kac (Katz) criterion
 (see e.g. \cite{Kac}, \cite{Krein}, or \cite[Section 3]{Aron}),
a function $f$ holomorphic on $\C^{+}$
admits the representation
\[
f(z)=C_1[\mu](z), \qquad z\in \C^{+},
\]
with $\mu \in \mathcal M_1^+(\R)$ if and only if
\begin{equation}\label{kkr1}
{\rm Im}\,f(z)\le 0,\qquad z\in \C^{+},
\end{equation}
and
\begin{equation}\label{kkr2}
\int_1^\infty \frac{|{\rm Im}\,f(iy)|\,dy}{y}<\infty.
\end{equation}
Such a $\mu$ is necessarily unique.

For $\alpha\in (0,1),$ define
\begin{equation}\label{QQR}
Q_{\alpha}(x,y):=\int_0^1\frac{d\tau}{(x\tau+1)^\alpha(y\tau+1)^{1-\alpha}},\qquad x,y\ge 0.
\end{equation}

\begin{theorem}\label{G11}
For $\alpha\in (0,1),$ let
$
\mu_1\in \mathcal{M}_{\alpha}^{+}(\R)$ and
$
 \mu_2\in \mathcal{M}_{1-\alpha}^{+}(\R).
$
Then there exists
$\mu\in \mathcal{M}_1^{+}(\R)$ such that
\begin{equation}\label{C100}
C_{\alpha}[\mu_1](z)C_{1-\alpha}[\mu_2](z)=C_1[\mu](z),\quad z\in \C^{+},
\end{equation}
if and only if there exists $a \ge 0$ satisfying
\begin{equation}\label{QQQ1}
\int_a^\infty \int_{-\infty}^{-a} Q_{\alpha}(|s|,t)\mu_1(ds)\mu_2(dt)<\infty,
\end{equation}
and
\begin{equation}\label{QQQ2}
\int_{-\infty}^{-a} \int_a^{\infty} Q_{\alpha}(s,|t|)\mu_1(ds)\mu_2(dt)<\infty.
\end{equation}
\end{theorem}

\begin{proof}

Setting $f=C_{\alpha}[\mu_1]C_{1-\alpha}[\mu_2],$ we will show that $f$ satisfies the conditions
\eqref{kkr1} and \eqref{kkr2} of Kac's criterion if and only if \eqref{QQQ1} and \eqref{QQQ2} hold.

Note  that since for all $s,t\in \R$ and $z\in \C^{+}:$
\[
{\rm Im}\,\frac{1}{(z+s)^{\alpha}(z+t)^{1-\alpha}}\le 0,
\]
we have
\begin{equation}\label{stiltj}
{\rm Im}\, f(z)= {\rm Im}\, \int_{\R}\int_{\R}  \frac{\mu_1(ds)\mu_2(dt)}{(z+s)^{\alpha}(z+t)^{1-\alpha}} \le 0, \qquad z \in \C^+,
\end{equation}
so that $f$ satisfies \eqref{kkr1}, and it is necessary to check that \eqref{kkr2} is then equivalent to
\eqref{QQQ1} and \eqref{QQQ2}.
To prove this we proceed with two reductions.
First, since
$
Q_\alpha(x,y)\le 1
$
 for all $x,y\ge 0,$
we may assume without loss of generality that $a=0.$
Second, note that it suffices to consider the case when
\begin{equation}\label{WarA}
\supp \mu_1 \subset (-\infty,0] \qquad \text{and} \qquad \supp \mu_2 \subset [0,\infty).
\end{equation}
   Indeed, for $j=1,2,$ and  $k=1,2,$
define  $\mu_{j,k} \in {\mathcal M}^+_{\alpha_j}(\R)$
as follows:
\begin{align*}
\mu_{1,k}:=&\chi_{(-\infty,0)} \mu_1 + \chi_{[0, \infty)} \mu_1, \\
\mu_{2,k}:=&\chi_{(-\infty, 0)} \mu_2 + \chi_{[0, \infty)} \mu_2.
\end{align*}
If $f_{1,k}=C_{\alpha}[\mu_{1,k}]$ and $f_{2, k}=C_{1-\alpha}[\mu_{2,k}], k=1,2,$ then
 write
\begin{equation*}\label{ExCCC}
f=(f_{1,1}+f_{1,2}) (f_{2,1}+f_{2,2})=F_1+F_2,
\end{equation*}
where
\begin{align*}
f_1=f_{1,1} \cdot f_{2,1} + f_{1,2}\cdot f_{2,2} \qquad \text{and} \qquad
f_2=f_{1,1}\cdot  f_{2,2}+f_{1,2}\cdot f_{2,1},
\end{align*}
and  $F_1$ and $F_2$ satisfy \eqref{kkr1}
 in view of \eqref{stiltj} applied to each of the four summands above.
Furthermore,  $f_1 \in C_{1}(\mathcal M^+_{1}(\R))$
by Corollary \ref{CorC11}, b) and c). Hence,
using Kac's criterion, we infer that $f \in C_1(\mathcal M^+_1(\R))$
if and only if $f_2 \in C_1(\mathcal M^+_1(\R)),$
and so may restrict our attention to $f_2.$
Taking into account that $Q_{1-\alpha}(x,y)=Q_\alpha (y,x), x,y \ge 0,$ it remains to note that \eqref{WarA} holds for $\mu_{1,1}$ and $\mu_{2,2},$
as well as
for $\mu_{1,2}$ and $\mu_{2,1},$ up to relabeling.

Thus,
it is sufficient to consider the case when $\mu_1=\mu_{1,1},$ $\mu_2=\mu_{2,2},$
\[ f=f_2=f_{1,1} \cdot f_{2,2}= C_\alpha [\mu_{1}] C_{1-\alpha}[\mu_{2}],
\]
and to show that \eqref{QQQ1} is then equivalent
 to \eqref{kkr2}.
To this aim, observe that for all $y>0$ and
$t\le 0$, $s \ge 0,$
\[
\arg\, \left((iy+s)^{\alpha}\right)\in [\pi\alpha/2,\pi\alpha],\quad \arg\,\left((iy+t)^{1-\alpha}\right)\in [0,\pi(1-\alpha)/2],
\]
so that
\[
-\arg\,\left(\frac{1}{(iy+s)^{\alpha}}\right)-
\arg\,\left(\frac{1}{(iy+t)^{1-\alpha}}\right) \in [\pi\alpha/2,\pi/2+\pi\alpha/2].
\]
Hence, if
$
\varphi:=\min (\pi\alpha/2,\pi/2-\pi\alpha/2),$ then $\varphi \in (0,\pi/4],$ and
\begin{align*}
\frac{\sin\varphi}{|iy+s|^{\alpha}|iy+t|^{1-\alpha}}\le&
-{\rm Im}\,\frac{1}{(iy+s)^{\alpha}(iy+t)^{1-\alpha}}\\
\le& \frac{1}{|iy+s|^{\alpha}|iy+t|^{1-\alpha}},\quad y>0.
\end{align*}
Therefore, if $y>0,$ then
\begin{align}\label{RR11}
\sin\varphi\int_0^\infty \int_{-\infty}^0\frac{\mu_1(ds)\mu_2(dt)}{|iy+s|^{\alpha}|iy+t|^{1-\alpha}}
\le& -{\rm Im}\,f(iy)\\
\le&
\int_0^\infty \int_{-\infty}^0\frac{\mu_1(ds)\mu_2(dt)}{|iy+s|^{\alpha}|iy+t|^{1-\alpha}}.\notag
\end{align}
Using (\ref{RR11}) and Fubini's theorem, we infer that
\eqref{kkr2} for $f$ is equivalent to
\begin{equation}\label{CCC}
\int_0^\infty \int_{-\infty}^0
\int_1^\infty\frac{dy}{y|iy+s|^{\alpha}|iy+t|^{1-\alpha}}
\mu_1(ds)\mu_2(dt)<\infty.
\end{equation}

In turn, since for all $s,t\in \R,$
\[
\int_1^\infty\frac{dy}{y|iy+s|^{\alpha}|iy+t|^{1-\alpha}}=
\int_0^1\frac{d\tau}{(s^2\tau^2+1)^{\alpha/2}(t^2\tau^2+1)^{(1-\alpha)/2}},
\]
and so
\[
Q_{\alpha}(|s|,t)
\le \int_1^\infty\frac{dy}{y|iy+s|^{\alpha}|iy+t|^{1-\alpha}}
\le \sqrt{2} Q_{\alpha}(|s|,t),
\]
the property (\ref{CCC}) is equivalent to
(\ref{QQQ1}), and the statement follows.
\end{proof}

It is useful to note that \eqref{QQQ1} and \eqref{QQQ2} can be replaced with a stronger but more transparent
condition on the
size of $\mu_1 \in \mathcal M^+_\alpha(\R)$ and $\mu_2\in \mathcal M^+_{1-\alpha}(\R).$
On this way, the next simple estimate of $Q_\alpha$ will be relevant.
\begin{lemma}\label{Dost}
If $Q_{\alpha}$ is given by
(\ref{QQR}), then
for every $\beta\in [0,1],$
\begin{equation}\label{ApBeta}
Q_\alpha(x,y)\le 2\frac{\log^\beta x\log^{1-\beta} y}{x^{\alpha}y^{1-\alpha}}, \qquad x,y \ge \exp{(\pi/\sin(\pi\alpha))}.
\end{equation}
\end{lemma}

\begin{proof}
Note that for every $x\ge \exp{(\pi/\sin(\pi\alpha))},$
\begin{align*}
&Q_{\alpha}(x,y)
\le \frac{1}{y^{1-\alpha}}
\int_0^1 \frac{d\tau}{\tau^{1-\alpha}(x\tau+1)^{\alpha}}
=\frac{1}{x^{\alpha}y^{1-\alpha}}
\int_0^x\frac{d\tau}{\tau^{1-\alpha}(\tau+1)^{\alpha}}\\
\le&  \frac{1}{x^{\alpha}y^{1-\alpha}}\left(
\int_0^1\frac{d\tau}{\tau^{1-\alpha}(\tau+1)^{\alpha}}+
\int_1^{x}\frac{d\tau}{\tau}\right)
=\frac{1}{x^{\alpha}y^{1-\alpha}}\left(\frac{\pi}{\sin(\pi\alpha)}+\log x\right),
\end{align*}
hence
\begin{equation}\label{ApBeta1}
 Q_{\alpha}(x,y)\le
2\frac{\log x}{x^{\alpha}y^{1-\alpha}},\qquad x\ge \exp{(\pi/\sin(\pi\alpha))}.
\end{equation}
Similarly,
\begin{equation}\label{ApBeta2}
Q_{\alpha}(x,y)\le
2\frac{\log y}{x^{\alpha}y^{1-\alpha}},\qquad  y\ge  \exp{(\pi/\sin(\pi\alpha))}.
\end{equation}
Now writing $Q_\alpha=Q^\beta_\alpha Q^{1-\beta}_\alpha$  and estimating $Q^\beta_\alpha$  by (\ref{ApBeta1}) and $Q^{1-\beta}_\alpha$ via (\ref{ApBeta2}),
we obtain \eqref{ApBeta}.
\end{proof}

\begin{corollary}\label{C11}
Let $\alpha\in (0,1),$ and assume that
$\mu_1\in \mathcal{M}_{\alpha}^{+}(\R)$ and
 $\mu_2\in \mathcal{M}_{1-\alpha}^{+}(\R).$
If
$
\beta_j\ge 0, j=1,2$,
$\beta_1+\beta_2= 1,$
and
\[
A:=\int_{|t|\ge a_0} \frac{\log^{\beta_1}|t|\,\mu_1(dt)}{|t|^{\alpha}}+
\int_{|t|\ge a_0} \frac{\log^{\beta_2}|t|\,\mu_2(dt)}{|t|^{1-\alpha}}<\infty,
\]
for some $a_0 \ge 1,$ then there exists $\mu\in \mathcal{M}_1^{+}(\R)$ satisfying (\ref{C100}).
\end{corollary}

\begin{proof}
Setting
\[
a=\max (a_0, \exp{(\pi/\sin(\pi\alpha))})
\]
and employing Lemma \ref{Dost},
we conclude that the left-hand side of either \eqref{QQQ1} or \eqref{QQQ2} is bounded by $2A^2.$
The  statement follows now from
Theorem \ref{G11}.
\end{proof}

Corollary \ref{C11} is nearly optimal, and  the following example illustrates this point.
\begin{example}\label{Exa1}
Let us illustrate the sharpness of
Corollary \ref{C11}.
Let $\alpha\in (0,1)$ be fixed and for $\delta_j>0$, $j=1,2,$ define $\mu_1\in \mathcal{M}_\alpha^{+}(\R)$ and $\mu_2\in \mathcal{M}_{1-\alpha}^{+}(\R)$ by
\begin{equation}\label{MeasL}
\mu_1(dt):=\frac{\chi_{(-\infty,-e)}(t)\,  dt}{|t|^{1-\alpha}\log^{1+\delta_1} |t|} \quad \text{and} \quad
\mu_2(dt):=\frac{\chi_{(e,\infty)}(t)\, dt}{t^{\alpha}\log^{1+\delta_2} t}.
\end{equation}
Then by Fubini's theorem
\begin{align*}
&\int_0^1\int_{s\ge e}\int_{t\le -e}
\frac{\mu_1(dt)\mu_2(ds)}{(|t|\tau+1)^\alpha(s\tau+1)^{1-\alpha}}
\,d\tau\\
\ge&\int_{0}^{1/e}\left(\int_{1/\tau}^\infty \frac{dt}{(\tau t+1)^{1-\alpha}t^{\alpha}\log^{1+\delta_1}t} 
 \int_{1/\tau}^\infty \frac{ds}{(\tau s+1)^{\alpha}s^{1-\alpha}\log^{1+\delta_2}s}
\right)\, d\tau \\
\ge&\int_0^{1/e}\left(\frac{1}{2^{1-\alpha} \tau^{1-\alpha}} 
\int_{1/\tau}^\infty \frac{dt}{t\log^{1+\delta_1}t} \cdot  \frac{1}{2^\alpha \tau^{\alpha}}
\int_{1/\tau}^\infty \frac{ds}{s\log^{1+\delta_2}s}\right)\, d\tau\\
 \ge& \frac{1}{2\delta_1\delta_2}\int_0^{1/e}\frac{d\tau}{\tau|\log\tau|^{\delta_1+\delta_2}}.
\end{align*}
If $\delta_1+\delta_2\le 1$,
then the latter inequality implies that
for $\mu_1$ and $\mu_2$ given by \eqref{MeasL}
the property (\ref{cauchy_id})
does not hold. In fact, taking into account Corollary \ref{C11},
(\ref{cauchy_id}) is satisfied for $\mu_1$ and $\mu_2$ from \eqref{MeasL}
if and only if $\delta_1+\delta_2>1$.
\end{example}

\subsection{Liftings of Cauchy transforms via positive measures}
Recall from Remark \ref{gener} that
if $\beta >\alpha >0,$   $\nu \in \mathcal M_\alpha^+(\R_+),$ and  $\mu \in\mathcal M^+_\beta(\R_+)$
 is given by \eqref{up}, then
$S_\alpha [\nu]=S_\beta [\mu],$
so that $S_\alpha [\nu]$ can be "lifted" to $\mathcal S_\beta.$
In particular,  this fact  leads sometimes to useful alternative versions of \eqref{elprod},
where for $\mu_1 \in \mathcal M^+_{\alpha_1}(\R)$ and $\mu_2 \in \mathcal M^+_{\alpha_2} (\R)$ a measure $\mu$ belongs to the class $\mathcal M^+_\beta(\R_+), \beta >\alpha_1+\alpha_2,$ wider than $\mathcal M^+_{\alpha_1+\alpha_2}(\R),$
and at the same time $\mu$ has better properties than the choice provided by \eqref{elprod}.
(For instance, one may pass from a singular $\mu$ to a locally absolutely continuous one.)

 The considerations in the preceding subsections raise a natural question whether a similar
property holds for the generalized Cauchy transforms, so that
$C_\alpha (\mathcal M^+_\alpha(\R)) \subset C_\beta (\mathcal M^+_\beta(\R))$.
Though Theorem \ref{StMc} may suggest a positive answer, the answer is in general
negative, see Example \ref{Gfun} below. At the same time,  we identify several natural situations, when such a choice is indeed possible. Our arguments will depend on the next auxiliary statement.

\begin{proposition}\label{Lkarp}
Let
$
\nu_1\in \mathcal{M}_\alpha^{+}(\R), \alpha >0,$ be such that $\supp\,\nu_1\in (-\infty,0].$
Then for every $\beta>0,$
\begin{equation}\label{BetaK}
C_\alpha[\nu_1](z)=e^{i\pi\beta} C_{\alpha+\beta}[\nu_2](z),\qquad z\in \C^{+},
\end{equation}
where the locally absolutely continuous  $\nu_2 \in \mathcal{M}_{\alpha+\beta}^{+}(\R)$ is given by
\begin{equation}\label{nu2}
\nu_2(d\tau):=\frac{1}{B(\alpha,\beta)}\left(\int_{(\tau,0]}(t-\tau)^{\beta-1}\,\nu_1(dt)\right)\,d\tau.
\end{equation}
\end{proposition}

\begin{proof}
By Fubini's theorem, for all $z \in \mathbb C^+,$
\begin{align}\label{Beta11}
B(\alpha,\beta)C_{\alpha+\beta}[\nu_2](z)=&
\int_{-\infty}^0\left(\int_{-\infty}^t\frac{(t-\tau)^{\beta-1}\,d\tau}{(z+\tau)^{\alpha+\beta}}\right)\,\nu_1(dt)\\
=&\int_{-\infty}^0\left(\int_0^\infty \frac{s^{\beta-1}\,ds}{(z+t-s)^{\alpha+\beta}}\right)\,\nu_1(dt).\notag
\end{align}

Next, let $t\le 0$ be fixed. Extending the function $z \to z^\beta$ by continuity to the upper side of the cut along $(-\infty,0],$ we have
\begin{align*}
\int_0^\infty& \frac{s^{\beta-1}\,ds}{(z+t-s)^{\alpha+\beta}}
=e^{-i(\alpha+\beta)\pi}\int_0^\infty\frac{s^{\beta-1}\,ds}{(|z+t|+s)^{\alpha+\beta}}\\
=&e^{i(\alpha+\beta)\pi}{|z+t|^{\alpha}}\int_0^\infty\frac{s^{\beta-1}\,ds}{(1+s)^{\alpha+\beta}}
=e^{i\pi\beta}\frac{B(\alpha,\beta)}{(z+t)^{\alpha}}
\end{align*}
for all $z<0.$ Then, by the boundary uniqueness for functions bounded and analytic in $\C^+,$
\begin{equation}\label{AnCon}
\int_0^\infty \frac{s^{\beta-1}\,ds}{(z+t-s)^{\alpha+\beta}}
=e^{i\pi \beta}\frac{B(\alpha,\beta)}{(z+t)^{\alpha}},\qquad z\in \C^{+}.
\end{equation}

Now (\ref{BetaK})
follows from (\ref{Beta11}) and (\ref{AnCon}).
\end{proof}

The following corollary from Proposition \ref{Lkarp} showing that
$C_\alpha (\mathcal M^+_\alpha(\R)) \subset C_{\alpha+2}(\mathcal M^+_\beta(\R))$
is straightforward.
\begin{corollary}\label{Incl}
Let $\nu\in \mathcal{M}_\alpha^{+}(\R)$, $\alpha>0$. Then
there exists $\mu\in \mathcal{M}_{\alpha+2}^{+}(\R)$ such that
\[
C_\alpha[\nu](z)=C_{\alpha+2}[\mu](z),\qquad z\in \C^{+}.
\]
\end{corollary}
To deduce the result from Proposition \ref{Lkarp}, it suffices to write $\nu=\chi_{(-\infty,0)} \nu +\chi_{[0,\infty)}\nu$ and to push-forward $\chi_{[0,\infty)} \nu$ to $(-\infty,0].$

However, the statement similar to Corollary \ref{Incl} may not hold if  $\mathcal{M}_{\alpha+2}^{+}(\R)$ replaced by the smaller class
$\mathcal{M}_{\alpha+1}^{+}(\R).$
We proceed with a criterion addressing the problem when $\alpha=1$
and relying on
Proposition \ref{Lkarp}.
\begin{theorem}\label{Inc1}
Let
$\nu\in \mathcal{M}^{+}_1(\R).$
Then there exists $\mu\in \mathcal{M}^+_2(\R)$ satisfying
\begin{equation}\label{T120}
C_1[\nu](z)=C_2[\mu](z),\qquad z\in \C^{+},
\end{equation}
if and only if
\begin{equation}\label{GvRe}
\int_{-\infty}^0\nu(dt)<\infty.
\end{equation}
\end{theorem}

\begin{proof}
Let us first assume that, in addition,  $\supp\,\nu\in(-\infty,0].$
Then the proof is similar to the proof of the
Theorem \ref{G1}. Suppose that (\ref{T120}) holds, and let
$\nu_2$ be given \eqref{nu2} with $\nu_1=\nu.$
Then, by Proposition \ref{Lkarp},  we obtain that
\[
C_2[\mu](z)+C_2[\nu_2](z)=0,\qquad z\in \C^{+}.
\]
So, in view of Lemma \ref{J}, there exists $c\ge 0$ such that
\[
\mu(dt)+\nu_2(dt)=c\,dt,
\]
and then
\[
\int_{(\tau,0]}\nu(dt)\le c,\qquad \tau\le  0,
\]
which is equivalent to (\ref{GvRe}).

Conversely, if (\ref{GvRe}) holds,
then
\[
C_1[\nu](z)=C_2[\mu](z),\qquad z\in \C^{+},
\]
with the locally absolutely continuous measure $\mu\in \mathcal{M}_2^{+}(\R)$ defined by
\begin{equation*}
\mu(dt):=\begin{cases} \left(\int_{(-\infty,t]}\nu(d\tau)\right)\,dt,& \qquad t\le 0,\\
\left(\int_{(-\infty,0]}\nu(d\tau)\right)\,dt,& \qquad t > 0.
\end{cases}
\end{equation*}

In the general case, write $\nu=\nu_1 +\nu_2$ with $\nu_1,\nu_2 \in \mathcal M^+_1(\R)$ such that $\supp \nu_1\subset (-\infty,0]$
and $\supp \nu_2 \in [0,\infty).$
By applying the preceding case to $\nu_1$ and
\eqref{up} to $\nu_2$, we get the sufficiency.
For the proof of necessity one may simply restrict $\nu$ to $(-\infty,0]$
and refer to the first part of the proof.
\end{proof}
\begin{example}\label{Gfun}
Let $\alpha\in (0,1)$ and
\[
\nu(dt)=\chi_{(-\infty,-1]}(t)\frac{dt}{|t|^\alpha},
\]
so that $\nu \in \mathcal{M}_1^{+}(\R)$ but (\ref{GvRe}) does not hold. Hence, there is no $\mu\in \mathcal{M}_2^{+}(\R)$ satisfying
(\ref{T120}). Note that, by a direct computation, for every $z \in \C^+,$
\[
C_1[\nu](z)=\int_{-\infty}^{-1}\frac{dt}{(z+t)|t|^\alpha}
=-\int_0^1\frac{\tau^{\alpha-1}\,d\tau}{1-z\tau}=-\alpha \,_2F_1(1,\alpha,\alpha+1;z),
\]
where $_2F_1$ is a hypergeometric function, see \cite[p. 65]{An}.
\end{example}

Invoking geometrical arguments and
Kac's theorem formulated above,  we finish this section with a lifting criterion from $\mathcal M^+_\alpha(\R), \alpha  \in (0,1),$
to $\mathcal M^+_1(\R).$ It looks similar to Theorem \ref{Inc1}, although
its proof is based on a different idea.
\begin{theorem}\label{C1dos}
Let $\nu\in \mathcal{M}_\alpha^{+}(\R),$ $\alpha\in (0,1)$. Then
there exists  $\mu\in \mathcal{M}_1^{+}(\R)$ such that
\begin{equation}\label{Kac11}
C_\alpha[\nu](z)=C_1[\mu](z),\qquad z\in \C^{+},
\end{equation}
if and only if
\[
\int_{-\infty}^0\frac{\log(|t|+1)\,\nu(dt)}{(1+|t|)^\alpha}<\infty.
\]
\end{theorem}

\begin{proof}
As in the proof of  Theorem \ref{Inc1},  we may assume without loss of generality that
$
\supp\,\nu\subset (-\infty,0].
$
First, observe that
\[
{\rm Im}\,C_\alpha[\nu](z)\le 0,\quad z\in \C^{+}.
\]
Since
\[
\arg \left((t+iy)^\alpha\right)\in [\pi\alpha/2,\pi\alpha],\qquad y>0,\quad t<0,
\]
then, similarly to the proof of Theorem \ref{G11},
\[
c_1 \int_{-\infty}^0 \frac{\nu(dt)}{|iy+t|^\alpha} \le {|\rm Im}\,C_\alpha[\nu](iy)|\le c_2 \int_{-\infty}^0\frac{\nu(dt)}{|iy+t|^\alpha},\qquad y>0,
\]
for some $c_1,c_2 >0.$
Thus,
\begin{align*}
c_1 & \int_{\R}\int_1^\infty \frac{dy}{y|iy+t|^\alpha}\,\nu(dt) \le \int_1^\infty \frac{|{\rm Im}\, C_\alpha(\nu)(iy)|}{y}\,dy\\
\le& c_2 \int_{\R}\int_1^\infty \frac{dy}{y|iy+t|^\alpha}\,\nu(dt).
\end{align*}
So, using Lemma \ref{elem}, we infer that 
\[
 \int_{-\infty}^0\int_1^\infty \frac{dy}{y|iy+t|^\alpha}\,\nu(dt)<\infty \quad \Longleftrightarrow \quad
\int_{-\infty}^0 \frac{\log(|t|+1)\,\nu(dt)}{(1+|t|)^\alpha}<\infty,
\]
and obtain the assertion by the Kac criterion \eqref{kkr1},\eqref{kkr2}.
\end{proof}

\section{Final remarks}\label{final}

We finish the paper with two remarks addressing the scope of the paper
and some relevant problems left open.

\subsection{} The emphasis in this paper is put on the study of the generalized Stieltjes and Cauchy transforms
of measures from $\mathcal M^+_{\alpha}$-spaces
or of functions from $L^1_\alpha$-spaces. As far as our considerations fit into a more general framework of complex Radon measures, which
help us to define the resulting product measures properly and to not distinguish between functions and measures, we have chosen it to simplify
and unify our presentation.
However, our choice looks more like a matter of convenience
than a substantial technical advance,
and we are not aware of any of its applications of a separate interest.

\subsection{} Unfortunately, we were not able to answer the next two questions, which we believe might deserve further research. First, in the framework of Theorem \ref{StMc}, given $\mu_1 \in \mathcal M^+_{\alpha_1}(\R)$
and $\mu_2 \in \mathcal M^+_{\alpha_2}(\R), \alpha_1, \alpha_2 >0,$ it is not, in general, clear whether there always exists a $\mu \in \mathcal M^+_{\alpha_1+\alpha_2}(\R)$  such that $C_{\alpha_1}[\mu_1]C_{\alpha_2}[\mu_2]=C_{\alpha_1+\alpha_2}[\mu].$
Theorems \ref{G1} and  \ref{G11} do not clarify this issue, and Theorem \ref{elprodC} produces such a $\mu$ under somewhat stringent additional assumptions.  Second, given $\alpha_2 >\alpha_1>0,$
    it is not quite clear how to characterize $\mu_1 \in \mathcal M_{\alpha_1}^+(\R),$
 satisfying $C_{\alpha_1}[\mu_1]=C_{\alpha_2}[\mu_2]$
for some $\mu_2 \in\mathcal M^+_{\alpha_2}(\R),$ and thus liftable to  $C_{\alpha_2}(\mathcal M^+_{\alpha_2}(\R)).$
(Recall that by Corollary \ref{Incl}, $\mu_1$ is liftable to $C_{\alpha_1+2}(\mathcal M^+_{\alpha_1+2}(\R)).$)
While we showed that such a lifting is not always possible and described in Theorems \ref{Inc1} and \ref{C1dos} the two particular
situations allowing for a complete answer, the question is, in general,  widely open.

\section{Acknowledgments}

We would like to thank the referee for careful reading
the manuscript and useful remarks and suggestions.

\end{document}